\documentclass[a4paper,12pt]{amsart}

\usepackage {amssymb}
\usepackage{amscd}

\theoremstyle{plain}
\newtheorem{thm}{Theorem}[section]
\newtheorem{lemma}[thm]{Lemma}

\newtheorem{prop}[thm]{Proposition}

\newtheorem{definition}[thm]{Definition}

\newtheorem{remark}[thm]{Remark}

\def\A{\operatorname{A}}
\def\B{\operatorname{B}}
\def\C{\operatorname{C}}
\def\D{\operatorname{D}}
\def\E{\operatorname{E}}
\def\F{\operatorname{F}}
\def\G{\operatorname{G}}

\def\I{\operatorname{I}}
\def\J{\operatorname{J}}

\def\M{\operatorname{M}}

\def\O{\operatorname{O}}

\def\U{\operatorname{U}}

\def\Z{\operatorname{Z}}
\def\PU{\operatorname{PU}}
\def\SO{\operatorname{SO}}
\def\Sp{\operatorname{Sp}}
\def\SU{\operatorname{SU}}
\def\PSU{\operatorname{PSU}}
\def\PSO{\operatorname{PSO}}

\def\Ad{\operatorname{Ad}}

\def\Aut{\operatorname{Aut}}

\def\defe{\operatorname{defe}}
\def\diag{\operatorname{diag}}

\def\GL{\operatorname{GL}}

\def\Graph{\operatorname{Graph}}

\def\Hom{\operatorname{Hom}}
\def\id{\operatorname{id}}

\def\Int{\operatorname{Int}}

\def\Lie{\operatorname{Lie}}

\def\Pin{\operatorname{Pin}}
\def\rank{\operatorname{rank}}
\def\Res{\operatorname{Res}}

\def\span{\operatorname{span}}
\def\Spin{\operatorname{Spin}}

\def\Sym{\operatorname{Sym}}

\def\tr{\operatorname{tr}}

\newcommand{\fre}{\mathfrak{e}}
\newcommand{\frf}{\mathfrak{f}}
\newcommand{\frg}{\mathfrak{g}}
\newcommand{\frh}{\mathfrak{h}}

\newcommand{\frk}{\mathfrak{k}}
\newcommand{\frl}{\mathfrak{l}}

\newcommand{\frp}{\mathfrak{p}}

\newcommand{\fru}{\mathfrak{u}}

\newcommand{\bbC}{\mathbb{C}}

\newcommand{\bbF}{\mathbb{F}}

\newcommand{\bbR}{\mathbb{R}}
\newcommand{\bbZ}{\mathbb{Z}}
\newcommand{\bbH}{\mathbb{H}}

\begin{document}

\title[Elementary abelian 2-subgroups]{Elementary abelian 2-subgroups of compact Lie groups}


\author{Jun Yu}
\address{Department of Mathematics, ETH Z\"urich, 8092 Z\"urich, Switzerland}

\email{jun.yu@math.ethz.ch}

\abstract{We classify elementary abelian 2-subgroups of compact simple Lie groups of adjoint 
type. This finishes the classification of elementary abelian $p$-subgroups of compact simple 
Lie groups (equivalently, complex linear algebraic simple groups) of adjoint type 
started in \cite{Griess} and continued in \cite{Andersen-Grodal-Moller-Viruel}.}
\endabstract




\maketitle

\section{Introduction}

An elementary abelian $p$-group of rank $n$ is a finite group isomorphic to $(C_{p})^{n}$. 
In this paper we study {\it elementary abelian $p$-subgroups} of compact simple Lie groups 
of adjoint type. A Lie group $G$ is simple means it Lie algebra $\frg_0=\Lie G$ is simple; 
it is of adjoint type means the adjoint homomorphism 
$\pi: G\longrightarrow\Aut(\frg_0)$ is injective. A subgroup of Lie group $G$ is 
called {\it toral} if it is contained in a maximal torus of $G$. 
The structure of non-toral elementary abelian $p$-subgroups of a compact group $G$ is 
related to the topology of $G$ and its classifying space  
(cf, \cite{Borel}, \cite{Serre}, \cite{Andersen-Grodal-Moller-Viruel}).
In 1950s, Borel made an observation that, for a compact connected Lie group $G$, 
the cohomology ring $H^{*}(G,\bbZ)$ has non-trivial $p$-torsion if and only if $G$ has 
a non-toral elementary abelian $p$-subgroup (cf, \cite{Borel} and \cite{Serre}).  
We call a prime $p$ a torsion prime of a compact (not necessary connected) Lie group $G$ if $G$ has 
a non-toral elementary abelian $p$-subgroup. For $G=\Aut(\fru_0)$ with $\fru_0$ a compact 
simple Lie algebra, the torsion primes are as follows.  

\begin{table}[ht]
\caption{Torsion primes}\centering
\begin{tabular}{|c |c |c |c | c| c| c| c| c| c| c|}
\hline $\A_{n-1},n\geq 2$ & $\B_{n},n\geq 2$ & $\C_{n},n\geq 3$ & 
$\D_{n},n\geq 5$ & $\D_{4}$ & $\E_6$ & $\E_7$  & $\E_8$ & $\F_4$ &  $\G_2$ 
\\ [0.5ex] \hline 

$p|2n$ & $2$ & $2$ & $2$ & $2,3$ & $2,3$ & $2,3$ & $2,3,5$ & $2,3$ & 
$2$\\ \hline

\end{tabular}
\end{table}

In \cite{Griess}, R. Griess classified maximal non-toral elementary abelian $p$-subgroups 
of any compact simple Lie group of adjoint type and gave criteria to distinguish which 
elementary abelian $p$-subgroups are non-toral (\cite{Griess} is about linear algebraic groups 
over an algebraic closed field of characteristic 0, which is equivalent to the consideration of 
compact Lie groups). For primes $p\geq 5$, Griess has classified non-toral elementary abelin 
$p$-subgroups, they are some $p$-subgroups of $\PU(n)$ ($p|n$) and a subgroup $(\bbZ/5\bbZ)^{3}$ 
of $\E_8$. For $p=3$, the non-toral subgroups of compact simple Lie groups of adjoint type up 
to conjugation was classified in \cite{Andersen-Grodal-Moller-Viruel}, they are $3$-subgroups 
of $\PU(n)$ ($3|n$) and some 3-subgroups of compact simple Lie groups of type 
$\E_6,\E_7,\E_8,\F_4$. This work was based on Griess' study of maximal elementary abelian 
$p$-subgroups. The authors used it as a step to their classification of $p$-compact groups 
(some objects in homotopy theory) for $p$ odd. In this paper, we classify elementary abelian 
2-subgroups of any compact simple Lie group of adjoint type. It suffices to classify elementary 
abelian 2-subgroups of the group $G=\Aut(\fru_0)$ for each compact simple Lie algebra $\fru_0$.


The paper is organized as follows. In Section 2, we do the classification for classical 
simple Lie algebras, which amounts to classify elementary abelian 2-subgroups of the groups 
$\PU(n)$ (Type A, inner case), $\O(n)/\langle -I\rangle$, $\Sp(n)/\langle-I\rangle$ and 
$\PU(n)\rtimes\langle\tau_0\rangle$ (Type A, full case, $\tau_0=\textrm{complex conjugation}$). 

For an elementary abelian 2-subgroup $F$ of $\PU(n)$, we define a map 
\[m: F\times F\longrightarrow\{\pm{1}\},\] show that it is a bilinear form when $F$ is viewed as 
a vector space over $\mathbb{F}_2=\bbZ/2\bbZ$ and $\{\pm{1}\}$ is identified with $\mathbb{F}_2$. 
We show that $\ker m$ is diagonalizable and the conjugacy class of $F$ is determined by 
the conjugacy class of $\ker m$ and the number $\rank(F/\ker m)$. 
For an elementary abelian 2-subgroup $F$ of $\O(n)/\langle -I\rangle$ or 
$\Sp(n)/\langle-I\rangle$, we define a bilinear map $m: F\times F\longrightarrow\{\pm{1}\}$ 
and a function \[\mu: F\longrightarrow\{\pm{1}\},\] which satisfy a compatibility relation 
($m(x,y)=\mu(x)\mu(y)\mu(xy)$). They give $F$ a structure we called {\it symplectic metric 
space} and we get invariants $r,s,\epsilon,\delta$ from this structure. We show that the conjugacy 
class of $F$ is determined by the conjugacy class of $A_{F}=\ker(\mu|_{\ker m})$ and the numbers 
$s,\epsilon,\delta$. The classification in the full type A case can be reduced to the above
three cases. 

In Subsection \ref{Subsection:a subclass}, we discuss a special class of subgroups of 
$\O(n)/\langle -I\rangle$ and $\Sp(n)/\langle-I\rangle$, introduce the notions of ``symplectic 
vector space'' and ``symplectic metric space'' and study their automorphism groups. The study 
of some elementary abelian 2-subgroups of $\Aut(\fru_0)$ for compact exceptional simple Lie 
algebras $\fru_0$ is reduced to this class of subgroups of $\Sp(n)/\langle-I\rangle$. Moreover, 
their Automizer groups (Definition \ref{Automizer}) are described in terms of the automorphism 
groups of symplectic vector spaces or symplectic metric spaces. 

In Section \ref{Section:exceptional}, we 
present our understanding of involutions in $\Aut(\fru_0)$ for compact exceptional simple Lie 
algebras (cf, \cite{Knapp}), the symmetric subgroups $\Aut(\fru_0)^{\theta}$ and Klein four 
subgroups of $\Aut(\fru_0)$ (cf, \cite{Huang-Yu}). We also discuss our method to the 
classification of elementary abelian 2-subgroups in $\Aut(\fru_0)$ for compact exceptional 
simple Lie algebras. 

In Sections 4-8, we give the classification in type $\G_2,\F_4,\E_6,\E_7,\E_8$ respectively. 
We show that an elementary abelian 2-subgroup of $\Aut(\frg_2)=\G_2$ is determined up to 
conjugacy by its rank, which is at most 3. For an elementary abelian 2-subgroup $F$ of 
$\Aut(\frf_4)=\F_4$, the subset $A_{F}$ of elements conjugate to $\sigma_2$  
(cf, Subsection \ref{Subsection:involution}) and the identity element is a subgroup. We show that 
the conjugacy class of $F$ is determined by the numbers $\rank A_{F}$ and $\rank F/A_{F}$, which 
satisfy $\rank A_{F}\leq 2$ and $\rank F/A_{F}\leq 3$. 
$\E_6$ has $\F_4$ and $\Sp(4)/\langle-I\rangle$ as its symmetric subgroups, we construct 
elementary abelian 2-subgroups of $\Aut(\fre_6)$ from subgroups of $\F_4$ and certain special 
subgroups (those considered in Subsection \ref{Subsection:a subclass}) of 
$\Sp(4)/\langle-I\rangle$, show they represent all conjugacy classes and are not conjugate to 
each other. 

Use an ascending chain $\E_6\subset\E_7\subset\E_8$, we construct some elementary abelian 
2-subgroups of $\Aut(\fre_7)$ and $\Aut(\fre_8)$ from subgroups of $\Aut(\fre_6)$. The 
remaining elementary abelian 2-subgroups of $\Aut(\fre_7)$ are related to very special subgroups 
of $\PSU(8)$ and $\PSO(8)$. Some of the remaining elementary abelian 2-subgroups $F$ of 
$\Aut(\fre_8)$ possess two canonical subgroups $A_{F},H_{F}$ with $\rank H_{F}/A_{F}=1$ and the 
conjugacy class of $F$ is determined by the numbers $\rank A_{F},\rank F/H_{F}$. Any other 
elementary abelian 2-subgroup of $\Aut(\fre_8)$ is determined up to conjugacy by its rank, which 
is at most 5.

\begin{thm}
For $\fru_0=\fre_6,\fre_7,\fre_8,\frf_4,\frg_2$, there are $51,78,66,12,4$ 
conjugacy classes of elementary abelian $2$-subgroups in $\Aut(\fru_0)$ 
respectively.
\end{thm}

For an elementary abelian 2-subgroup $F$ of a compact Lie group $G$, 
we call $W(F)=N_{G}(F)/C_{G}(F)$ the {\it Automizer group} of $F$ (in $G$), the name of 
Automizer is suggested by Professor R. Griess. We determine the Automizer group $W(F)$ 
for each elementary abelian 2-subgroup $F$ of $\Aut(\fru_0)$ with $\fru_0$ a compact 
exceptional simple Lie algebra. 


{\it Notations.} 
Let $\Z_{n}=\{\lambda \I_{n}: |\lambda|=1\}$, $\I_{p,q}=\diag\{-\I_{p},\I_{q}\}$, 
\[\J_{n}=\left(\begin{array}{cc}0&\I_{n}\\-\I_{n}&0\\\end{array}\right),   
\J'_{n}=\left(\begin{array}{cc}0&\I_{n}\\\I_{n}&0\\\end{array}\right),
K_{n}=\left(\begin{array}{cccc}0&\I_{n}&0&0\\-\I_{n}&0&0&0\\0&0&0&-\I_{n}\\
0&0&\I_{n}&0\\\end{array}\right).\]

Let $\fre_6$ denote the compact simple Lie algebra of type $\bf E_6$, $\E_6$  
denote the connected and simply connected Lie group with Lie algebra $\fre_6$, and  
$\fre_6(\bbC), \E_6(\bbC)$ their complexifications. Similar
notations will be used for other types.

Let $Z(G)$ denote the center of a group $G$ and $G_0$ denote the connected component of $G$ 
containing identity element. For $H\subset G$ ($\frh\subset\frg$),
let $C_{G}(H)$ ($C_{\frg}(\frh)$) denote the centralizer of $H$ in $G$ ($\frh$ in $\frg$), 
let $N_{G}(H)$ denote the normalizer of $H$ in $G$.

In the case $G=\E_6,\E_7$, let $c$ denotes a non-trivial element in $Z(G)$.

Let $\Pin(n)$ ($\Spin(n)$) be the Pin (Spin) group in degree $n$.
Write $c=e_1e_2\cdots e_{n}\in \Pin(n)$. If $n$ is odd, then $\Spin(n)$ has a Spinor module $M$ of
dimension $2^{\frac{n-1}{2}}$.  If $n$ is even, then $\Spin(n)$ has two Spinor
modules $M_{+},M_{-}$  of dimension $2^{\frac{n-2}{2}}$. We distinguish
$M_{+},M_{-}$ by requiring that  $c$ acts on $M_{+}$  and $M_{-}$ by scalar 1 or -1
respectively when $4|n$; and by $-i$ or $i$ respectively when $4|n-2$.

For a quotient group $G=H/N$, let $[x]=xN$ ($x\in H$) denote a coset.  

\noindent{\it Acknowledgement.} The author would like to thank Professors Griess and Grodal for 
some helpful communications. He would like to thank Professor Jong-Song Huang for discussions 
on symmetric pairs and thank Professor Doran for a lot of suggestions on the mathematical 
writing. The author 's research is supported by a grant from Swiss National Science 
Foundation (Schweizerischer Nationalfonds).

\section{Matrix groups}\label{Section:M}

For $F=\bbR,\bbC,\bbH$, let $\M_{n}(F)$ be the set of $n\times n$ matrices with 
entries in $F$. Let \begin{eqnarray*} && \O(n)=\{X\in \M_{n}(\bbR)|XX^{t}=\I\}, 
\SO(n)=\{X\in \O(n)|\det X=1\},\\&& \U(n)=\{X\in \M_{n}(\bbC)|XX^{\ast}=\I\}, 
\SU(n)=\{X\in \U(n)|\det X=1\}, \\&& \Sp(n)=\{X\in \M_{n}(\bbH)|XX^{\ast}=\I\},
\end{eqnarray*} then $\O(n)\subset\U(n)\subset\Sp(n) $. Let \begin{eqnarray*}
&&\mathfrak{so}(n)=\{X\in M_{n}(\bbR)|X+X^{t}=0\},\\&&
\mathfrak{su}(n)=\{X\in M_{n}(\bbC)|X+X^{\ast}=0, \tr X=0\},\\
&&\mathfrak{sp}(n)=\{X\in M_{n}(\bbH)|X+X^{\ast}=0\},\end{eqnarray*} then 
$\mathfrak{so}(n), \mathfrak{su}(n),\mathfrak{sp}(n)$ are Lie algebras of 
$\SO(n),\SU(n),\Sp(n)$ respectively.

\subsection{Projective unitary groups}\label{Subsection:A}

For $G=\U(n)/\Z_{n}$, any involution in $G$ has a representative $A\in\U(n)$ 
with $A^{2}=I$. Then $A\sim I_{p,n-p}= \left(\begin{array}{cc}-I_{p}&0\\ 0& I_{n-p}\\ 
\end{array} \right)$ for some $p$, $1\leq p \leq n-1$. One has 
\begin{eqnarray*}&&(\U(n)/\Z_{n})^{[I_{p,n-p}]}=(\U(p)\times\U(n-p))/\Z_{n}\textrm{ when } 
p\neq\frac{n}{2};\\&&(\U(n)/\Z_{n})^{[I_{\frac{n}{2},\frac{n}{2}}]}=
((\U(n/2)\times\U(n/2))/\Z_{n})\rtimes\langle[J'_{n}]\rangle.\end{eqnarray*}

Let $F\subset\U(n)/\Z_{n}$ be an elementary abelian 2-subgroup. For any $x,y\in F$, 
choose $A,B\in\U(n)$ with $A^{2}=B^{2}=I$ representing $x,y$. Then 
$1=[x,y]=[A,B]\Z_{n}/\Z_{n}$ implies $[A,B]=\lambda_{A,B} I$ for some 
$\lambda_{A,B}\in\bbC$, which is obviously independent with the choice of $A,B$.  
Also, $x^2=y^{2}=1$ implies $\lambda_{A,B}=\pm{1}$.
For any $x,y\in F$, let $m(x,y)=m_{F}(x,y)=\lambda_{A,B}$.

\begin{lemma}\label{A bilinear}
For any $x, y, z \in F$, $m(x,x)=1$ and $m(xy,z)=m(x,z)m(y,z)$.
\end{lemma}

\begin{proof}
$m(x,x)=1$ is clear. Choose $A,B,C\in\U(n)$ with $A^{2}=B^{2}=C^{2}=I$ representing $x,y,z$. 
Let $[A,C]=\lambda_1 I$, $[B, C]=\lambda_2 I$ for some numbers $\lambda_1,\lambda_2=\pm{1}$ 
. We have
\[[AB,C]=A[B,C]A^{-1}[A,C]=A(\lambda_2 I)A^{-1}(\lambda_1 I)=(\lambda_1 \lambda_2)I.\] 
So $m(xy,z)=m(x,z)m(y,z)$.
\end{proof}

If we regard $F$ as a vector space on $\mathbb{F}_2=\bbZ/2\bbZ$ and identify $\{\pm{1}\}$ 
with $\mathbb{F}_2=\bbZ/2\bbZ$, Lemma \ref{A bilinear} just said $m$ is an {\it 
anti-symmetric bilinear $2$-form} on $F$. Let 
\[\ker m=\{x\in F| m(x,y)=1,\forall y \in F\},\] which is a subgroup of $F$.

For an even $n$, let $\Gamma_0=\langle[I_{\frac{n}{2},\frac{n}{2}}],[J'_{\frac{n}{2}}]\rangle$,  
then any non-identity element of $\Gamma_0$ is conjugate to $I_{\frac{n}{2},\frac{n}{2}}$ and 
$(\U(n)/\Z_{n})^{\Gamma_0}\cong(\U(n/2)/\Z_{\frac{n}{2}})\times\Gamma_0$.

\begin{lemma}\label{Type A: Klein four}
For a Klein four group $F\subset G=\U(n)/\Z_{n}$, if $m_{F}$ is non-trivial, then $F$ 
is conjugate to $\Gamma_0$.
\end{lemma}

\begin{proof}
Choose $A,B\in\U(n)$ with $A^{2}=B^{2}=I$ and $F=\langle[A],[B]\rangle$. Since $m_{F}$ 
is non-trivial, we have $[A,B]=-I$. Since $A^{2}=I$, we may assume that $A=I_{p,n-p}$ 
for some $1\leq p\leq\frac{n}{2}$. From $[A,B]=-I$, we get $ABA^{-1}=-B$. Then $B$ is 
of the form $B=\left(\begin{array}{cc}&B_1\\B_2^{t}&\end{array}\right)$ 
for some $B_1,B_2\in M_{p,n-p}$. Since $B$ is invertible, we get $p=\frac{n}{2}$. 
Since $B^{2}=I$, we have $B_1B_2^{t}=I$. Let $S=\diag\{I_{n/2},B_1\}$. Then 
$(SAS^{-1},SBS^{-1})=(\left( \begin{array}{cc}-I_{\frac{n}{2}}&0\\
0&I_{\frac{n}{2}}\\\end{array}\right), \left( \begin{array}{cc} 0&I_{\frac{n}{2}}
\\I_{\frac{n}{2}}&0\\\end{array} \right))$.                                                   
\end{proof}

\begin{lemma}\label{Diagonal map}
For any $m,k\geq 1$ , let $\Delta(A)=\diag\{A,A,\dots,A\}$ ($A\in\U(m)$) be the diagonal 
map $\U(m)\longrightarrow\U(km)$. Then for any two subgroups $S_1,S_2\subset\U(m)$, 
\[\Delta(S_1)\sim\Delta(S_2)\Leftrightarrow S_1\sim S_2.\]
\end{lemma}

\begin{proof}
Both similarity conditions are equivalent to,  
\[\exists\phi: S_1\rightarrow S_2 \textrm{ with }\tr(\phi(x))=\tr(x), 
\forall x\in S_1. \] 
\end{proof}

\begin{prop}\label{Type A: classification}
Let $F$ be an elementary abelian 2-subgroup of $\U(n)/\Z_{n}$, 
\begin{itemize}
\item[(1)]{when $\ker m=1$, the conjugacy class of $F$ is determined by $\rank F$;}
\item[(2)]{in general, $\ker m$ is diagonalizable and the conjugacy class of $F$ is 
determined by the conjugacy class of $\ker m$ and the number $\rank F$.}
\end{itemize}
\end{prop}

\begin{proof}
For $(1)$, prove by induction on $n$. Since $\ker m=1$, $\rank F$ is of even. When 
$\rank F\geq 2$, choose any $x_1,x_2\in F$ with $m(x_1,x_2)=-1$. By Lemma 
\ref{Type A: Klein four}, $\langle x_1,x_2\rangle\sim\Gamma_0$. We may assume that 
$\langle x_1,x_2\rangle=\Gamma_0$, then $F\subset(\U(n)/\Z_{n})^{\Gamma_0}=
\Delta(\U(n/2)/\Z_{\frac{n}{2}})\times\Gamma_0$. And so $F=\Delta(F')\times\Gamma_0$ 
for some $F'\subset\U(n/2)/\Z_{\frac{n}{2}}$. We also have $\ker m_{F'}=1$. By induction, 
the conjugacy class of $F'$ is determined by $\rank F'$, so the conjugcay class of $F$ 
is determined by $\rank F$.  

For $(2)$, since $\pi^{-1}(\ker m)$ is abelian by the definition of $m$ and $\ker m$, where 
$\pi$ is the natural projection from $\U(n)$ to $\U(n)/\Z_{n}$. So it is diagonalizable. Then 
$\ker m$ is diagonalizable. 
We may write $F$ as $F=\ker m \times F'$ with $m(\ker m,F')=1$ and $m|_{F'}$ 
non-degenerate. By $(1)$, the conjugcay class of $F'$ is determined by 
$\rank F'=\rank F-\rank\ker m$. Moreover, it is clear that 
$(\U(n)/\Z_{n})^{F'}=\Delta(\U(n')/\Z_{n'})\times F'$, where $n'=n/2^{\frac{\rank F'}{2}}$. 
So $\ker m=\Delta(F'')$ for some $F''\subset\U(n')/\Z_{n'}$. Fix $F'$, by Lemma 
\ref{Diagonal map}, the conjugacy class of $\ker m$ in $\U(n)/\Z_{n}$ and that of 
$F''$ in $\U(n')/\Z_{n'}$ determine each other. Since the conjugacy of $F$ is 
determined by $F'$ and the class of $\ker m$ in $(\U(n)/\Z_{n})^{F'}$, we get the 
last statement of $(2)$.
\end{proof}


\subsection{Orthogonal and symplectic groups}\label{Subsection:BDC}

Let $F$ be an elementary abelian 2-subgroup of $G=\O(n)/\langle-I\rangle,n\geq 2$. 
For any $x\in F$, choose $A\in\O(n)$ representing $x$, then $A^{2}=\lambda_{A} I$ 
for some $\lambda_{A}=\pm{1}$. For any $x,y\in F$, choose $A,B\in\O(n)$ representing 
$x,y$, then $[A,B]=\lambda_{A,B} I$ for some $\lambda_{A,B}=\pm{1}$. The values of 
$\lambda_{A},\lambda_{A,B}$ are independent with the choice of $A,B$.
For any $x\in F$, let $\mu(x)=\lambda_{A}$; for any $x,y\in F$, let 
$m(x,y)=\lambda_{A,B}$.

\begin{lemma}\label{D bilinear}
For any $x, y, z \in F$, $m(x,x)=\mu(x)=1$, $m(xy,z)=m(x,z)m(y,z)$ and 
\[m(x,y)=\mu(x)\mu(y)\mu(xy).\] 
\end{lemma}

\begin{proof}
The proof is similar as that for \ref{A bilinear}. 
\end{proof}

\begin{lemma}\label{L:outer A} 
$\mathfrak{su}(n)$ has two (or one) conjugacy classes of outer involutive automorphisms 
with representative $\tau_0=\textrm{complex conjugation}$ and $\tau'_0=\tau_0\Ad(J_{n/2})$. 
\end{lemma}
\begin{proof}
This follows from Cartan 's classification of compact Riemannian symmetric pairs.   
\end{proof}

\begin{lemma}\label{Type BD: Klein four}
For $x\in F$, $\mu(x)=-1$ if and only if $x\sim [J_{\frac{n}{2}}]$.

For $x,y\in F$ with $m(x,y)=-1$, 
\begin{itemize}
\item[(1)]{when $\mu(x)=\mu(y)=-1$, $(x,y)\sim([J_{\frac{n}{2}}],[K_{\frac{n}{4}}])$;} 
\item[(2)]{when $\mu(x)=\mu(y)=1$, $(x,y)\sim ([I_{\frac{n}{2},\frac{n}{2}}],
[J'_{\frac{n}{2}}])$.} 
\end{itemize}
\end{lemma}

\begin{proof}
If $\mu(x)=-1$, then $x=[A]$ for some $A\in\O(n)$ with $A^{2}=-I$. The 
$A\sim J_{\frac{n}{2}}$  and so $x\sim [J_{\frac{n}{2}}]$. The proof for $(2)$ is 
the same as that for Lemma \ref{Type A: Klein four}. For $(1)$, we may assume that 
$x=[J_{\frac{n}{2}}]$, then $\mathfrak{so}(n)^{x}=\mathfrak{u}(n/2)$. By Lemma 
\ref{L:outer A}, after replace $y$ by some $gyg^{-1}$ with $g\in G^{x}$, we may assume 
that $(\mathfrak{so}(n)^{x})^{y}=(\mathfrak{u}(n/2))^{y}=$ $\mathfrak{so}(n/2)$ or 
$\mathfrak{sp}(n/4)$, then a little more argument shows $y=[K_{\frac{n}{4}}]$. 
\end{proof}

\begin{definition}
For an elementary abelian 2-group $F\subset O(n)/\langle-I\rangle$, $n\geq 2$, 
define $A_{F}=\ker(\mu|_{\ker m})$ and the defect index 
\[\defe F=|\{x\in F: \mu(x)=1\}|-|\{x\in F:\mu(x)=-1\}|.\] 
Define $(\epsilon_{F},\delta_{F})$ as follows, 
\begin{itemize} 
\item when $\mu|_{\ker m}\neq 1$, define $(\epsilon_{F},\delta_{F})=(1,0)$;
\item when $\mu|_{\ker m}=1$ and $\defe F<0$, define $(\epsilon_{F},\delta_{F})=(0,1)$;
\item when $\mu|_{\ker m}=1$ and $\defe F>0$, define $(\epsilon_{F},\delta_{F})=(0,0)$. 
\end{itemize}
  
Define $r_{F}=\rank A_{F}$, $s_{F}=\frac{1}{2}\rank (F/\ker m)-\delta_{F}$.
\end{definition}

We will see in the proof of Proposition \ref{Type BD: classification} that $\defe F=0$ 
if and only if $\mu|_{\ker m}\neq 1$. It is clear that $\epsilon_{F},\delta_{F},r_{F},s_{F}$ and 
the conjugcay class of $A_{F}$ are determined by the conjugacy class of $F$.

Let $\Gamma_1=\langle [I_{\frac{n}{2},\frac{n}{2}}],[J'_{\frac{n}{2}}]\rangle$ and 
$\Gamma_2=\langle[J_{\frac{n}{2}}],[K_{\frac{n}{4}}\rangle]$. Then $\defe\Gamma_1=2$, 
$\defe\Gamma_2=-2$, \[(\O(n)/\langle-I\rangle)^{\Gamma_1}=
\Delta(\O(\frac{n}{2})/\langle-I\rangle)\times\Gamma_1\] and 
\[(\O(n)/\langle-I\rangle)^{\Gamma_2}=\Delta(\Sp(\frac{n}{4})/\langle-I\rangle)\times\Gamma_2.\]

\begin{lemma}\label{Sub F1}
Let $F$ be a non-trivial elementary abelian 2-subgroup of $\O(n)/\langle-I\rangle$, if 
$\rank(F/\ker m)>2$, there exists a Klein four subgroup $F'\subset F$ with 
$F'\sim \Gamma_1$. 
\end{lemma}

\begin{proof}
Choose a subgroup $F''\subset F$ such that $F=\ker m \times F''$, then $\ker(m_{F''})=1$ 
and $\rank F''>2$. Replace $F$ by $F''$, we may assume that $\ker m=1$ and $\rank F>2$. 

We first show that, there exists $1\neq x\in F$ with $\mu(x)=1$. From $\rank F>2$, we get  
$\rank F\geq 4$ since it is even ($m_{F}$ is non-degenerate). Suppose any $1\neq x\in F$ has 
$\mu(x)=-1$, then for any distinct non-trivial elements $x,y\in F$, $m(x,y)=\mu(x)\mu(y)\mu(xy)=-1$ 
by Lemma \ref{D bilinear}, which contradicts to $m$ is bi-linear on $F$.

Upon we get $1\neq x\in F$ with $\mu(x)=1$, choose any $z\in F$ with $m(x,z)=-1$. Then 
$\mu(xz)\mu(z)=m(x,z)\mu(x)=-1$. So exactly one of $\mu(z),\mu(xz)$ is equal to -1. By Lemma 
\ref{Type BD: Klein four}, we have $\langle x,z\rangle\sim\Gamma_1$.
\end{proof}

\begin{lemma}\label{Sub F1-2}
Let $F$ be a non-trivial elementary abelian 2-subgroup of $\O(n)/\langle-I\rangle$, if 
$\rank(\ker m/A_{F})=1$ and $\rank(F/A_{F})>1$, there exists a Klein four subgroup 
$F'\subset F$ with $F'\sim\Gamma_1$. 
\end{lemma}

\begin{proof}
Choose a subgroup $F''\subset F$ such that $F=A_{F} \times F''$, then $\rank\ker(m_{F''})=1$, 
$A_{F''}=1$ and $\rank F''>1$. Replace $F$ by $F''$, we may assume that $A_{F}=1$, 
$\rank\ker m=1$ and $\rank F>1$. 
 
$F$ is of the form $F=\ker m\times F''$ with $m(\ker m, F'')=1$, $\rank F''\geq 2$, and 
$m_{F''}$ non-degenerate. When $\rank F''>2$ or $F''\sim\Gamma_1$, there exists 
$F'\subset F'$ with $F'\sim\Gamma_1$ by Lemma \ref{Sub F1}. Otherwise $F''\sim\Gamma_2$. 
Choose $x,y\in F''$ generating $F''$ and $1\neq z\in \ker m$, then 
$F'=\langle xz,yz\rangle\sim\Gamma_1$ since $(\mu(xz),\mu(yz),\mu(xy))=(1,1,-1)$.
\end{proof}

\begin{lemma}\label{L:O-Sp-U}
For any $n\geq 1$ , let $T:\O(n)\hookrightarrow\U(n)$, $T':\Sp(n/2)\hookrightarrow\U(n)$) 
be the natural inclusions. Then for any two subgroups $S_1,S_2\subset\O(n)$ or  
$S'_1,S'_2\subset\Sp(n/2)$
\[T(S_1)\sim T(S_2)\Leftrightarrow S_1\sim S_2,\] 
\[T'(S'_1)\sim T'(S'_2)\Leftrightarrow S'_1\sim S'_2.\]
\end{lemma}
\begin{proof}
This is a classical fact, one may refer \cite{Yu} for a proof.    
\end{proof}

\begin{prop}\label{Type BD: classification}
Let $F$ be an elementary abelian 2-subgroup of $\O(n)/\langle-I\rangle$,
\begin{itemize}
\item[(1)]{when $\ker m=1$, the conjugacy class of $F$ is determined by 
$\delta_{F}$ and $s_{F}$;}
\item[(2)]{in general, $\ker m$ is diagonalizable and the conjugacy class 
of $F$ is determined by the conjugacy class of $A_{F}$ and the invariants 
$(\epsilon_{F},\delta_{F},s_{F})$.}
\item[(3)]{we have $\defe(F)=(1-\epsilon_{F})(-1)^{\delta_{F}}2^{r_{F}+s_{F}+\delta_{F}}$.}
\end{itemize}
\end{prop}

\begin{proof}
For $(1)$, since $\ker m=1$, $\rank F$ is even. When $\rank F=2$, 
$F\sim\Gamma_1\textrm{ or }\Gamma_2$ by Lemma \ref{Type A: Klein four}. When 
$\rank F\geq 2$, there exists a Klein four subgroup $F'\subset F$ with 
$F'\sim\Gamma_1$ by Lemma \ref{Sub F1}. We may assume that, $\Gamma_1\subset F$, 
then $F\subset(\O(n)/\langle-I\rangle)^{\Gamma_1}=
\Delta(\O(\frac{n}{2})/\langle-I\rangle)\times\Gamma_1$. So $F=\Delta(F')\times\Gamma_1$ 
for some $F'\subset\O(\frac{n}{2})/\langle-I\rangle$. By induction, we 
can show $\defe F\neq 0$ and the conjugacy class of $F$ is determined 
by $\delta_{F}$ and $s_{F}$. 

For $(2)$, $\ker m$ is diagonalizable since $\pi^{-1}(\ker m)$ is abelian by the definition
of $m$, where $\pi$ is the natural projection $\pi:\O(n)\longrightarrow\O(n)/
\langle-I\rangle$. $F$ is of the form $F=\ker m\times F'$ with $m|_{F'}$ 
non-degenerate.  When $\epsilon_{F}=1$, by Lemma \ref{Sub F1-2}, $F$ is of the form 
$F=\ker m\times F'$ with $m_{F'}$ non-degenerate and $\defe F'>0$. 
By $(1)$, the class of $F'$ is determined by $s_{F}=\frac{\rank F'}{2}$. 
We have $(\O(n)/\langle-I\rangle)^{F'}=\Delta(\O(n')/\langle-I\rangle)\times F'$, where 
$n'=\frac{n}{2^{\frac{\rank F'}{2}}}$. Fix $F'$, by Lemmas \ref{Diagonal map} and 
\ref{L:O-Sp-U}, the class of $\ker m$ in $\O(n)/\langle-I\rangle$ determines the class of 
it in $(\O(n)/\langle-I\rangle)^{F'}$. Moreover, as $\epsilon_{F}=1$ is given, the class of 
$\ker m$ is determined by the class of $A_{F}=\ker\mu|_{\ker m}$. So the conjugacy class of 
$F$ is determined by that of $A_{F}$ and the invariants $(\delta_{F},s_{F})$ (and 
$\epsilon_{F}=1$ in this case). The proof for $\epsilon_{F}=0$ case is similar as that 
for $\epsilon_{F}=1$ case. 

(3) follows from Lemma \ref{Type BD: Klein four} and (2). 
\end{proof}

The classification of elementary abelian 2-subgroup of 
$G=\Sp(n)/\langle-I\rangle$ is similar as that for $\O(n)/\langle-I\rangle$. 
We give the definitions and results below but omit the proofs. 

Let $F$ be an elementary abelian 2 subgroup of 
$G=\Sp(n)/\langle-I\rangle,n\geq 2$. For any $x\in F$, choose 
$A\in\Sp(n)$ representing $x$, then $A^{2}=\lambda_{A} I$ for some 
$\lambda_{A}=\pm{1}$. For any $x,y\in F$, choose $A,B\in\Sp(n)$ 
representing $x,y$, then $[A,B]=\lambda_{A,B} I$ for some 
$\lambda_{A,B}=\pm{1}$. The values of $\lambda_{A},\lambda_{A,B}$ 
are obviously independent with the choice of $A,B$. For any $x\in F$, 
let $\mu(x)=\lambda_{A}$; for any $x,y\in F$, let $m(x,y)=\lambda_{A,B}$.

\begin{lemma}\label{C bilinear}
For any $x,y,z\in F$, $m(x,x)=\mu(x)=1$, $m(xy,z)=m(x,z)m(y,z)$, and   
\[m(x,y)=\mu(x)\mu(y)\mu(xy).\] 
\end{lemma}

\begin{lemma}\label{Type C: Klein four}
For $x\in F$, $\mu(x)=-1$ if and only if 
$x\sim [J_{\frac{n}{2}}]$.

For $x,y\in F$ with $m(x,y)=-1$, 
\begin{itemize}
\item[(1)]{when $\mu(x)=\mu(y)=-1$, $(x,y)\sim
([\textbf{i}I],[\textbf{j}I])$;} 
\item[(2)]{when $\mu(x)=\mu(y)=1$, $(x,y)\sim
([I_{\frac{n}{2},\frac{n}{2}}],
[J'_{\frac{n}{2}}])$.} 
\end{itemize}
\end{lemma}


\begin{definition}
For an elementary abelian 2-group $F\subset G=\Sp(n)/\langle-I\rangle$, 
define $A_{F}=\ker(\mu|_{\ker m})$ and the defect index 
\[\defe F=|\{x\in F: \mu(x)=1\}|-|\{x\in F:\mu(x)=-1\}|.\]  
Define $(\epsilon_{F},\delta_{F})$ as follows,  
\begin{itemize} 
\item when $\mu|_{\ker m}\neq 1$, define $(\epsilon_{F},\delta_{F})=(1,0)$;
\item when $\mu|_{\ker m}=1$ and $\defe F<0$, define $(\epsilon_{F},\delta_{F})=(0,1)$;
\item when $\mu|_{\ker m}=1$ and $\defe F>0$, define $(\epsilon_{F},\delta_{F})=(0,0)$.
\end{itemize} 

Define $r_{F}=\rank A_{F}$, $s_{F}=\frac{1}{2}\rank(F/\ker m)-\delta_{F}$.
\end{definition}

\begin{prop}\label{Type C: classification}
Let $F$ be an elementary abelian 2 subgroup of $\Sp(n)/\langle-I\rangle$,
\begin{itemize}
\item[(1)]{when $\ker m=1$, the conjugacy class of $F$ is determined by 
$\delta_{F}$ and $s_{F}$;}
\item[(2)]{in general, $\ker m$ is diagonalizable and the 
conjugacy class of $F$ is determined by the conjugacy class of $A_{F}$ and 
the invariants $(\epsilon_{F},\delta_{F},s_{F})$.}
\item[(3)]{we have $\defe(F)=(1-\epsilon_{F})(-1)^{\delta_{F}}2^{r_{F}+s_{F}+\delta_{F}}$.}
\end{itemize}
\end{prop}

\subsection{Twisted projective unitary groups}\label{Subsection:twisted A}

For $n\geq 3$, $G=\Aut(\mathfrak{su}(n))$ has two connected components and 
$G_0=\Int(\mathfrak{su}(n))=\U(n)/\Z_{n}$. There are two conjugacy classes of 
outer involutions in $G$ with representatives $\tau_0$, $\tau_0\Ad(J_{n/2})$  
(the latter exists only when $n$ is even), where $\tau_0$ is 
the complex conjugation. One has \begin{eqnarray*}
&&\Int(\mathfrak{su}(n))^{\tau_0}=\O(n)/\langle-I\rangle,\\&&
\Int(\mathfrak{su}(n))^{\tau_0\Ad(J_{n/2})}=\Sp(n/2)/\langle-I\rangle.
\end{eqnarray*}

Let $F$ be an elementary abelian 2-subgroup of $\Aut(\mathfrak{su}(n))$.
From the inclusion $F\cap\Int(\mathfrak{su}(n))\subset\Int(\mathfrak{su}(n))=
\U(n)/\Z_{n}$, we have a bilinear form \[m: F\cap\Int(\mathfrak{su}(n))\times 
F\cap\Int(\mathfrak{su}(n))\longrightarrow\{\pm{1}\}.\] Moreover, we define a 
function \[\mu: F-F\cap\Int(\mathfrak{su}(n))\longrightarrow\{\pm{1}\}\] by 
$\mu(z)=1$ if $z\sim\tau_0$, and $\mu(z)=-1$ if $z\sim\tau_0\Ad(J_{n/2})$. 
On the other hand, for any $z\in F-\Int(\mathfrak{su})(n)$, 
define $\mu_{z}: F\cap\Int(\mathfrak{su}(n))\longrightarrow \{\pm{1}\}$ and 
\[m_{z}:(F\cap\Int(\mathfrak{su}(n)))\times (F\cap\Int(\mathfrak{su}(n)))
\longrightarrow\{\pm{1}\}\] from the inclusion 
\[F\cap\Int(\mathfrak{su}(n))\subset \Int(\mathfrak{su}(n))^{z}\cong 
\O(n)/\langle-I\rangle\textrm{ or }\Sp(n/2)/\langle-I\rangle.\]

\begin{definition}
For an elementary abelian 2-group $F\subset\Aut(\mathfrak{su}(n))$,
define \[A_{F}=\{x\in F\cap\Int(\mathfrak{su}(n))|z\sim zx, \forall z\in F-
F\cap\Int(\mathfrak{su}(n))\}\] and 
\[\defe F=|\{x\in F: x\sim\tau_0\}|-|\{x\in F: x\sim\tau_0\Ad(J_{n/2})\}|.\] 

Define $(\epsilon_{F},\delta_{F})$ as follows, \begin{itemize}
\item when $\defe F=0$, define $(\epsilon_{F},\delta_{F})=(1,0)$;
\item when $\defe F>0$, define $(\epsilon_{F},\delta_{F})=(0,0)$;
\item when $\defe F<0$, define $(\epsilon_{F},\delta_{F})=(0,1)$.
\end{itemize}

Define $r_{F}=\rank A_{F}$ and $s_{F}=\frac{1}{2}\rank(F/\ker m)-\delta_{F}$. 
\end{definition}

\begin{lemma}\label{L:TA-m coincidence}
For any $z\in F-F\cap\Int(\mathfrak{su}(n))$, $m_{z}=m$ on $F\cap\Int(\mathfrak{su}(n))$. 
\end{lemma}

\begin{proof}
For $z\in F-\Int(\mathfrak{su}(n))$ and $x,y\in F\cap\Int(\mathfrak{su}(n))$, 
\[m(x,y)=-1\Leftrightarrow\langle x,y\rangle\sim
\langle[I_{\frac{n}{2},\frac{n}{2}}],[J'_{\frac{n}{2}}]\rangle
\Leftrightarrow m_{z}(x,y)=-1.\] So $m_{z}(x,y)=m(x,y)$.
\end{proof}

\begin{lemma}\label{L:TA-mu compare}
For any $z\in F-F\cap\Int(\mathfrak{su}(n))$ and $x\in F\cap\Int(\mathfrak{su}(n))$, 
$\mu_{z}(x)=\mu(z)\mu(zx)$. 
\end{lemma}

\begin{proof}
We may assume that $z=\tau_0$ or $\tau_0\Ad(J_{n/2})$. 

In the case of $z=\tau_0$ and $\mu_{z}(x)=1$, we may assume that 
$x=[I_{p,n-p}]\in\O(n)/\langle-I\rangle=\Int(\mathfrak{su}(n))^{z}$ for some $0\leq p\leq n$. 
Let $u=[\diag\{iI_{p},I_{n-p}\}]$, then 
\begin{eqnarray*}&&uzu^{-1}=z(z^{-1}uz)u^{-1}=z(\overline{u})u^{-1}\\&=&z[\diag\{-iI_{p},I_{n-p}\}]
[\diag\{-iI_{p},I_{n-p}\}]\\&=&z[I_{p,n-p}]=zx.\end{eqnarray*} 
So $zx\sim z$ and $1=\mu_{z}(x)=\mu(z)\mu(zx)$.

In the case of $z=\tau_0$ and $\mu_{z}(x)=-1$, we may assume that 
$x=[J_{n/2}]\in\O(n)/\langle-I\rangle=\Int(\mathfrak{su}(n))^{z}$. Then 
$zx=\tau_0\Ad(J_{n/2})=\tau'_0$. So $-1=\mu_{z}(x)=\mu(z)\mu(zx)$. 

The proof in the case of $z=\tau'_0=\tau_0\Ad(J_{n/2})$ is similar. 
\end{proof}

\begin{lemma}\label{L:TA-translation} 
We have $A_{F}\subset\ker m$ and $A_{F}=\ker(\mu_{z}|_{\ker m})$ for any 
$z\in F-F\cap\Int(\mathfrak{su}(n))$.
\end{lemma}

\begin{proof}
Choose an element $z\in F-F\cap\Int(\mathfrak{su}(n))$. For any $z\in A_{F}$, by the 
definition, for any $y\in F\cap\Int(\mathfrak{su}(n))$, we have $\mu(zy)=\mu(zyx)$, in 
particular $\mu(z)=\mu(zx)$. Then 
\[m(x,y)=m_{z}(x,y)=\mu_{z}(x)\mu_{z}(y)\mu_{z}(xy)=\mu(z)\mu(zx)\mu(zy)\mu(zxy)=1.\] 
So $A_{F}\subset\ker m$. 

On the other hand, for any $x\in\ker m=\ker m_{z}$, $x\in A_{F}$ if and only if 
$\forall y\in F\cap\Int(\mathfrak{su}(n)), \mu(zy)=\mu(zyx)$.  
Since \[\mu(zy)\mu(zyx)=\mu_{z}(y)\mu_{z}(xy)=m_{z}(x,y)\mu_{z}(x)=\mu_{z}(x)\]  
(we use $x\in\ker m=\ker m_{z}$ in the last equality), we get $A_{F}=\ker(\mu_{z}|_{\ker m})$.    
\end{proof}

\begin{prop}\label{P:TA-classification}
For an elementary abelian 2 group $F\subset\Aut(\mathfrak{su}(n))$ but 
not contained in $\Int(\mathfrak{su}(n))$, $\ker m$ is diagonalizable 
and the conjugacy class of $F$ is determined by the conjugacy class of $A_{F}$ and 
the invariants $(\epsilon_{F},\delta_{F},\rank F)$.
\end{prop}

\begin{proof}
When there exists $z\in F$ with $z\sim\tau_0$, 
\[F\subset\Aut(\mathfrak{su}(n))^{z}=(\O(n)/\langle-I\rangle)\times
\langle z\rangle.\] 
By Lemma \ref{L:TA-m coincidence}, $m_{z}=m$. Then $(\epsilon_{F},\delta_{F})$ 
coincides with $(\epsilon_{F'},\epsilon_{F'})$ when $F'=F\cap\Int(\mathfrak{su}(n))$ 
is considered as a subgroup of $\O(n)/\langle-I\rangle$. Then the conclusion 
follows from Proposition \ref{Type BD: classification} and Lemma \ref{L:O-Sp-U}. 

Otherwise, for any $z\in F-F\cap \Int(\mathfrak{su}(n))$, $z\sim\tau_0\Ad(J_{n/2})$, so 
\[F\subset\Aut(\mathfrak{su}(n))^{z}=(\Sp(n/2)/\langle-I\rangle)\times
\langle z\rangle\] and $\mu_{z}\equiv 1$. In this case, 
$A_{F}=\ker m=F\cap\Int(\mathfrak{su}(n))$, $(\epsilon_{F},\delta_{F})=(0,1)$, so the 
invariants $(\epsilon_{F},\delta_{F};r_{F},s_{F})$ must be different with 
that for a subgroup considered in the first case. The conclusion follows from 
\ref{L:O-Sp-U}. 
\end{proof}

\subsection{A class of elementary abelian 2-subgroups and symplectic metric 
spaces}\label{Subsection:a subclass} 

The elementary abelian 2-subgroups $F$ of $\O(n)/\langle-I\rangle$ (or 
$\Sp(n)/\langle-I\rangle$) with non-identity elements all conjugate to 
$[I_{\frac{n}{2},\frac{n}{2}}], [J_{\frac{n}{2}}]$ (or 
$[I_{\frac{n}{2},\frac{n}{2}}],[\textbf{i}I]$) have a particular nice shape.

\begin{prop}\label{Sub:definition}
For an elementary abelian 2 subgroup $F$ of $\O(n)/\langle-I\rangle$ (or 
$\Sp(n)/\langle-I\rangle$), any non-identity element of $F$ is conjugate to 
$[I_{\frac{n}{2},\frac{n}{2}}], [J_{\frac{n}{2}}]$ (or 
$[I_{\frac{n}{2},\frac{n}{2}}],[\textbf{i}I]$) if and only if any 
non-identity element of $A_{F}$ is conjugate to $[I_{\frac{n}{2},\frac{n}{2}}]$. 
\end{prop}
\begin{proof} 
Since elements in $F-A_{F}$ are all conjugate to $[I_{\frac{n}{2},\frac{n}{2}}], [J_{\frac{n}{2}}]$ 
(or $[I_{\frac{n}{2},\frac{n}{2}}],[\textbf{i}I]$) and any element of $A_{F}$ is not conjugate 
to $[J_{\frac{n}{2}}]$ (or $[\textbf{i}I]$), the conclusion follows.  
\end{proof}

Regard $A_{F}$ as a subgroup of $G'=\O(n')/\langle-I\rangle$,  
$\U(n')/\langle-I\rangle$ or $\Sp(n')/\langle-I\rangle$, where 
$n'=\frac{n}{2^{s+k}}$ ($k=2,1,0$), then the condition of any non-identity 
element of $A_{F}$ is conjugate to $[I_{\frac{n}{2},\frac{n}{2}}]$ in $G$ 
is equivalent to any non-identity element of $A_{F}$ is conjugate to 
$[I_{\frac{n'}{2},\frac{n'}{2}}]$ in $G'$. 

Let $F^{\ast}=\Hom(F,\mathbb{F}_2)$ be the dual group of an elementary abelian 2-group. 

\begin{lemma}\label{Lemma: subclass}
For $n=2^{m}s$, $s$ odd, let $K=\{\pm{1}\}^{n}/\langle\underbrace{(-1,...,-1)}
\rangle$ and $F\subset K$ be a rank $r$ subgroup with any non-identity element 
has $\frac{n}{2}$ components $-1$ and other $\frac{n}{2}$ components 1. 
Then one can divide $J=\{1,2,...,n\}$ into a disjoint union of $2^{r}$ subsets 
$\{J_{\alpha}: \alpha\in F^{\ast}\}$ with each $J_{\alpha}$ of cardinality 
$\frac{n}{2^{r}}=2^{m-r}s$ such that any element $x\in F$ is of the form 
\[x=\underbrace{[t_1,t_2,...,t_{n}]}, \forall i\in J_{\alpha}, t_{i}=\alpha(x).\]
\end{lemma}

\begin{proof}
Choose a representative of $F$ in $\{\pm{1}\}^{n}$, we may assume that 
$F\subset\{\pm{1}\}^{n}$ with any non-identity element has $\frac{n}{2}$ components 
$-1$ and other $\frac{n}{2}$ components 1. For any $x\in F$, let 
$x=\underbrace{(t_{x,1},t_{x,2},...,t_{x,n})}$, $t_{x,i}=\pm{1}$. For any $\alpha\in F^{\ast}$, 
define \[J_{\alpha}=\{1\leq i\leq n| \{x\in F: t_{x,i}=1\}=\ker\alpha\}.\]
Then $J=\{1,2,...,n\}$ is a disjoint union of $2^{r}$ subsets 
$\{J_{\alpha}: \alpha\in F^{\ast}\}$ and any element $x\in F$ is of the form 
$x=\underbrace{(t_1,t_2,...,t_{n})}$, $\forall i\in J_{\alpha}$, $t_{i}=\alpha(x)$. 
We only need to show that the cardinarity of each $J_{\alpha}$ is 
$\frac{n}{2^{r}}=2^{m-r}s$. 

Let $\alpha_0=0\in F^{\ast}$ be the zero element. 
For any $\alpha\neq\alpha_0$, count the number of pairs $(x,i)$ with 
$\alpha(x)=-1$ and $t_{x,i}=-1$. For a fixed $x\in F$, when $x\not\in\ker\alpha$, 
there are $\frac{n}{2}$ such $(x,i)$; when $x\in\ker\alpha$, there are
 no such $(x,i)$. For a fixed $i, 1\leq i\leq n$, when 
$x\not\in J_{\alpha_0}\cup J_{\alpha}$, there are $2^{r-2}$ such $(x,i)$; 
when $i\in J_{\alpha}$, there are $2^{r-1}$ such $(x,i)$; when 
$i\in J_{\alpha_0}$, there are no such $(x,i)$. So we have an equality 
\[2^{r-1}\frac{n}{2}=(n-|J_{\alpha}|-|J_{\alpha_0}|)2^{r-2}+|J_{\alpha}|2^{r-1},\] 
by which we get $|J_{\alpha}|=|J_{\alpha_0}|$. Then the cardinarity of each 
$J_{\alpha}$ is $\frac{n}{2^{r}}=2^{m-r}s$.  
\end{proof}

\begin{prop}\label{Subclass: classification}
For an elementary abelian 2-group $F\subset\O(n)/\langle-I\rangle$ (or 
$F\subset\Sp(n)/\langle-I\rangle$) with non-identity elements all conjugate to 
$[I_{\frac{n}{2},\frac{n}{2}}]$, $[J_{\frac{n}{2}}]$ (or 
$[I_{\frac{n}{2},\frac{n}{2}}]$, $[\textbf{i}I]$), the conjugacy class of $F$ 
is determined by the tuple $(\epsilon_{F},\delta_{F},r_{F},s_{F})$.
\end{prop}

\begin{proof}
This follows from propositions \ref{Type BD: classification}, 
\ref{Type C: classification}, \ref{Sub:definition} and Lemma \ref{Lemma: subclass}. 
\end{proof}

Let $F_{r,s,\epsilon,\delta}$ be an elementary abelian 2-subgroup of 
$\O(n)/\langle-I\rangle$ (or $\Sp(n)/\langle-I\rangle$) satisfying the properties in 
Proposition \ref{Subclass: classification} and with invariants $(\epsilon,\delta,r,s)$, 
which is unique up to conjugation.

\begin{definition}\label{D:symplectic metric space}
A finite-dimensional vector space $V$ over the field $\mathbb{F}_2=
\mathbb{Z}/2\mathbb{Z}$ is called a symplectic vector space if it is 
associated with a map $m: V\times V\longrightarrow \mathbb{F}_2$ such that 
$m(x,x)=0$, $m(x,y)=m(y,x)$ and $m(x+y,z)=m(x,z)m(y,z)$ for any $x,y,x\in F$. 

Moreover, it is called a symplectic metric space if there is 
another map $\mu: V\longrightarrow \mathbb{F}_2$ such that $\mu(0)=0$ and 
$m(x,y)=\mu(x)+\mu(y)+\mu(x+y)$ for any $x,y\in V$.

Two symplectic metric spaces $(V,m,\mu)$ and $(V',m',\mu')$ are called isomorphic 
if there exists a linear space isomorphism $f: V\longrightarrow V'$ transferring 
$(m,\mu)$ to $(m'\mu')$. 
\end{definition}

\begin{definition}\label{D:SMS invariants}
For a symplectic metric space $V$, define $A_{V}=\ker \mu|_{\ker m}$ and the 
defect index $\defe V=|\{x\in V: \mu(x)=1\}|-|\{x\in V:\mu(x)=-1\}|$. 
\begin{itemize} 
\item When $\mu|_{\ker m}\neq 1$, define $(\epsilon_{V},\delta_{V})=(1,0)$;
\item when $\mu|_{\ker m}=1$ and $\defe V<0$, define $(\epsilon_{V},\delta_{V})=(0,1)$;
\item when $\mu|_{\ker m}=1$ and $\defe V>0$, define $(\epsilon_{V},\delta_{V})=(0,0)$.
\end{itemize}
Define $r_{V}=\rank A_{V}$, $s_{V}=\frac{1}{2}\rank V/\ker m-\delta_{F}$.
\end{definition}

\begin{remark}\label{Arf}
When $m$ is non-degenerate, $\mu$ is a non-degenerate quadratic form, in this case 
$\delta_{V}$ is the Arf invariant of $\mu$. We would like to thank Professor Madsen 
for reminding this.   
\end{remark}

The following proposition is an analogue of \ref{Type BD: classification}.
\begin{prop}\label{Symplectic metric space: classification}
The isomorphism class of a symplectic metric space is determined by 
the invariants $(r_{V},s_{V},\epsilon_{V},\delta_{V})$.

We have $\defe V=(1-\epsilon)(-1)^{\delta}2^{r+s+\delta}$.


\end{prop}





\begin{prop}\label{rank three symplectic space}
For a vector space $V$ over $\bbF_2$ of rank 3 with a map $\mu: V\longrightarrow \bbF_2$ 
($\mu(0)=0$), let $m(x,y)=\mu(x)+\mu(y)+\mu(x+y)$. Then $(V,m, \mu)$ is a symplectic metric 
space if and only if $m$ is bilinear, if and only if there are even number of elements in $V$ with 
non-trivial values of the function $\mu$.
\end{prop}

\begin{proof}
$m$ is bilinear if and only if the sum of the values of $\mu$ over all elements of $V$ is 0,
if and only if there are even elements in $V$ with $\mu$ value $1$.

Other parts of this proposition are clear. 
\end{proof}

Let $V_{r,s;\epsilon,\delta}$ be a symplectic metric space with 
the prescribed invariants, which is unique up to isomorphism . Let 
$\Sp(r,s;\epsilon,\delta)$ be the group of automorphisms of $V_{r,s,\epsilon,\delta}$ 
preserving $m$ and $\mu$. Let $V_{s;\epsilon,\delta}=V_{0,s;\epsilon,\delta}$ and 
$\Sp(s;\epsilon,\delta)=Sp(0,s;\epsilon,\delta)$. It is clear that, 
$\Sp(r,s;\epsilon,\delta)=\Hom(V_{s;\epsilon,\delta},\mathbb{F}_2^{r})\rtimes
(\Sp(s;\epsilon,\delta)\times\GL(\mathbb{F}_2^{r}))$. Let $\Sp(s)=\Sp(s,\mathbb{F}_2)$ 
be the degree-$s$ symplectic group over the field $\mathbb{F}_2$. 

\begin{lemma}\label{Order of symplectic groups}
We have the following formulas for the orders of 
$\Sp(s;\epsilon,\delta)$, \begin{eqnarray*}
&&|\Sp(s;0,0)|=(\prod_{1\leq i\leq s-1}(2^{i+1}-1)(2^{i}+1))\cdot2^{s^{2}-s+1},
\\&&|\Sp(s-1;0,1)|=3\cdot(\prod_{1\leq i\leq s-1}(2^{i}-1)(2^{i+1}+1))\cdot2^{s^{2}-s+1},\\
&&|\Sp(s;1,0)|=|\Sp(s)|=(\prod_{1\leq i\leq s}(2^{i}-1)(2^{i}+1))2^{s^{2}}.\\
\end{eqnarray*}
\end{lemma}

\begin{proof}
When $s=1$ or 0, these are clear. So we just need to calculate 
\[|\Sp(s;\epsilon,\delta)|/|\Sp(s-1;\epsilon,\delta)|.\] We calculate it 
for the case $\epsilon=\delta=0$, other cases are similar.

$\Sp(s;0,0)$ permutes the non-identity elements $x\in V_{s;0,0}$ 
with $\mu{x}=1$, there are $\frac{2^{2s}+2^{s}}{2}-1=
(2^{s}-1)(2^{s-1}+1)$ such elements. Fix two distinct non-identity 
elements $x_1,x_2\in V_{s,0,0}$ with $\mu(x_1)=\mu(x_2)=1$ and 
$m(x_1,x_2)$. For any other $x$ with $\mu(x)=1$ and $(x_1,x)=-1$, 
$(x_1,x)$ is transformed to $(x_1,x_2)$ under some transformation 
in $\Sp(s:0,0)$. There are $2^{2s-2}$ such elements $x$. So 
$|\Sp(s;\epsilon,\delta)|/|\Sp(s-1;\epsilon,\delta)|=
(2^{s}-1)(2^{s-1}+1)2^{2s-2}$.
\end{proof}

Since \[V_{s;0,0}\oplus V_{0;1,0}\cong V_{s-1;0,1}\oplus V_{0;1,0}
\cong V_{s;1,0},\] $\Sp(s;0,0)$ and $\Sp(s-1;0,1)$ are subgroups 
of $\Sp(s;1,0)$. 

\begin{prop}\label{Compare symplectic 1}
$\Sp(s;1,0)\cong\Sp(s)$. 
\end{prop}

\begin{proof}
Since $V_{s;1,0}/\ker m=\mathbb{F}_2^{2s}$ is a symplectic space 
of dimension $2s$, by restriction we get a natural homomorphism 
$\pi:\Sp(s;1,0)\longrightarrow\Sp(s)$. 

Suppose $\pi(f)=1$ for some $f\in Sp(s;1,0)$, then for any 
$x\in V_{s;1,0}$, $f(x)=x$ or $f(x)=xz$, where $1\neq z\in\ker m$ 
(which is unique). Since $\mu(xz)=\mu(z)\mu(z)m(x,z)=-\mu(x)\neq\mu(x)$, 
so $f(x)=x$ for any $x\in V_{s;1,0}$. Thus $\pi$ is injective. 

Surjectivity of $\pi$ follows from $|\Sp(s;1,0)|=|\Sp(s)|$.
\end{proof}

Since an element in $\Sp(s;0,0)$ or $\Sp(s-1;0,1)$ 
preserves the symplectic form $m$ on $V=\bbF_2^{2s}$, so we have inclusions 
$\Sp(s;0,0)\subset\Sp(s)$ and $\Sp(s-1;0,1)\subset\Sp(s)$.

\begin{prop}\label{Compare symplectic 2}
The indices \[[\Sp(s):\Sp(s;0,0)]=2^{s-1}(2^{s}+1)\] and 
\[[\Sp(s):\Sp(s-1;0,1)]=2^{s-1}(2^{s}-1).\] 
\end{prop}

\begin{proof}
This follows from Proposition \ref{Order of symplectic groups} directly. 
\end{proof}

We also define groups $\Sp(s;t)$ ($s,t\geq 0$) as the automorphism 
group a symplecic vector space $(V,m)$ over $\mathbb{F}_2$ with $\rank V=2s+t$ and 
$\rank \ker m=t$. It is clear that $\Sp(s;0)=\Sp(s)$ and 
\[\Sp(s;t)=\Hom(\mathbb{F}_2^{2s},\mathbb{F}_2^{t})\rtimes(\GL(t,\mathbb{F}_2)\times\Sp(s)).\]

\section{Exceptional compact simple Lie groups (algebras)}\label{Section:exceptional}


For a complex simple Lie algebra $\frg$ with a specified Cartan subalgebra $\frh$, 
let $\Delta=\Delta(\frg,\frh)$ be the root system of $\frg$ and $B$ be its Killing form. 
For any $\lambda \in \frh^{\ast}$, let $H_{\lambda}\in \frh$ be determined by
\[B(H_{\lambda}, H)=\lambda(H), \forall H \in \frh.\] For any root $\alpha\in\Delta$, let
\[H'_{\alpha}=\frac{2}{\alpha(H_{\alpha})}H_{\alpha}.\] 
Then $\{H'_{\alpha}\}$ correspond to the co-roots
$\{\alpha^\vee=\frac{2\alpha}{\langle\alpha,\alpha\rangle}|\alpha\in\Delta\}$.

For any $\alpha\in\Delta$, let $X_{\alpha}$ be a root vector of the root $\alpha$, one can normalize the 
root vectors $\{X_{\alpha}, X_{-\alpha}\}$ so that 
$B(X_{\alpha}, X_{-\alpha})=2/\alpha(H_{\alpha})$ . Then
$[X_{\alpha}, X_{-\alpha}]=H'_\alpha$. Moreover, one can normalize
$\{X_{\alpha}\}$ appropriately, such that
\[\fru_0=\span\{X_{\alpha}-X_{-\alpha}, i(X_{\alpha}+X_{-\alpha}), i H_{\alpha}:
\alpha \in \Delta\}\] is a compact real form of $\frg$ (\cite{Knapp}
pages $348-354$). From now on, when we talk on a compact real simple Lie algebra, we just mean this 
specific real form of the complex simple Lie algebra $\frg$. 

We order the simple roots in Bourbaki style, and write $H'_{\alpha_{i}}$ as $H'_{i}$ for brevity 
for any simple root $\alpha_{i}$. 

\subsection{Involutions}\label{Subsection:involution}
Let $\fru_0$ be a compact simple Lie algebra and $G=\Aut(\fru_0)$ be its automorphism group. 
The conjugacy classes of involutions in $G$ are in one-one correspondence with the isomorphism 
classes of real forms of the complexified Lie algebra $\frg=\fru_0\otimes_{\bbR}\bbC$, and also 
in one-one correspondence with compact irreducible Riemannian symmetric pairs $(\fru_0,\frh_0)$. 
These objects were classified by \'Elie Cartan in 1920s. 
We give representatives of conjugacy classes of involutions the automorphism group 
$G=\Aut(\fru_0)$ for each compact simple exceptional Lie algebra $\fru_0$. 






i) Type $\bf E_6$. For $\fru_0=\fre_6$, let $\tau$ be a specific diagram involution
defined by
\begin{eqnarray*} &&\tau(H_{\alpha_1})=H_{\alpha_6},
\tau(H_{\alpha_6})=H_{\alpha_1}, \tau(H_{\alpha_3})=H_{\alpha_5},
\tau(H_{\alpha_5})=H_{\alpha_3},
\\&& \tau(H_{\alpha_2})=H_{\alpha_2},
\tau(H_{\alpha_4})=H_{\alpha_4},
\tau(X_{\pm{\alpha_1}})=X_{\pm{\alpha_6}},
\tau(X_{\pm{\alpha_6}})=X_{\pm{\alpha_1}},\\&&
\tau(X_{\pm{\alpha_3}})=X_{\pm{\alpha_5}},
\tau(X_{\pm{\alpha_5}})=X_{\pm{\alpha_3}},
\tau(X_{\pm{\alpha_2}})=X_{\pm{\alpha_2}},
\tau(X_{\pm{\alpha_4}})=X_{\pm{\alpha_4}}.
\end{eqnarray*} Let \[\sigma_1=\exp(\pi i H'_2), \sigma_2=\exp(\pi i
(H'_1+H'_6)), \sigma_3=\tau, \sigma_4=\tau \exp(\pi i H'_2).\] Then
$\sigma_1,\sigma_2,\sigma_3,\sigma_4$ represent all conjugacy
classes of involutions in $\Aut(\fru_0)$, which correspond to
Riemannian symmetric spaces of type {\bf EII, EIII, EIV, EI} and
the corresponding real forms are
$\fre_{6,-2},\fre_{6,14},\fre_{6,26},\fre_{6,-6}$.
$\sigma_1,\sigma_2$ are inner automorphisms, $\sigma_3,\sigma_4$
are outer automorphisms.

ii) Type $\bf E_7$. For $\fru_0=\fre_7$, let \[\sigma_1=\exp(\pi i H'_2),
\sigma_2=\exp(\pi i \frac{H'_2+H'_5+H'_7}{2}), \sigma_3=\exp(\pi i
\frac{H'_2+H'_5+H'_7+2H'_1}{2}).\] Then $\sigma_1,\sigma_2,\sigma_3$
represent all conjugacy classes of involutions in $\Aut(\fru_0)$,
which correspond to Riemannian symmetric spaces of type {\bf EVI, EVII, EV}
and the corresponding real forms are
$\fre_{7,3},\fre_{7,25},\fre_{7,-7}$.

iii) Type $\bf E_8$. For $\fru_0=\fre_8$, let \[\sigma_1=\exp(\pi i H'_2),
\sigma_2=\exp(\pi i (H'_2+H'_1)).\] Then $\sigma_1,\sigma_2$
represent all conjugacy classes of involutions in $\Aut(\fru_0)$,
which correspond to Riemannian symmetric spaces of type {\bf EIX, EVIII}
and the corresponding real forms are $\fre_{8,24},\fre_{8,-8}$.

iv) Type $\bf F_4$. For $\fru_0=\frf_4$, let \[\sigma_1=\exp(\pi i H'_1),
\sigma_2=\exp(\pi i H'_4).\] Then $\sigma_1,\sigma_2$ represent all
conjugacy classes of involutions in $\Aut(\fru_0)$, which correspond
to Riemannian symmetric spaces of type {\bf FI, FII} and the
corresponding real forms are $\frf_{4,-4},\frf_{4,20}$.

v) Type $\bf G_2$. For $\fru_0=\frg_2$, let $\sigma=\exp(\pi H'_1)$, 
which represents the unique conjugacy class of involutions in $\Aut(\fru_0)$,
corresponds to Riemannian symmetric space of type {\bf G} and the
corresponding real form is $\frg_{2,-2}$.

\smallskip

We remark that, in types $\E_8,\F_4,\G_2$, the automorphism groups of the simple 
Lie algebras, $\Aut(\fre_8),\Aut(\frf_4),\Aut(\frg_2)$, are connected and simply connected. 
$\Aut(\fre_6)$ is not connected and $\Int(\fre_6)$ is not simply connected, 
the image of the natural homomorphism $\pi: \E_6\longrightarrow\Aut(\fre_6)$ 
is $\Int(\fre_6)$ and the kernel of $\pi$ ($=Z(\E_6)$) is of odd 3. 
So $\E_6$ has two conjugacy classes of involutions with representatives 
$\sigma'_1=\exp(\pi i H'_2)$, $\sigma'_2=\exp(\pi i(H'_1+H'_6))$, where 
$\exp: \fre_6\longrightarrow\E_6$ is the exponential map of $\E_6$. 
$\Aut(\fre_7)$ is connected but not simply connected, the natural homomorphism 
$\pi: \E_7\longrightarrow\Aut(\fre_7)$ is surjective and the kernel of $\pi$ 
($=Z(\E_7)$) is of order 2. The preimages of $\sigma_2,\sigma_3$ are elements of order 4, 
and preimages of $\sigma_1$ are two non-conjugate involutions. So $\E_7$ has 
two conjugacy classes of involutions with representatives $\sigma'_1=\exp(\pi i H'_2)$, 
$\sigma'_2=\exp(\pi i(H'_1+H'_6))$, where $\exp: \fre_7\longrightarrow\E_7$ is 
the exponential map of $\E_7$.

There is an ascending sequence \[\F_4\subset\E_6\subset\E_7\subset\E_8,\] we observe that 
under these inclusions, the involutions $\sigma_2$ in $\F_4$, $\sigma_2'$ in $\E_6$, 
$\sigma'_2$ in $\E_7$ are conjugate to the involutions $\sigma'_2$ in $\E_6$, $\sigma_2'$ 
in $\E_7$, $\sigma_2$ in $\E_8$ when they are viewed as elements in the larger groups. 
The conjugacy class containing $\sigma_2$ (or $\sigma'_2$) is particularly important when 
we define the translation subgroup $A_{F}$ for an elementary abelian 2 subgroup $F$.     

The following Table 2 (cf, \cite{Huang-Yu} Tables 1$\&$2) describes the symmetric subgroup 
$\Aut(\fru_0)^{\theta}$ and its isotropic module $\frp=\frg^{-\theta}$ for each compact 
symmetric pair $(\fru_0,\frk_0)$ with $\fru_0$ an exceptional compact simple Lie algebra. 

\begin{table}[ht]\label{Ta:symmetric subgroups} 
\caption{Symmetric subgroups and isotropic modules}
\centering
\begin{tabular}{|c |c |c |c|}
\hline & $\theta$  & $\Aut(\fru_0)^{\theta}$ & $\frp$  \\ [0.3ex] \hline

{\bf EI} &$\sigma_4=\tau\exp(\pi i H'_{2})$ & $(Sp(4)/\langle-1\rangle)\times
\langle\theta\rangle$ & $V_{\omega_4}$    \\
\hline

{\bf EII}&$\sigma_1=\exp(\pi i H'_2)$&
$(\!SU\!(6)\!\!\times\!\!Sp(1)\!/\!\langle(e^{\frac{2\pi
i}{3}}\!I\!,\!1\!),\!(\!-\!I\!,\!-\!1\!)\rangle\!)
\!\rtimes\!\langle\tau\rangle$ & $\wedge^{3}\bbC^{6}\otimes\bbC^{2} $ \\
&&$\tau^{2}=1$, $\frk_0^{\tau}=\mathfrak{sp}(3)\oplus\mathfrak{sp}(1)$ &  
\\ \hline

{\bf EIII}&$\sigma_2=\exp(\pi i(H'_1+H'_6))$& 
$(\Spin(10)\times U(1)/\langle(c,i)\rangle)\rtimes\langle\tau\rangle$ & 
$(M_{+}\otimes 1)\oplus(M_{-}\otimes\overline{1})$ \\
&&$\tau^{2}=1$, $\frk_0^{\tau}=\mathfrak{so}(9)$ & \\ \hline

{\bf EIV}&$\sigma_3=\tau$& $F_4\times\langle\theta\rangle$  & $V_{\omega_4}$  \\
\hline

{\bf EV}&$\sigma_3=\exp(\pi i(H'_1+H'_0))$& 
$(SU(8)/\langle iI\rangle)\rtimes\langle\omega\rangle$  & 
$\wedge^{4}\bbC^{8}$ \\ &&$\omega^{2}=1$, $\frk_0^{\omega}=\mathfrak{sp}(4)$ &\\
\hline

{\bf EVI}&$\sigma_1=\exp(\pi i H'_2)$& 
$(\Spin(12)\!\times \!Sp(1))/\langle(c,1),(-1,-1)\rangle$ & 
$M_{+}\otimes\bbC^{2}$ \\
\hline

{\bf EVII}&$\sigma_2=\exp(\pi i H'_0)$& $((E_{6}\times
U(1))/\langle (c,e^{\frac{2\pi i}{3}})\rangle)
\rtimes\langle\omega\rangle$  &
$(V_{\omega_1}\otimes 1)\oplus(V_{\omega_6}\otimes\overline{1}) $ \\
&&$\omega^{2}=1$, $\frk_0^{\omega}=\frf_4$ & \\ 
\hline

{\bf EVIII}&$\sigma_2=\exp(\pi i(H'_1+H'_2))$& $\Spin(16)/\langle c\rangle$
   & $M_{+}$ \\
\hline

{\bf EIX}&$\sigma_1=\exp(\pi i H'_1)$&  $E_{7}\times Sp(1)/\langle(c,-1)\rangle$
  & $V_{\omega_7}\otimes\bbC^{2}$\\
\hline

{\bf FI}&$\sigma_1=\exp(\pi i H'_1)$& $(Sp(3)\times Sp(1))/\langle(-I,-1)\rangle$
  & $V_{\omega_3}\otimes\bbC^{2}$ \\
\hline

{\bf FII}&$\sigma_2=\exp(\pi i H'_4)$& $\Spin(9)$   & $M$ \\
\hline

{\bf G} &$\sigma=\exp(\pi i H'_1)$& $(Sp(1)\times Sp(1))/\langle(-1,-1)\rangle$ 
& $\Sym^{3}\bbC^{2}\otimes\bbC^{2} $ \\
\hline
\end{tabular}
\end{table}

\begin{remark}
We apologize that the notation $\sigma_{i}$ is used to represent conjugacy classes of 
involutions in all types of simple Lie groups (algebras). To avoid the ambiguity, we always 
specify in which group we are talking about conjugacy classes.  
\end{remark}

\subsection{Klein four subgroups}\label{Subsection:Klein}

In \cite{Huang-Yu} Section 4, we classified conjugacy classes of involutions in each 
symmetric subgroup $\Aut(\fru_0)^{\theta}$ and identify their conjugacy classes in 
$\Aut(\fru_0)$ (some details are also given in the following sections), 
constructed some Klein four subgroups of $\Aut(\fru_0)$ and described the conjugacy 
classes of involutions in them, it is showed that they represent all conjugacy classes. 

The groups $\Aut(\fre_6)$, $\Aut(\fre_7)$, $\Aut(\fre_8)$, $\Aut(\frf_4)$, 
$\Aut(\frg_2)$ have 8,8,4,3,1 conjugacy classes of Klein four subgroups in them, most of them 
are distinguished by their {\it involution types} (Definition \ref{Involution type}) except that 
the Klein four subgroups $\Gamma_1,\Gamma_2$ of $\Aut(\fre_7)$ have the same involution type 
(both are $(\sigma_1,\sigma,\sigma_1)$). $\Gamma_1,\Gamma_2\subset\Aut(\fre_7)$ can be 
characterized in this way, a Klein four subgroup $F\subset\E_7$ with $\pi(F)=\Gamma_1$ 
(respectively, $\pi(F)=\Gamma_2$) have an odd number of elements (respectively, an even number 
of elements) conjugate to $\sigma'_2$, where $\pi: \E_7\longrightarrow\Aut(\fre_7)$ is a double 
covering. That is equivalent to say, we can choose a Klein four subgroup $F\subset\E_7$ with 
$\pi(F)=\Gamma_1$ (respectively, $\pi(F)=\Gamma_2$) such that all of its involutions are 
conjugate to $\sigma'_1$ (respectively, $\sigma'_2$). 

Given a Klein four subgroup $F\subset G$, we have six different pairs $(\theta,\sigma)$ 
generating $F$, but some of them may be conjugate. 

\begin{thm}[\cite{Huang-Yu} Theorem 5.2]\label{T:regular}
Let $(\theta,\sigma)$, $(\theta',\sigma')$ be two pairs of commuting involutions in 
$G=\Aut(\fru_0)$ for $\fru_0$ a compact exceptional simple Lie algebra, then they are 
conjugate if and only if \[\theta\sim\theta', \sigma\sim\sigma',
\theta\sigma\sim\theta'\sigma'\] and the Klein four subgroups 
$\langle\theta,\sigma\rangle, \langle\theta',\sigma'\rangle$ are conjugate.
\end{thm}

We remark that, $\Aut(\fre_7)$ has two non-conjugate Klein four subgroups with 
involutions all conjugate to $\sigma_1$, so the condition of ``the Klein four subgroups 
$\langle\theta,\sigma\rangle$, $\langle\theta',\sigma'\rangle$ are conjugate'' can't be 
omitted in Theorem \ref{T:regular}. By Theorem \ref{T:regular}, Table 3  
also classifies conjugacy classes of ordered pairs of commuting involutions in $\Aut(\fru_0)$. 
Which is another approach to Berger 's classification of semisimple symmetric pairs.

\begin{table}[ht]\label{Ta:Klein four}
\caption{Klein four subgroups in $\Aut(\fru_0)$ for exceptional case}
\centering
\begin{tabular}{|c |c |c |c |c}
\hline $\fru_0$ & $\Gamma_{i}$ & $\frl_0=\fru_0^{\Gamma_{i}}$  & Type  \\ [0.5ex]
\hline
$\fre_6$& $\Gamma_{1}=\langle\exp(\pi i H'_2), \exp(\pi i H'_4)\rangle$ &
$(\mathfrak{su}(3))^{2}\oplus (i\bbR)^{2}$ & $(\sigma_1,\sigma_1,\sigma_1)$\\
\hline

 $\fre_6$& $\Gamma_{2}=\langle\exp(\pi i H'_4), \exp(\pi i (H'_3+H'_4+H'_5))\rangle$
 & $\mathfrak{su}(4) \oplus(\mathfrak{sp}(1))^{2}\oplus i \bbR$&  $(\sigma_1,\sigma_1,\sigma_2)$\\
\hline

$\fre_6$& $\Gamma_{3}=\langle\exp(\pi i (H'_2+H'_1)), \exp(\pi i
(H'_4+H'_1))\rangle$ & $\mathfrak{su}(5) \oplus (i\bbR)^{2}$ &
$(\sigma_1,\sigma_2,\sigma_2)$ \\
\hline

$\fre_6$& $\Gamma_{4}=\langle\exp(\pi i (H'_1+H'_6)), \exp(\pi i
(H'_3+H'_5))\rangle$ &  $\mathfrak{so}(8)\oplus (i\bbR)^{2}$  &
$(\sigma_2,\sigma_2,\sigma_2)$ \\
\hline

$\fre_6$& $\Gamma_{5}=\langle\exp(\pi i H'_2), \tau \rangle$ &
$\mathfrak{sp}(3) \oplus \mathfrak{sp}(1)$  &
$(\sigma_1,\sigma_3,\sigma_4)$ \\
\hline

$\fre_6$& $\Gamma_{6}=\langle\exp(\pi i H'_2), \tau \exp(\pi i
H'_4)\rangle$ & $\mathfrak{so}(6)\oplus i\bbR$  &
$(\sigma_1,\sigma_4,\sigma_4)$ \\
\hline

$\fre_6$& $\Gamma_{7}=\langle\exp(\pi i(H'_1+H'_6))), \tau\rangle$  &
$\mathfrak{so}(9)$ & $(\sigma_2,\sigma_3,\sigma_3)$ \\
\hline

$\fre_6$& $\Gamma_{8}=\langle\exp(\pi i(H'_1+H'_6)), \tau \exp(\pi i
H'_2)\rangle$ &  $\mathfrak{so}(5)\oplus \mathfrak{so}(5)$
& $(\sigma_2,\sigma_4,\sigma_4)$ \\
\hline

$\fre_7$& $\Gamma_{1}=\langle\exp(\pi i H'_2), \exp(\pi i H'_4)\rangle$ &
$\mathfrak{su}(6)\oplus (i\bbR)^{2}$&$(\sigma_1,\sigma_1,\sigma_1)$ \\
\hline

$\fre_7$& $\Gamma_{2}=\langle\exp(\pi i H'_2), \exp(\pi i H'_3)\rangle$ &
$\mathfrak{so}(8)\oplus (\mathfrak{sp}(1))^{3}$  &
$(\sigma_1,\sigma_1,\sigma_1)$ \\
\hline

$\fre_7$& $\Gamma_{3}=\langle\exp(\pi i H'_2), \tau \rangle$ &
$\mathfrak{so}(10) \oplus (i\bbR)^{2}$   &
$(\sigma_1,\sigma_2,\sigma_2)$ \\
\hline

$\fre_7$& $\Gamma_{4}=\langle\exp(\pi i H'_1), \tau \rangle$  &
$\mathfrak{su}(6)\oplus \mathfrak{sp}(1) \oplus i\bbR$   &
$(\sigma_1,\sigma_2,\sigma_3)$  \\
\hline

$\fre_7$& $\Gamma_{5}=\langle\exp(\pi i H'_2), \tau \exp(\pi i
H'_1)\rangle$ & $\mathfrak{su}(4)\oplus \mathfrak{su}(4) \oplus
i\bbR$ & $(\sigma_1,\sigma_3,\sigma_3)$ \\
\hline

$\fre_7$& $\Gamma_{6}=\langle\tau, \omega\rangle$ & $\frf_4$ &
$(\sigma_2,\sigma_2,\sigma_2)$ \\
\hline

$\fre_7$& $\Gamma_{7}=\langle\tau, \omega \exp(\pi i H'_1)\rangle$ &
$\mathfrak{sp}(4)$ &  $(\sigma_2,\sigma_3,\sigma_3)$ \\
\hline

$\fre_7$& $\Gamma_{8}=\langle\tau \exp(\pi i H'_1), \omega \exp(\pi i
H'_3)\rangle$ & $\mathfrak{so}(8)$ &
$(\sigma_3,\sigma_3,\sigma_3)$ \\
\hline

$\fre_8$& $\Gamma_{1}=\langle\exp(\pi i H'_2), \exp(\pi i H'_4)\rangle$ &
$\fre_6\oplus (i\bbR)^{2}$  &  $(\sigma_1,\sigma_1,\sigma_1)$ \\
\hline

$\fre_8$& $\Gamma_{2}=\langle\exp(\pi i H'_2), \exp(\pi i H'_1)\rangle$ &
$\mathfrak{so}(12)\oplus(\mathfrak{sp}(1))^{2}$ &$(\sigma_1,\sigma_1,\sigma_2)$ \\
\hline

$\fre_8$& $\Gamma_{3}=\langle\exp(\pi i H'_2), \exp(\pi i
(H'_1+H'_4))\rangle$
& $\mathfrak{su}(8)\oplus i\bbR$ & $(\sigma_1,\sigma_2,\sigma_2)$ \\
\hline

$\fre_8$& $\Gamma_{4}=\langle\exp(\pi i (H'_2+H'_1)),\exp(\pi i
(H'_5+H'_1))\rangle$ & $\mathfrak{so}(8)\oplus \mathfrak{so}(8)$ &
$(\sigma_2,\sigma_2,\sigma_2)$ \\
\hline

$\frf_4$& $\Gamma_{1}=\langle\exp(\pi i H'_2),\exp(\pi i H'_1)\rangle$ &
$\mathfrak{su}(3)\oplus (i\bbR)^{2}$ &  $(\sigma_1,\sigma_1,\sigma_1)$ \\
\hline

$\frf_4$& $\Gamma_{2}=\langle\exp(\pi i H'_3),\exp(\pi i H'_2)\rangle$ &
$\mathfrak{so}(5)\oplus(\mathfrak{sp}(1))^{2}$ &
$(\sigma_1,\sigma_1,\sigma_2)$ \\
\hline

$\frf_4$& $\Gamma_{3}=\langle\exp(\pi i H'_4),\exp(\pi i H'_3)\rangle$ &
$\mathfrak{so}(8)$& $(\sigma_2,\sigma_2,\sigma_2)$ \\
\hline

$\frg_2$& $\Gamma=\langle\exp(\pi i H'_1), \exp(\pi i H'_2)\rangle$ &
$(i\bbR)^{2}$ &  $(\sigma,\sigma,\sigma)$ \\
\hline

\end{tabular}
\end{table}

\subsection{An outline of the method of the classification}\label{Subsection:method}

When $\fru_0$ is of type $\G_2$, it turns out the maximal rank is 3 and there is a unique 
conjugacy class of elementary abelian 2-subgroups for each rank $\leq 3$. 

When $\fru_0$ is of type $\F_4$, as $\Aut(\frf_4)=\F_4$ doesn't possess a Klein four subgroup with 
involution type $(\sigma_1,\sigma_2,\sigma_2)$, the union of the subset of elements in an elementary 
abelian 2-subgroup $F$ of $\F_4$ which are conjugate to $\sigma_2$ and the identity element is a 
subgroup of $F$. Let $A_{F}$ be this subgroup, then $r=\rank A_{F}$ and $s=\rank F/A_{F}$ are 
obviously invariants (up to conjugation). Our work is to show the conjugacy class of $F$ is 
determined by $r,s$ and $r\leq 2$, $s\leq 3$. 

When $\fru_0$ is of type $\E_6$, we divide the elementary abelian 2-subgroups $F$ of 
$\Aut(\fre_6)$ into four (disjoint and exhausting) classes: 
\begin{itemize}
\item[1),]{$F$ contains an involution conjugate to $\sigma_3$;} 
\item[2),]{$F$ doesn't contain any element conjugate to $\sigma_3$, 
but contains one conjugate to $\sigma_4$;}
\item[3),]{$F\subset\Int(\fre_6)$ and it contains no Klein four subgroups conjugate to $\Gamma_3$;} 
\item[4),]{$F\subset\Int(\fre_6)$ and it contains a Klein four subgroup conjugate to $\Gamma_3$.} 
\end{itemize}

As $\Int(\fre_6)^{\sigma_3}\cong \F_4$ and 
$\Int(\fre_6)^{\sigma_4}\cong\Sp(4)/\langle(-I,-1)\rangle$, the classification in Class 1 reduces 
to $\F_4$ case; the classification in Class 2 reduces to $\Sp(4)/\langle(-I,-1)\rangle$ case, but only 
the subgroups with any involution conjugate to $\textbf{i}I$ or $\diag\{-I_2,I_2\}$ are concerned (see the 
detailed discussion in Section \ref{Section:E6} for the reason). Two amusing observations are, 
any subgroup in Class 3 is of the form $F\cap\Int(\fre_6)$ for a subgroup $F$ in Class 1, 
and any subgroup in Class 4 is of the form $F\cap\Int(\fre_6)$ for a subgroup $F$ in Class 2 
(with some additional condition).  

When $\fru_0$ is of type $\E_7$, we divide the elementary abelian 2-subgroups $F$ of 
$\Aut(\fre_7)$ into three disjoint classes: 
\begin{itemize}
\item[1),]{$F$ contains an involution conjugate to $\sigma_2$;} 
\item[2),]{$F$ doesn't contain any element conjugate to $\sigma_2$, 
but contains one conjugate to $\sigma_3$;}
\item[3),]{any involution in $F$ is conjugate to $\sigma_1$.} 
\end{itemize}

Since \[\Aut(\fre_7)^{\sigma_2}\cong((\E_6\times\U(1))/\langle(c,e^{\frac{2\pi i}{3}})\rangle
\rtimes\langle\omega\rangle,\] where $1\neq c\in Z_{\E_6}$, $\omega^{2}=1$, 
$(\fre_6\oplus i\bbR)^{\omega}=\frf_4\oplus 0$, $\sigma_2=(1,-1)$. Modulo 
$\U(1)$, we have a homomorphism $\pi: G^{\sigma_2}\longrightarrow\Aut(\fre_6)$. 
It turns out there is a bijection between conjugacy classes of elementary abelian 
2-subgroups of $\Aut(\fre_7)$ in Class 1 and elementary abelian 2-subgroups of 
$\Aut(\fre_6)$. 

Since $\Aut(\fre_7)^{\sigma_3}\cong(\SU(8)/\langle iI\rangle)\rtimes\langle\omega_0\rangle$, 
where $\omega_0^{2}=1$, $\omega_0 X\omega_0^{-1}=\overline{X}$ for any $X\in\SU(8)$, 
$\sigma_3=\frac{1+i}{\sqrt{2}}I$. So we have a homomorphism 
\[\pi: G^{\sigma_3}\longrightarrow\Aut(\mathfrak{su}(8))
=(\U(8)/\Z_{8})\rtimes\langle\omega_0\rangle.\] There is a bijection between 
conjugacy classes of elementary abelian 2-subgroups of $\Aut(\fre_7)$ in Class 2 and 
elementary abelian 2 subgroups of 
$\Aut(\mathfrak{su}(8))$ whose inner involutions are all conjugate to 
$I_{4,4}=\diag\{I_4,-I_4\}$ and outer involutions all conjugate to $\omega_0$. 

For an elementary abelian 2-subgroup $F$ of $\Aut(\fre_7)$ in Class 3, we show either $F$ is 
toral or it contains a rank 3 subgroup whose Klein four subgroups are all conjugate to $F_1$. 
In the first case, we can find an involution $\theta\in\Aut(\fre_7)^{F}$ such that elements 
in $\theta F$ are all conjugate to $\sigma_3$. In the second case, we can find a 
Klein four subgroup $F'\subset\Aut(\fre_7)^{F}$ conjugate to $\Gamma_6$. Then $F$ is 
a (canonical) subgroup of some (well-chosen) subgroup in Class 2 or Class 1.  

In type $\E_8$, $\Aut(\fre_8)=\E_8$ is connected, simply connected and of adjoint type. $\E_8$ 
has two conjugacy classes of involutions (with representatives $\sigma_1$,$\sigma_2$). A nice 
observation is, for an elementary abelian 2-subgroup $F$ of $\Aut(\fre_8)$ and an element $x$
of $F$ conjugate to $\sigma_1$, the subset \[H_{x}=\{y\in F|xy\not\sim y\}\] is a subgroup. 
We define $H_{F}$ as the subgroup generated by elements of $F$ conjugate to $\sigma_1$, and 
\[A_{F}=\{1\}\cup\{x\in F|x\sim\sigma_2, \textrm{ and }\forall y\in F-\{1,x\}, xy\sim y\}.\] 
$A_{F}\subset H_{F}$ if $H_{F}\neq 1$. 

Since \[\Aut(\fre_8)^{\Gamma_1}\cong((\E_6\times\U(1)\times\U(1))/\langle(c,e^{\frac{2\pi i}{3}},1)
\rangle\rtimes\langle\omega\rangle,\] where $1\neq c\in Z_{\E_6}$, $\omega^{2}=1$, 
$(\fre_6\oplus i\bbR\oplus i\bbR)^{\omega}=\frf_4\oplus 0\oplus 0$, 
$\Gamma_1=\langle(1,1,-1),(1,-1,1)\rangle$. Modulo 
$\U(1)\times \U(1)$, we have a homomorphism $\pi: G^{\sigma_2}\longrightarrow\Aut(\fre_6)$. 
$\pi$ doesn't give a bijection between conjugacy classes of elementary abelian 2-subgroups 
of $\Aut(\fre_8)$ containing a Klein four subgroup conjugate to $\Gamma_1$ and elementary 
abelian 2-subgroups of $\Aut(\fre_6)$. But we can distinguish this ambiguity explicitly. 
In this way, we get a classification of elementary abelian 2-subgroups of $\Aut(\fre_8)$ 
containing a Klein four subgroup conjugate to $\Gamma_1$.   

When $F$ doesn't contain any Klein four subgroup conjugate to $\Gamma_1$ and $H_{F}\neq 1$,
we can show $\rank H_{F}/A_{F}=1$, $\rank A_{F}\leq 3$ and $\rank F/H_{F}\leq 2$. Moreover, 
the conjugacy class of $F$ is determined by $\rank A_{F}$ and $\rank F/H_{F}$. 

When $H_{F}=1$, $\rank F\leq 5$ and the conjugacy class of $F$ is determined by $\rank F$.

\subsection{Some notions} 

\begin{definition}[Involution type]\label{Involution type}
For an elementary abelian 2 subgroup $F$ of a compact Lie group $G$, we call the distribution 
of conjugacy classes of involutions in $F$ the {\it involution type} of $F$.  
\end{definition}

\begin{definition}[Automizer group]\label{Automizer}
For an elementary abelian 2 subgroup $F$ of a compact Lie group $G$, 
let $C_{G}(F)=\{g\in G|gxg^{-1}=x,\forall x\in F\}$ and $N_{G}(F)=\{g\in G|gxg^{-1}\in F,\forall x\in F\}$.
we call $W(F)=N_{G}(F)/C_{G}(F)$ the Automizer group of $F$.  
\end{definition}

Conjugation action gives a natural inclusion $W(F)\subset\Aut(F)=\GL(\rank F,\mathbb{F}_2)$. 
Te get $W(F)$, we need to know which automorphisms of $F$ can the realized as $\Ad(g)$ 
for some $g\in G$. $W(F)$ is called Weyl group in \cite{Andersen-Grodal-Moller-Viruel}. 


We will also introduce other notions like {\it translation subgroup}, {\it defect index}, 
{\it residual rank} in the following sections. As the definitions of these notions depend on 
the types of the Lie algebras (or Lie groups), we give the precise definitions 
in each section below.  These notions help us to show the subgroups we constructed in 
different classes or in the same class but with different parameters are non-conjugate.


\section{$\G_2$}\label{Section:G2}
For $G=\G_2$, it is of adjoint type and simply connected, there exists 
a unique conjugacy class of involutions in $G$. For any involution 
$x\in G$, $G^{x}\cong\Sp(1)\times\Sp(1)/\langle(-1,-1)\rangle$. 
There is a unique conjugacy class of ordered pairs $(x,y)$ generating 
a Klein four subgroup. For any of such a pair $(x,y)$, 
$G^{x,y}\cong(\U(1)\times\U(1))\rtimes\langle t \rangle$, where $t^{2}=1$ 
and \[\forall z_1,z_2 \in\U(1),\Ad(t)(z_1,z_2)=(z_1^{-1},z_2^{-1}).\] 
There is a unique conjugacy class of ordered triples $(x,y,z)$ generating a 
rank 3 elementary abelian 2-subgroup. For any such a triple
$(x,y,z)$, $G^{x,y,z}=\langle x,y,z \rangle$. 
There are no elementary abelian 2-subgroups of $G$ of rank more than $3$.

Let $F_{r}$ be an elementary abelian 2-subgroup of $G=\G_2$ of rank $r$, then 
\[W(F_{r})=N_{G}(F_{r})/C_{G}(F_{r})\cong\GL(r,\mathbb{F}_2).\]

\section{$\F_4$}\label{Section:F4}

Let $G=\F_4$, which is of adjoint type and simply connected. There are 
two conjugacy classes of involutions in $G$ with $\sigma_1,\sigma_2$ as  
their representatives, we have 
\[G^{\sigma_1}\cong\Sp(3)\times\Sp(1)/\langle (-I,-1)\rangle,\] and  
\[G^{\sigma_2} \cong\Spin(9).\] 

There are three conjugacy classes of Klein four subgroups in $G$ 
with $\Gamma_1,\Gamma_2,\Gamma_3$ as their representatives, whose involution types are 
\[\Gamma_1:(\sigma_1,\sigma_1,\sigma_1);\ \Gamma_2:(\sigma_1,\sigma_1,\sigma_2);\  
\Gamma_3:(\sigma_2,\sigma_2,\sigma_2).\]

For $\sigma_1$, $G^{\sigma_1}=\Sp(3)\times\Sp(1)/\langle (-I,-1)\rangle$, and we have 
$\frp\cong V_{\omega_4}\otimes\bbC^{2}$ as an $\Sp(3)\times\Sp(1)$ module.
$G^{\sigma_1} $has three conjugacy 
classes of involutions except $\sigma_1=(-I,1)=(I,-1)$, their representatives 
are $(\textbf{i}I,\textbf{i})$, $\big(\left(\begin{array}{ccc}-1&&\\&1&\\&&1\\
\end{array}\right),1 \big)$, $\big(\left(\begin{array}
{ccc}-1&&\\&-1&\\&&1\\\end{array}\right), 1\big)$. And in $G$, we have conjugacy  
relations $\big(\textbf{i}I,\textbf{i}\big)\sim\sigma_1$, 
\begin{eqnarray*}&&\big(\left(\begin{array}{ccc}-1&&\\&1&\\&&1\\
\end{array}\right),1 \big)\sim\sigma_1,\ \big(\left(\begin{array}
{ccc}-1&&\\&-1&\\&&1\\\end{array}\right), 1\big)\sim\sigma_2.\end{eqnarray*}
We explain why $(\diag\{-1,1,1\},1)\sim\sigma_1$. Since the weight space of 
$V_{\omega_4}$ (an $\Sp(3)$ module) is $\{\pm{e_1}\pm{e_2}\pm{e_3},\pm{e_1},\pm{e_2},\pm{e_3}\}$ 
(each of multiplicity one), the subspace fixed by $\diag\{-1,I_2\}$ are the sum of weights spaces 
of weights $\{\pm{e_2}\pm{e_3},\pm{e_2},\pm{e_3}\}$, so 
\[\dim\fru_0^{(\diag\{-1,1,1\},1)}=\dim(\mathfrak{sp}(1)+\mathfrak{sp}(2)+\mathfrak{sp}(1))+8=24,\] 
so $(\diag\{-1,1,1\},1)\sim\sigma_1$. Similar calculations lead to other conjugacy relations. 

For $\sigma_2$, $G^{\sigma_2}=\Spin(9)$, and $\frp\cong M$, the Spinor module of $\Spin(9)$. 
$G^{\sigma_2}$  has two conjugacy classes of involutions except 
$\sigma_2=-1$, their representatives are $e_1e_2e_3e_4$, $e_1e_2e_3e_4e_5e_6e_7e_8$. 
And in $G$, we have conjugacy relations \[e_1e_2e_3e_4\sim\sigma_1, e_1e_2...e_8\sim\sigma_2.\]
The consideration on weights calculates $\dim\fru_0^{\sigma}$, from which we get 
$\dim\fru_0^{\sigma}$ for all relevant $\sigma$ and the conjugacy relations.

\smallskip

In $G^{\sigma_1}=\Sp(3)\times\Sp(1)/\langle(-I,-1)\rangle$, let 
$x_1=(I,-1)$, $x_2=(\textbf{i}I,\textbf{i})$, $x_3=(\textbf{j}I,\textbf{j})$, 
\begin{eqnarray*}&& x_4=\left(\begin{array}{ccc} -1&&\\&-1&\\&&1\\\end{array}\right),
x_5=\left(\begin{array}{ccc} -1&&\\&1&\\&&-1\\\end{array}\right).\end{eqnarray*} 
For $0\leq r\leq 2, 0\leq s\leq 3$, define \begin{eqnarray*}&&F_{r,s}=
\langle x_{1},...,x_{s},x_{4},...,x_{3+r}\rangle,\ 
A_{r}=\langle  x_{4},...,x_{3+r}\rangle.\end{eqnarray*}

\begin{definition}
For an elementary abelian 2-subgroup $F\subset G=\F_4$, define  
\[A_{F}=\{x\in F: x\sim\sigma_2\}\cup\{1\}.\]
\end{definition}

\begin{lemma} \label{F 4 r,s}
For an elementary abelian $2$ subgroup $F\subset G$, $A_{F}$ is a subgroup 
with $\rank{A_{F}}\leq 2$ and $\rank(F/A_{F})\leq 3$.

For each $r,s$ with $0\leq r\leq 2$, $0\leq s\leq 3$, there exists a unique 
conjugacy class of elementary abelian 2 groups $F\subset F_4$ with 
$\rank{A_{F}}=r$ and $\rank(F/A_{F})=s$.
\end{lemma}

\begin{proof}
Let $F\subset G$ be an elementary abelian $2$-group. Since there are no 
Klein four subgroups of $G$ with involutions type 
$(\sigma_1,\sigma_2,\sigma_2)$, for any distinct non-identity elements 
$x,y\in F$ with $x\sim y\sim \sigma_2$, we have $xy\sim \sigma_2$. 
So $A_{F}$ is a subgroup.

In $G^{\sigma_2}=\Spin(9)$, besides $\sigma_2=-1$, the elements conjugate 
to $\sigma_2$ in $G$ are all conjugate to $e_1e_2\cdots e_8$ in $\Spin(8)$. 
There doesn't exist $x,y\in\Spin(9)$ with $x,y,xy$ all conjugate to 
$e_1e_2\cdots e_8$, so $\rank A_{F}\leq 2$.

In $G^{\sigma_1}\cong\Sp(3)\times\Sp(1)/\langle (-I,-1)\rangle$, the elements 
$x$ with $x,\sigma_1 x=(-I,1)x$ both conjugate to $\sigma_1$ in $G$ are 
all conjugate to $(\textbf{i} I,\textbf{i})$. By this, it is obvious that 
any elementary abelian 2-subgroup of $G$ whose non-identity elements 
all conjugate to $\sigma_1$ has rank at most 3 
($\langle\sigma_1,(\textbf{i} I,\textbf{i}),(\textbf{j} I,\textbf{j})\rangle$ 
is a rank 3 example). Since non-identity elements of a complement of $A_{F}$ 
in $F$ are all conjugate to $\sigma_1$, so $\rank F/A_{F}\leq 3$.

In $F=F_{r,s}$, we have $A_{F}=A_{r}$ is of rank $r$ and 
$F/A_{F}=F_{r,s}/A_{r}$ is of rank $s$, so $F=F_{r,s}$ satisfies 
$\rank A_{F}=r$ and $\rank F/A_{F}=s$.

When $s=0$, the uniquness is showed in the proof for $r\leq 2$ as above. When 
$s\geq 1$, one argues in $G^{\sigma_1}\cong\Sp(3)\times\Sp(1)/\langle(-I,-1)\rangle$ 
to get the uniqueness.
\end{proof}

\begin{prop}\label{F4:uniqueness} 
For two elementary abelian 2-subgroups $F,F'\subset G=\F_4$, if $f: F\longrightarrow F'$ 
is an isomorphism with $f(x)\sim x, \forall x\in F$, then there exists $g\in G$ such 
that $f=\Ad(g)$. \end{prop} 
\begin{proof}
This is proved in the proof of Lemma \ref{F 4 r,s}
\end{proof}

\begin{thm} \label{F4}
For any $r\leq 2$, $s\leq 3$, $W(F_{r,s})\cong P(r,s,F_2)$, where $P(r,s,F_2)$ is 
the group of $(r,s)$ blockwise upper triangular matrices in $GL(r+s,F_2)$.
\end{thm}

\begin{proof}
For $F=F_{r,s}$, $A_{F}=A_{r}$ and any $g\in N_{G}(F)$ satisfies 
$g A_{r} g^{-1}=A_{r}$. By Proposition \ref{F4:uniqueness}, we have 
$W(F)=N_{G}(F)/C_{G}(F)\cong P(r,s,F_2)$.
\end{proof}

\section{$\E_6$}\label{Section:E6}

Let $G=\Aut(\fre_6)$, then $G_0=\Int(\fre_6)$ and $G/G_0\cong\bbZ/2\bbZ$. There are 
four conjugacy classes of involutions in $G$, two of them consist of inner automorphisms 
(with representatives $\sigma_1,\sigma_2$) and the other two consist of outer automorphisms 
(with representatives $\sigma_3,\sigma_4$), we have 
\[G_0^{\sigma_1}\cong\SU(6)\times\Sp(1)/\langle(e^{\frac{2\pi i}{3}}I,1),(-I,-1)\rangle,\] 
\[G_0^{\sigma_2} \cong\Spin(10)\times\U(1)/\langle(c,i)\rangle, c=e_1e_2...e_{10},\]  
\[G_0^{\sigma_3}\cong F_4, \quad G_0^{\sigma_4}\cong Sp(4)/
\langle -I\rangle.\]

There are $8$ conjugacy classes of Klein four subgroups in $G$ with representatives 
$\Gamma_1,\Gamma_2,\Gamma_3,\Gamma_4,\Gamma_5,\Gamma_6,\Gamma_7,\Gamma_8$, whose involution types are 
\begin{eqnarray*} &&\Gamma_1:(\sigma_1,\sigma_1,\sigma_1),\quad 
\Gamma_2:(\sigma_1,\sigma_1,\sigma_2),\quad\Gamma_3:(\sigma_1,\sigma_2,\sigma_2),\quad 
\Gamma_4:(\sigma_2,\sigma_2,\sigma_2),\\&&\Gamma_5:(\sigma_1,\sigma_3,\sigma_4),\quad 
\Gamma_6:(\sigma_1,\sigma_4,\sigma_4),\quad\Gamma_7:(\sigma_2,\sigma_3,\sigma_3),\quad 
\Gamma_8:(\sigma_2,\sigma_4,\sigma_4).\end{eqnarray*}

\smallskip

For $\sigma_1$, $G_0^{\sigma_1}\cong\SU(6)\times\Sp(1)/\langle(e^{\frac{2\pi i}{3}}I,1),
(-I,-1)\rangle$ and $\frp=\wedge^{3}(\bbC^{6})\otimes\bbC^2$. Representatives of conjugacy classes 
of involutions in $G_0^{\sigma_1}$ except $\sigma_1=(I,-1)$ and their conjugacy classes in $G$ are 
as follows, 
\begin{eqnarray*}&& \big(\left(\begin{array}
{cc}-I_{4}&\\&I_{2}\\\end{array}\right),1\big)\sim \sigma_2,\big(\left(
\begin{array}{cc}-I_{2}&\\&I_{4}\\\end{array}\right),1\big)\sim \sigma_1,
\\&&\big(\left(\begin{array}{cc}iI_{5}&\\&-i\\\end{array}\right),
i\big)\sim\sigma_2,\big(\left(\begin{array}{cc}iI_{3}&\\&-iI_{3}\\
\end{array}\right), i\big)\sim \sigma_1.\end{eqnarray*}
The identification of conjugacy classes is obtained by calculating all relevant 
$\dim\fru_0^{\sigma}$.

For $\sigma_2$, $G_0^{\sigma_2}\cong \Spin(10)\times\U(1)/\langle(c,i)\rangle$ and 
$\frp=(M_{+}\otimes 1)\oplus(M_{-}\otimes\overline{1})$, where $1$ is the natural module 
(not the trivial module) of $\U(1)$, i.e., $z\mapsto z,\forall z\in\U(1)$, and $\overline{1}$ is 
its contragredient module.  The representatives of conjugacy classes 
of involutions in $G_0^{\sigma_2}$ except $\sigma_2=(-1,1)=(1,-1)$ and their conjugacy classes 
in $G$ are as follows, \[\big(e_1e_2e_3e_4, 1 \big)\sim \sigma_1,\   
\big(e_1e_2...e_8,1\big)\sim \sigma_2,\] 
\[\big(\Pi,\frac{1+i}{\sqrt{2}}\big)\sim\sigma_2,\ 
\big(-\Pi,\frac{1+i}{\sqrt{2}}\big)\sim \sigma_1,\] 
where \[\Pi=\frac{1+e_1e_2}{\sqrt{2}}\frac{1+e_3e_4}{\sqrt{2}}\cdots
\frac{1+e_9e_{10}}{\sqrt{2}}.\] The identification of conjugacy classes is obtained by 
calculating $\fru_0^{\sigma}$ from the weights of $M_{+},M_{-}$ as $\Spin(10)$ modules. 

For an elementary abelian $2$-group $F\subset G$, define 
$\mu: F\cap G_0\longrightarrow\{\pm{1}\}$ by $\mu(y)=-1$ if $y\sim\sigma_1$, 
and $\mu(y)=1$ if $y\sim\sigma_2$. Define \[m: (F\cap G_0)\times(F\cap G_0)
\longrightarrow\{\pm{1}\}\] by $m(y_1,y_2)=\mu(y_1y_2)\mu(y_1)\mu(y_2)$. 
Here $m$ is not always a bilinear form. 

\begin{definition}
Define 
\[A_{F}=\{x\in H\cap G_0| \mu(x)=1\textrm{ and } m(x,y)=1,\forall y\in F\cap G_0\}\] and 
$\defe(F)=|\{y\in F\cap G_0: \mu(y)=1\}|-|\{y\in F\cap G_0: \mu(y)=-1\}|$. 
\end{definition}

We call $A_{F}$ the {\it translation subgroup} of $F$, which has an equivalent 
definition \[A_{F}=\{1\}\cup\{x\in F|x\sim\sigma_2, \textrm{ and } 
y\sim xy \textrm{ for any } y\in F-\langle x\rangle, \}.\]

\subsection{Subgroups from $\F_4$}

In $G_0^{\sigma_3}\cong\F_4$, let $\tau_1,\tau_2$ be involutions with 
$\frf_4^{\tau_1}\cong \mathfrak{sp}(3)\oplus\mathfrak{sp}(1)$, 
$\frf_4^{\tau_2}\cong \mathfrak{so}(9)$, then $\tau_1,\tau_2,\sigma_3\tau_1,\sigma_3\tau_2$ 
represent all conjugacy classes of involutions in $G^{\sigma_3}$ except $\sigma_3$. And in $G$ we have 
conjugacy relations $\tau_1\sim\sigma_1$, $\tau_2\sim \sigma_2$, 
\[\sigma_3\tau_1\sim\sigma_4,\ \sigma_3\tau_2\sim\sigma_3.\]
Since we have $\sigma_3=\tau$ and we can choose $\tau_1=\exp(\pi iH'_2)$ and 
$\tau_2=\exp(\pi i (H'_1+H'_6))$. Then the above conjugacy relations are clear. 

We have $F_4^{\tau_1}\cong\Sp(3)\times\Sp(1)/\langle (-I,-1)\rangle$. 
Let $x_0=\sigma_3$, $x_1=\tau_2=\big(I,-1\big)$, \[x_2=\big(\textbf{i}I,\textbf{i}\big), 
x_3=\big(\textbf{j}I,\textbf{j}\big),\]  
\[x_4=\big(\left(\begin {array}{ccc}-1&&\\&-1&\\&&1\\ \end{array}\right),1\big),   
x_5=\big(\left(\begin{array}{ccc} -1&&\\&1&\\&&-1\\\end{array}\right),1\big).\]
Define \[F_{r,s}=\langle x_0,x_1,...x_{s},x_{4},...,x_{3+r}\rangle, \] and 
\[F'_{r,s}=\langle x_1,...x_{s},x_{4}, ...,x_{3+r} \rangle.\]

\begin{prop}\label{E6 to F4}
For an elementary abelian $2$-group $F\subset G$,if $F$ contains  
an element conjugate to $\sigma_3$, then $F\sim F_{r,s}$ for some 
$(r,s)$ with $r\leq 2, s\leq 3$; if $F\subset G_0$ and contains no  
Klein four subgroups conjugate to $\Gamma_3$, then $F\sim F'_{r,s}$ 
for some$(r,s)$ with $r\leq 2, s\leq 3$.
\end{prop}

\begin{proof}
For the first statement, we may assume that $\sigma_3\in F$, then 
$F\subset G^{\sigma_3}=F_4\times\langle\sigma_3\rangle$. So 
$F\sim F_{r,s}$ ($r\leq 2, s\leq 3$) by Theorem \ref{F4}.

For the latter statement, since $F$ doesn't contain any Klein four subgroup 
of involutions type $(\sigma_1,\sigma_2,\sigma_2)$, so 
$A_{F}=\{1\}\cup\{x\in F|x\sim\sigma_2\}$. Consideration in the groups $G_0^{\sigma_1}\cong
\SU(6)\times\Sp(1)/\langle(e^{\frac{2\pi i}{3}}I,1),(-I,-1)\rangle$ and 
$G_0^{\sigma_2}\cong \Spin(10)\times\U(1)/\langle(c,i)\rangle$ (with the 
comparison of conjugacy classes of involutions in them and $G$) enables 
we to show that $\rank A_{F}\leq 2$, $\rank(F/A_{F})\leq 3$ 
and the conjugacy class of $F$ is uniquely determined by $\rank A_{F}$ and 
$\rank(F/A_{F})$. The proof is similar as the proof for Lemma \ref{F 4 r,s}. 
Then $F\sim F'_{r,s}$ ($r\leq 2, s\leq 3$) since $F'_{r,s}$ has the same invariants as $F$.
\end{proof}

\begin{lemma}\label{E6 to F4:m}
For an subgroup $F$ in Proposition \ref{E6 to F4}, 
\[\forall x,y\in F\cap G_0,\ m(x,y)=-1\Leftrightarrow 
x,y\in F\cap G_0-A_{F}.\] 
\end{lemma}

\begin{proof}
It follows from $\{1\}\cup\{x\in F|x\sim\sigma_2\}=A_{F}$ and it is a subgroup. 
\end{proof}

\subsection{Subgroups from $\Sp(4)/\langle-I\rangle$}

For $\sigma_4$, $G_0^{\sigma_4}\cong\Sp(4)/\langle -I\rangle$ and $\frp\cong V_{\omega_4}$. Let 
$\tau_1=\textbf{i}I$, $\tau_2=\left(\begin{array}{cc}-I_{2}&\\&I_{2}\\
\end{array}\right)$, $\tau_3=\left(\begin{array}{cc}-1&\\&I_{3}\\ \end{array}
\right)$, then $\tau_1,\tau_2,\tau_3, \sigma_4\tau_1,\sigma_4\tau_2,\sigma_4\tau_3$  
represent all conjugacy classes of involutions in $G^{\sigma_4}$ except $\sigma_4$. 
And in $G$ we have conjugacy relations $\tau_1\sim\sigma_1$, $\tau_2\sim\sigma_2$, 
$\tau_3\sim\sigma_1$, \[\sigma_4\tau_1\sim\sigma_4,\ \sigma_4\tau_2\sim\sigma_4,\ 
\sigma_4\tau_3\sim\sigma_3.\] The identification is obtained by calculating $\dim\fru_0^{\sigma}$ 
with considering the weights of $V_{\omega_4}$ as an $\Sp(4)$ module. 

Let $x_0=\sigma_4$, $x_1=\textbf{i}I$, $x_2=\textbf{j}I$, 
\[x_3=\left(\begin{array}{cc}-I_{2}&\\&I_{2}\\ \end{array}\right), 
x_4=\left(\begin{array}{cc}0&I_{2}\\I_{2}&0\\ \end{array}\right),\]  
\[x_5=\left(\begin{array}{cccc}1&0&&\\0&-1&&\\&&1&0\\&&0&-1\\ \end{array}\right),
x_6=\left(\begin{array}{cccc}0&1&&\\1&0&&\\&&0&1\\&&1&0\\ \end{array}\right).\]  
For any $(\epsilon,\delta,r,s)$ with 
$\epsilon+\delta\leq 1, r+s\leq 2$, define \[F_{\epsilon,\delta,r,s}=
\langle x_0,x_1,...,x_{\epsilon+2\delta},x_3,...,x_{r+2s}\rangle\] 
and \[F'_{\epsilon,\delta,r,s}=\langle x_1,...,x_{\epsilon+2\delta},
x_3,...,x_{r+2s}\rangle.\]

\begin{prop}\label{E6 to C4}
For an elementary abelian $2$-group $F\subset G$, if $F\not\subset G_0$
and it contains no elements conjugate to $\sigma_3$, then 
$F\sim F_{\epsilon,\delta,r,s}$ for some $(\epsilon,\delta,r,s)$ with 
$\epsilon+\delta\leq 1$ and $r+s\leq 2$; if $F\subset G_0=\Int(\fre_6)$ 
and contains a Klein four subgroup conjugate to $\Gamma_3$, then 
$F\sim F'_{\epsilon,\delta,r,s}$ for some $(\epsilon,\delta,r,s)$ with 
$\epsilon+\delta\leq 1, r+s\leq 2, s\geq 1$.
\end{prop}

\begin{proof}
For the first statement, we may assume that $\sigma_4\in F$, then 
$F\cap G_0\subset G_0^{\sigma_4}\cong\Sp(4)/\langle-I\rangle$. Any involution 
in $\Sp(4)/\langle-I\rangle$ is conjugate to one of 
\[\tau_1=[\textbf{i}I], \tau_2=[\diag\{I_2,-I_2\}], \tau_3=[\diag\{1,-I_{3}\}].\]
Since $\sigma_4\tau_3\sim\sigma_3$ in $G$ and we assume $F$ contains no elements 
conjugate to $\sigma_3$, so any non-identity element of $F\cap G_0$ is conjugate 
to $\tau_1$ or $\tau_2$ in $\Sp(4)/\langle-I\rangle$. Then 
$F\cap G_0\subset\Sp(4)/\langle-I\rangle$ is in the subclass discussed in Subsection 
\ref{Subsection:a subclass}. Then $F\sim F_{\epsilon,\delta,r,s}$ for some $(\epsilon,\delta,r,s)$ 
with $\epsilon+\delta\leq 1$ and $r+s\leq 2$ by Proposition \ref{Subclass: classification}.

For the second statement, we may assume that $\Gamma_3\subset F$, then 
\[F\subset G_0^{\Gamma_3}\cong(\U(5)\times\U(1))/\langle(-I,-1),(e^{\frac{2\pi i}{3}},1)
\rangle\cong(\U(5)/\langle e^{\frac{2\pi i}{3}}\rangle)\times\U(1).\] 
($(A,\lambda)\longmapsto(\lambda A,\lambda^{2})$ gives an isomorphism 
$(\U(5)\times\U(1))/\langle(-I,-1) \rangle\cong\U(5)\times\U(1)$.) 
Since any abelian subgroup of $\U(5)\times\U(1)$ is total, so $F\subset G_0$ is total. 
We may assume that $F\subset \exp(\frh_0)$ for a maximal torus $\frh_0$ of 
$\fru_0=\fre_6$. Choose a Chevelley involution $\theta$ of $\fru_0$ with respect to 
$\frh_0$. Then $\langle F,\theta\rangle$ is an elementary abelian $2$-group without 
elements conjugate to $\sigma_3$. By the first statement,  
$\langle F,\theta\rangle\sim F_{\epsilon,\delta,r,s}$ for some $(r,s)$. Then 
$F\sim F'_{\epsilon,\delta,r,s}$. Since we assume $F$ contains a Klein four subgroup conjugate to 
$\Gamma_3$, so $s\geq 1$.
\end{proof}

For an elementary abelian $2$-group $F\subset G$ without elements conjugate to $\sigma_3$,  
but contains an element conjugate to $\sigma_4$, for any $x\in F$ with $x\sim\sigma_4$, we have 
$F\cap G_0\subset G_0^{\sigma_4}\cong\Sp(4)/\langle-I\rangle$, with this inclusion we have a function 
$\mu_{x}:F\cap G_0\longrightarrow\{\pm{1}\}$ and a map 
$m_{x}:(F\cap G_0)\times(F\cap G_0)\longrightarrow\{\pm{1}\}$ (Subsection \ref{Subsection:BDC}).  

\begin{lemma}\label{E6 to C4:m} 
We have $\mu_{x}=\mu$ and $m_{x}=m$. 
\end{lemma}
\begin{proof}
We may assume that $x=\sigma_4$, then $F\cap G_0\subset G_0^{\sigma_4}\cong\Sp(4)/\langle-I\rangle$. 
As $F$ doesn't have any element conjugate to $\sigma_3$, any element of $F\cap G_0$ is conjugate to 
$\tau_1=\textbf{i}I$ or $\tau_2=\left(\begin{array}{cc}-I_{2}&\\&I_{2}\\
\end{array}\right)$ in $G_0^{\sigma_4}\cong\Sp(4)/\langle-I\rangle$. Since 
$\tau_1\sim_{G}\sigma_1$ and $\tau_2\sim_{G}\sigma_2$, so $\mu_{x}=\mu$. Then $m_{x}=m$ as well.
\end{proof}

\subsection{Automizer groups}

\begin{prop}\label{E6 rank and defect}
We have the following formuas for $\rank A_{F}$ and $\defe F$, 
\begin{itemize}
\item[(1)]{for $F=F_{r,s}$, $r\leq 2$, $s\leq 3$, $\rank A_{F}=r$, $\defe F=2^{r}(2-2^{s})$;}
\item[(2)]{for $F=F'_{r,s}$, $r\leq 2$, $s\leq 3$, $\rank A_{F}=r$, $\defe F=2^{r}(2-2^{s})$;}
\item[(3)]{for $F=F_{\epsilon,\delta,r,s}$, $\epsilon+\delta\leq 1$, $r+s\leq 2$, 
$\rank A_{F}=r$, $\defe F=(1-\epsilon)(-1)^{\delta}2^{r+s+\delta}$;}
\item[(4)]{for $F=F'_{\epsilon,\delta,r,s}$, $\epsilon+\delta\leq 1$, $r+s\leq 2$, $s\geq 1$, 
$\rank A_{F}=r$, $\defe F= (1-\epsilon)(-1)^{\delta} 2^{r+s+\delta}$.}
\end{itemize}
\end{prop}
\begin{proof}
They follow from Lemmas \ref{E6 to F4:m} and \ref{E6 to C4:m}. 
\end{proof}

From the formulas of $\rank A_{F}$ and $\defe F$, we know that the subgroups in the same 
family with different parameters are non-conjugate, the subgroups in different families 
are clearly non-conjugate, so these subgroups are non-conjugate to each other. In total, 
we have $3\times 4+3\times 4+3\times 6+3\times 3=51$ conjugacy classes.

\begin{prop}\label{uniqueness:E6}
For two elementary abelian 2-subgroups $F,F'\subset G$, if an automorphism 
$f: F\longrightarrow F'$ has the property $f(x)\sim x$ for any $x\in F$, then $f=\Ad(g)$
for some $g\in G$.
\end{prop}

\begin{proof}
We may assume that $F=F'$ and they are one of $F_{r,s},F'_{r,s},F_{\epsilon,\delta,r,s},
F_{\epsilon,\delta,r,s}$.

When $F=F'=F_{r,s}$, we may assume $f(\sigma_3)=\sigma_3$, then 
$F\cap G_0=F'\cap G_0\subset G_0^{\sigma_3}=F_4$. By the proof of Lemma \ref{F 4 r,s}, 
we get there exists $g\in G_0^{\sigma_3}$ such that $f=\Ad(g)$.

When $F=F'=F_{r,s}$, similar as the proof of Lemma \ref{F 4 r,s}, 
we can show there exists $g\in G_0$ such that $f=\Ad(g)$.

When $F=F'=F_{\epsilon,\delta,r,s}$, we may assume $f(\sigma_4)=\sigma_4$, then 
$F\cap G_0=F'\cap G_0\subset G_0^{\sigma_4}=Sp(4)/\langle-I\rangle$ and non-identity 
elements of $F\cap G_0=F'\cap G_0$ all conjugate to $\textbf{i}I,[I_{2,2}]$ 
in $Sp(4)/\langle-I\rangle$. Then $f=\Ad(g)$ for some $g\in G_0^{\sigma_4}$ by 
Proposition \ref{Sub F1-2}.

When $F=F'=F'_{\epsilon,\delta,r,s}$, since 
$F'_{\epsilon,\delta,r,s}\subset G_0^{\sigma_4}=Sp(4)/\langle-I\rangle$, then 
$f=\Ad(g)$ for some $g\in G_0^{\sigma_4}$ by Proposition \ref{Sub F1-2}. 
\end{proof}

\begin{prop}
We have the following description for the Automizer groups, 
\begin{itemize}
\item[1.]{$r\leq 2$, $s\leq 3$, $W(F_{r,s})\cong (\bbF_2)^{r}\rtimes P(r,s,\bbF_2)$;}

\item[2.]{$r\leq 2$, $s\leq 3$, $W(F'_{r,s})\cong P(r,s,\bbF_2)$;}

\item[3.]{$\epsilon+\delta\leq 1$, $r+s\leq 2$, \[W(\bbF_{\epsilon,\delta,r,s})\cong 
F_2^{r+2s+\epsilon+2\delta}\rtimes\big(\Hom(\bbF_2^{\epsilon+2\delta+2s},\bbF_2^{r})
\rtimes(\GL(r,\bbF_2)\times \Sp(s;\epsilon,\delta))\big);\]}

\item[4.]{$\epsilon+\delta\leq 1$, $r+s\leq 2$, $s\geq 1$, \[W(F'_{\epsilon,\delta,r,s})
\cong \Hom(\bbF_2^{\epsilon+2\delta+2s},\bbF_2^{r})
\rtimes(\GL(r,\bbF_2)\times \Sp(s;\epsilon,\delta)).\]}
\end{itemize}
\end{prop}
\begin{proof}
The action of any $w\in W(F)$ preserves $\mu$ and $m$ on $F\cap G_0$, and conjugacy classes 
of elements in $F-G_0$. By Proposition \ref{uniqueness:E6}, an automorphism of $F$ 
preserves these data is actually the action of some $w\in W(F)$ on $F$. 

Then by Lemmas \ref{E6 to F4:m} and \ref{E6 to C4:m}, we get these Automizer groups.  
\end{proof}

\section{$\E_7$}\label{E7}

Let $G=\Aut(\fre_7)$, there exists three conjugacy classes of involutions in $G$ 
with representatives $\sigma_1,\sigma_2,\sigma_3$, we have   
\begin{eqnarray*}&&G^{\sigma_1}\cong (\Spin(12)\times\Sp(1))/\langle(c,1),(-c,-1)\rangle,
\\&& G^{\sigma_2}\cong((\E_{6}\times\U(1))/\langle(c',e^{\frac{2\pi i}{3}}))
\rtimes\langle\omega\rangle,\\&& G^{\sigma_3}\cong(\SU(8)/\langle iI\rangle)
\rtimes\langle\omega\rangle, \end{eqnarray*} where $c=e_1e_2\cdots e_{12}$, 
$1\neq c'\in Z_{E_6}$, $\omega^{2}=1$, and \[(\fre_6\oplus i\bbR)^{\omega}=\frf_4\oplus 0,  
\mathfrak{su}(8)^{\omega}\cong \mathfrak{sp}(4).\]

There exists 8 conjugacy classes of Klein four subgroups of $G$, their representatives and 
involution types are as follows, \begin{eqnarray*} &&\Gamma_1: (\sigma_1,\sigma_1,\sigma_1),\  
\Gamma_2: (\sigma_1,\sigma_1,\sigma_1),\ \Gamma_3: (\sigma_1,\sigma_2,\sigma_2),\ 
\Gamma_4: (\sigma_1,\sigma_2,\sigma_3),\\&&\Gamma_5: (\sigma_1,\sigma_3,\sigma_3),\ 
\Gamma_6:(\sigma_2,\sigma_2,\sigma_2),\ \Gamma_7: (\sigma_2,\sigma_3,\sigma_3),\ 
\Gamma_8: (\sigma_3,\sigma_3,\sigma_3).\end{eqnarray*}

For an elementary abelian 2-group $F\subset G$, define 
\[H_{F}=\{1\}\cup\{x\in F|x\sim\sigma_1\}.\]
Define $m: H_{F}\times H_{F}\longrightarrow \{\pm{1}\}$ by 
$m(x,y)=-1$ if $\langle x,y\rangle\sim\Gamma_1$, and $m(x,y)=1$ otherwise.

\begin{definition}
Define the translation subgroup 
\[A_{F}:=\{x\in H_{F}| \forall y\in F-H_{F},y\sim xy,\textrm{ and } 
\forall y\in H_{F}, m(x,y)=1\}\] and the defect index 
\[\defe(F)=|\{x\in F: x\sim\sigma_2\}|-|\{x\in F: x\sim\sigma_3\}|.\]
\end{definition} 

For $x\in F$ with $x\sim\sigma_2$, let $H_{x}:=\{y\in H_{F}|xy\sim\sigma_2\}$, 
which is not always a subgroup. 

\begin{lemma}
$H_{F}$ is a subgroup of $F$ and $\rank(F/H_{F})\leq 2$.
\end{lemma}

\begin{proof}
Since the product of any two distinct elements in $F$ conjugate to 
$\sigma_1$ is also conjugate to $\sigma_1$, so $H_{F}$ is a subgroup.

Suppose $\rank(F/H_{F})\geq 3$, then there exists a rank 3 subgroup 
$F'\subset F$ with $H_{F'}=1$. For any $1\neq x\in F'$, $G^{x}\cong G^{\sigma_2}$ 
or $G^{\sigma_3}$ has only two connected components, so 
$\rank(G^{x}_0\cap F')\geq 2$. Choose $y\in G^{x}_0-\langle x\rangle$, then 
$\langle x,y\rangle$ is a toral Klein four subgroup of $G$. Then at least 
one of $x,y,xy$ is conjugate to $\sigma_1$, which contradicts to $H_{F'}=1$.
\end{proof}

\subsection{Subgroups from $\E_6$}

In \[G^{\sigma_2}\cong\big((\E_{6}\times\U(1))/\langle(c,e^{\frac{2\pi i}{3}})\big)
\rtimes\langle\omega\rangle,\] let $\tau_1,\tau_2\in\E_6$ be involutions with 
\[\fre_6^{\tau_1}\cong \mathfrak{su}(6)\oplus \mathfrak{sp}(1),\ 
\fre_6^{\tau_2}\cong\mathfrak{so}(10)\oplus i\bbR.\] 
Let $\eta_1, \eta_2\in\E_6^{\omega}\cong\F_4$ be involutions with 
\[\frf_4^{\eta_1}\cong \mathfrak{sp}(3)\oplus\mathfrak{sp}(1), 
\frf_4^{\eta_2}\cong \mathfrak{so}(9),\] and let $\tau_3=\omega$, $\tau_4=\eta_1\omega$. 
Then $\tau_1,\tau_2,\sigma_2\tau_1,\sigma_2\tau_2,\tau_3,\tau_4$ represent all conjugacy 
classes of involutions in $G^{\sigma_2}$ except $\sigma_2$. In $G$ we have conjugacy relations 
\[\tau_1 \sim\tau_2\sim\sigma_1, \sigma_2\tau_1\sim \sigma_3, 
\sigma_2\tau_2\sim\sigma_2,\] \[\tau_3\sim\sigma_2\tau_3\sim\sigma_2, 
\tau_4\sim\sigma_2\tau_4\sim\sigma_3.\]
Moreover, in $G^{\sigma_2}$, we have conjugacy relations 
\[\eta_1\sim_{\E_6}\tau_1,\eta_2\sim_{\E_6}\tau_2,\eta_2\omega\sim_{\E_6}\omega.\]  

Recall that, we choose $\sigma_2=\exp(\pi i \frac{H'_2+H'_5+H'_7}{2})$, 
$\frg^{\sigma_2}$ has a simple root system 
\[\{\alpha_5+\alpha_6,\alpha_1,\alpha_2+\alpha_4,\alpha_3,\alpha_4+\alpha_5,\alpha_6+\alpha_7\}(\textrm{Type } \E_6)\]
we can choose $\tau_1=\exp(\pi i H'_1)$, $\tau_2=\exp(\pi i (H'_5+H'_7))$, then the conjugacy relations 
$\tau_1\sim\tau_2\sim\sigma_1$, $\sigma_2\tau_1\sim \sigma_3$, $\sigma_2\tau_2\sim\sigma_2$ 
are clear.  

Moreover, we have 
\[\omega=\exp(\frac{\pi (X_{\alpha_2}-X_{-\alpha_2})}{2}) \exp(\frac{\pi 
(X_{\alpha_5}-X_{-\alpha_5})}{2}) \exp(\frac{\pi (X_{\alpha_7}-X_{-\alpha_7})}{2})\] and we can choose 
$\eta_1=\exp(\pi i H'_1)$, then we get the conjugacy relations $\tau_3\sim\sigma_2\tau_3\sim\sigma_2$  
and $\tau_4\sim\sigma_2\tau_4\sim\sigma_3$. 

The conjugacy relations relations $\eta_1\sim_{E_6}\tau_1$, $\eta_2\sim_{E_6}\tau_2$, 
$\eta_2\omega\sim_{E_6}\omega$ are obtained from consideration in the group 
$E_6\rtimes\langle\omega\rangle$, similar as the study of conjugacy classes in 
$\Aut(\fre_6)^{\sigma_3}$ in $\E_6$ case, the elements $\tau_1,\tau_2,\omega,\eta_1\omega,\eta_1,\eta_2$ 
correspond to $\sigma_1,\sigma_2,\sigma_3,\sigma_4,\tau_1,\tau_2$ there.  


Let $L_1,L_2,L_3,L_4$ be Klein four subgroups of $\E_6$ 
of involution types $(\tau_1,\tau_1,\tau_1)$, $(\tau_1,\tau_1,\tau_2)$, 
$(\tau_1,\tau_2,\tau_2)$, $(\tau_2,\tau_2,\tau_2)$ respectively, then 
\begin{eqnarray*}&&(\fre_7^{\sigma_2})^{L_1}\cong\mathfrak{su}(3)^{2}\oplus(i\bbR)^{3},\  
(\fre_7^{\sigma_2})^{L_2}\cong\mathfrak{su}(4)\oplus\mathfrak{su}(2)^{2}\oplus(i\bbR)^{2},
\\&&(\fre_7^{\sigma_2})^{L_3}\cong\mathfrak{su}(5)\oplus(i\bbR)^{3},\ 
(\fre_7^{\sigma_2})^{L_4}\cong\mathfrak{so}(8)\oplus(i\bbR)^{3}.\end{eqnarray*}

\begin{lemma}\label{L to F}
In $G$, $L_1\sim L_3\sim F_1$ and $L_2\sim L_4\sim F_2$.
\end{lemma}

\begin{proof}
Since $\mathfrak{su}(3)^{2}\oplus(i\bbR)^{3}, \mathfrak{su}(5)\oplus(i\bbR)^{3}$ 
are not symmetric subalgebras of $\fre_7^{F_2}\cong\mathfrak{so}(8)\oplus
\mathfrak{su}(2)^{3}$ and $\mathfrak{su}(4)\oplus\mathfrak{su}(2)^{2}
\oplus(i\bbR)^{2},\mathfrak{so}(8)\oplus(i\bbR)^{3}$ are not symmetric 
subalgebras of $\fre_7^{F_1}\cong\mathfrak{su}(6)\oplus(i\bbR)^{2}$, so $L_1,L_3$ 
are conjugate to $F_1$ in $G$ and $L_2,L_4$ are conjugate to $F_2$ in $G$.
\end{proof}

Let $F\subset G=\Aut(\fre_7)$ be an elementary abelian 2-subgroup containing an element 
conjugate to $\sigma_2$, we may assume that $\sigma_2\in F$, then 
\[F\subset G^{\sigma_2}\cong((\E_{6}\times\U(1))/\langle(c,e^{\frac{2\pi i}{3}}))
\rtimes\langle\omega\rangle,\] where $c$ is a non-trivial central element of 
$\E_6$, $c^{3}=1$, $\omega^{2}=1$ and $(\fre_6\oplus i\bbR)^{\omega}=\frf_4\oplus 0$. 
Let $G_{\sigma_2}=(\E_6\times 1)\ltimes\langle\omega\rangle$ be the subgroup 
generated by $\E_6$ ($=\E_6\times 1$) and $\omega$. This definition of $G_{\sigma_2}$ is 
not quite canonical, another choice is to define it as 
$(\E_6\times 1)\ltimes\langle\sigma_2\omega\rangle$, but these are conjugate since 
\begin{eqnarray*}&&(1,i)\omega(1,i)^{-1}=\omega(\omega^{-1}(1,i)\omega)(1,i)^{-1}
\\&&=\omega(1,-i)(1,-i)=\omega(1,-1)\\&&=\omega\sigma_2.\end{eqnarray*}
and so they are equivalent, 

\begin{lemma}\label{E7 to E6:m}
For an elementary abelian 2-group $F\subset G=\Aut(\fre_7)$ contains $\sigma_2$, 
in the inclusion $F\subset G^{\sigma_2}\cong((\E_{6}\times\U(1))
/\langle(c,e^{\frac{2\pi i}{3}}))\rtimes\langle\omega\rangle$, we have $H_{F}=F\cap \E_6$. 

Moreover, the map $m: H_{F}\times H_{F}\longrightarrow\{\pm{1}\}$ is equal to the the 
similar map when $H_{F}$ is viewed as a subgroup of $\E_6$ (or 
$\E_6/\langle c\rangle=\Int(\fre_6)$). 
\end{lemma}
\begin{proof}
$H_{F}=F\cap \E_6$ follows from the comparison of conjugacy classes of involutions in
$G^{\sigma_2}$ and in $G$. The two maps $m$ are equal follows from Lemma \ref{L to F}. 
\end{proof}

We have a 3-fold covering $\pi: G_{\sigma_2}\longrightarrow \Aut(\fre_6)$ 
and an inclusion $p: G_{\sigma_2}\longrightarrow G^{\sigma_2}$. For any 
elementary abelian 2-subgroup $K$ of $\Aut(\fre_6)$, 
$p(\pi^{-1}K)\times\langle\sigma_2\rangle$ is the direct product of 
its unique Sylow 2 subgroup $F$ and $\langle(c,1)\rangle$. Let 
\begin{eqnarray*}&&\{F_{r,s}:r\leq 2,s\leq 3\}, \{F'_{r,s}:r\leq 2,s\leq 3\},
\\&& \{F_{\epsilon,\delta,r,s}:\epsilon+\delta\leq 1, r+s\leq 2\},
\\&&\{F'_{\epsilon,\delta,r,s}:\epsilon+\delta\leq 1, r+s\leq 2, s\geq 1\}
\end{eqnarray*} be elementary abelian 2 subgroups of $G^{\sigma_2}\subset G$ 
obtained from the subgroups of $\Aut(\mathfrak{e}_6)$ with the 
corresponding notation in this way.

\begin{prop}\label{E7 to E6 II}
Any elementary abelian 2-subgroup of $G$ with an element conjugate to 
$\sigma_2$ is conjugate to one of $F_{r,s}, F'_{r,s}$, 
$F_{\epsilon,\delta,r,s}$, $F'_{\epsilon,\delta,r,s}$.
\end{prop}

\begin{proof}
We may assume that $\sigma_2\in F$, then $F\subset G^{\sigma_2}\cong
((\E_{6}\times U(1))/\langle(c,e^{\frac{2\pi i}{3}}))
\rtimes\langle\omega\rangle$. 

When $\rank(F/H_{F})=2$, we may assume that $\omega\in F$ or 
$\tau_4=\eta_1\omega\in F$, then $F\subset(\E_6\rtimes\langle\omega\rangle)
\times\langle \sigma_2\rangle$. When $\rank(F/H_{F})=1$, we have 
$F\subset E_6\times\langle\sigma_2\rangle$. 
The conclusion follows from the classification in $\E_6$ case 
(Propositions \ref{E6 to F4} and \ref{E6 to C4}). 
\end{proof}

\begin{prop}\label{E7 to E6}
The four families have the following characterization, 
so subgroups in different families are not conjugate to each other.
\begin{itemize}
\item[(1)]{$F$ is conjugate to some $F_{r,s}$ if and only if $F$ 
contains a subgroup conjugate to $\Gamma_6$;}
\item[(2)]{$F$ is conjugate to some $F'_{r,s}$ if and only if $\rank (F/H_{F})=1$, 
$F$ contains an element $x$ conjugate to $\sigma_2$ and $H_{x}$ is a subgroup;}
\item[(3)]{$F$ is conjugate to some  $F_{\epsilon,\delta,r,s}$ if and only if 
$\rank(F/H_{F})=2$, $F$ contains an element conjugate to $\sigma_2$ but contains 
no subgroups conjugate to $\Gamma_6$;}
\item[(4)]{$F$ is conjugate to some $F'_{\epsilon,\delta,r,s}$ if and only if 
$\rank(F/H_{F})=1$, $F$ contains an element conjugate to $\sigma_2$ and
$H_{x}$ is not a subgroup.}
\end{itemize}
\end{prop}
\begin{proof}
(1) and (3) are clear. (2) and (4) follow from the comparison of conjugacy classes of 
involutions in $G^{\sigma_2}$ and $G$ and the classification of elementary abelian 
2-subgroups of $\Int(\fre_6)$ (Propositions \ref{E6 to F4} and \ref{E6 to C4}). 
\end{proof}

We make a remark here, for subgroups in case (2) and (4), if one $H_{x}$ 
($x\in F$ with $x\sim\sigma_2$) is an subgroup, then any other $H_{x'}$ 
($x'\in F$ with $x'\sim\sigma_2$) is an subgroup (happens in Case (2)); 
conversely, if one $H_{x}$ ($x\in F$ with $x\sim\sigma_2$) is not an subgroup, 
then any other $H_{x'}$ ($x'\in F$ with $x'\sim\sigma_2$) is not an subgroup.

\begin{prop}\label{E7 rank and defect}
We have the following formulas for $\rank A_{F}$ and $\defe F$.
\begin{itemize}
\item[(1)]{For $F=F_{r,s}$, $r\leq 2$, $s\leq 3$, $\rank A_{F}=r$, 
$\defe F=3\cdot 2^{r}(2-2^{s})$;}
\item[(2)]{For $F=F'_{r,s}$, $r\leq 2$, $s\leq 3$, $\rank A_{F}=r$, 
$\defe F=2^{r}(2-2^{s})$;}
\item[(3)]{For $F=F_{\epsilon,\delta,r,s}$, $\epsilon+\delta\leq 1$, 
$r+s\leq 2$, $\rank A_{F}=r$, $\defe F=(1-\epsilon)(-1)^{\delta}2^{r+s+\delta}
-2^{1+r+\epsilon+2s+2\delta}$;}
\item[(4)]{For $F=F'_{\epsilon,\delta,r,s}$, $\epsilon+\delta\leq 1$, $r+s\leq 2$, 
$s\geq 1$, $\rank A_{F}=r$, $\defe F=(1-\epsilon)(-1)^{\delta} 2^{r+s+\delta}$.}
\end{itemize}
\end{prop}
\begin{proof}
These follows from Lemma \ref{E7 to E6:m} 
and Proposition \ref{E6 rank and defect}.  
\end{proof}

By Propositions \ref{E7 to E6} and \ref{E7 rank and defect}, any two of 
$\{F_{r,s}\}$, $\{F'_{r,s}\}$, $\{F_{\epsilon,\delta,r,s}\}$, 
$\{F'_{\epsilon,\delta,r,s}\}$ are non-conjugate.

\subsection{Subgroups from $\SU(8)$ or $\SO(8)$}\label{E7 to AD}
We have \[G^{\sigma_3}\cong(\SU(8)/\langle iI\rangle)\rtimes\langle\omega\rangle,\] 
$\omega^{2}=1$, $(\fru_0^{\sigma_3})^{\omega}=\mathfrak{sp}(4)$, and 
$\frp\cong\wedge^{4}(\bbC^{8})$. Let $\tau_1=[I_{2,6}]$ and  
$\tau_2=[I_{4,4}]$.  
              
Let $\omega_0=\omega\left(\begin{array}{cc}0&I_{4}\\-I_{4}&0\end{array}\right)$, then 
$\omega_0^{2}=1$ and $(\SU(8)/\langle iI\rangle)^{\omega_0}
=(\SO(8)/\langle -I\rangle)\times\langle\sigma_3\rangle$. In 
$(\SU(8)/\langle iI\rangle)_0^{\omega_0}=\SO(8)/\langle-I\rangle$, let 
\[\eta_1=\left(\begin{array}{cc}0&I_{4}\\-I_{4}&0\end{array}\right),\ 
\eta_2=\left(\begin{array}{cc}-I_{4}&\\&I_{4}\\\end{array}\right),\ 
\eta_3=\left(\begin{array}{cc}-I_{2}&\\&I_{6}\\\end{array}\right),\]
$\eta_4=\left(\begin{array}{cc}0&I_{1,3}\\-I_{1,3}&0\end{array}\right)$, where 
$I_{1,3}=\diag\{-1,1,1,1\}$. Let $\tau_3=\omega_0$, $\tau_4=\eta_1\omega_0$. 

Then $\tau_1,\tau_2,\sigma_3\tau_1,\sigma_3\tau_2,\tau_3,\tau_4,\sigma_3\tau_4$ 
represent all conjugacy classes of involutions in $G^{\sigma_3}$ except 
$\sigma_3=\frac{1+i}{\sqrt{2}}I$, and we have conjugacy relations in $G$, 
\[\tau_1\sim\tau_2\sim\sigma_1,\ \sigma_3\tau_1\sim\sigma_2,\ \sigma_3\tau_2\sim\sigma_3,\] 
\[\tau_3\sim\sigma_3,\tau_4\sim\sigma_2, \tau_4\sigma_3\sim\sigma_3.\] 
The conjugacy relations  $\tau_1\sim\tau_2\sim\sigma_1$, $\sigma_3\tau_1\sim\sigma_2$,  
$\sigma_3\tau_2\sim\sigma_3$ are obtained by calculating $\dim\fru_0^{\sigma}$ for all 
relevant $\sigma$. Since $\tau_4=\omega_0\eta_1=\omega$, so $\tau_4\sim\sigma_2$. 
Since $\tau_4\sigma_3=\omega\sigma_3=\omega\sigma_2\exp(\pi i H'_1)$, $\omega,\sigma_2$ 
are in the subgroup generated by roots $\alpha_2,\alpha_5,\alpha_7$, which are 
perpendicular to each other, and also perpendicular to $\alpha_1$, so 
\[\tau_4\sigma_3=\omega\sigma_2\exp(\pi i H'_1)\sim\sigma_2\exp(\pi i H'_1)=\sigma_3.\] 
Let $\phi$ be the action of $G^{\sigma_3}\cong(\SU(8)/\langle iI\rangle)
\rtimes\langle\omega\rangle$ on $\frp\cong\wedge^{4}(\bbC^{8})$, for any $X\in\SU(8)$, 
\[\phi(\omega_0)^{-1}\phi([X])\phi(\omega_0)=\phi(\omega^{-1}[X]\omega_0)=\phi([\overline{X}]),\] 
Then $\phi(\omega_0)$ maps the weight space of any weight $\lambda$ to the weight 
space of weight $-\lambda$. As $\wedge^{4}(\bbC^{8})$ has no zero-weight, so 
$\dim\frp^{\omega}=\frac{1}{2}\dim\frp=35$. Then $\dim\fru_0^{\omega_0}=28+35=63$, so 
$\tau_3=\omega_0\sim\sigma_3$. 

\smallskip 

Moreover, in $G^{\sigma_3}$, we have $\eta_3\sim_{G^{\sigma_3}}\tau_1$, 
\[\eta_1\sim_{G^{\sigma_3}}\eta_2\sim_{G^{\sigma_3}}\eta_4\sim_{G^{\sigma_3}}\tau_2,\]   
\[\eta_2\omega_0\sim_{G^{\sigma_3}}\eta_3\omega_0\sim_{G^{\sigma_3}}\omega_0=\tau_3,\] 
\[\eta_4\omega_0\sim_{G^{\sigma_3}}\eta_1\omega_0\sigma_3=\tau_4\sigma_2.\]

These conjugacy relations are obtained from consideration in the classical 
group $(\SU(8)/\langle iI\rangle)\rtimes\langle\omega_0\rangle$.  We show 
$\eta_4\omega_0\sim_{G^{\sigma_3}}\eta_1\omega_0\sigma_3$ here, which is the most complicated 
one among them. Let $y=e^{\frac{\pi i}{8}}\diag\{I_7,-1\}\in\SU(8)/\langle iI\rangle$, then 
\begin{eqnarray*}&&y(\eta_4\omega_0)y^{-1}=(y\eta_4y^{-1})\omega_0(\omega_0^{-1}y\omega_0)y^{-1}\\&=&
\eta_1\omega_0y^{-1}y^{-1}=\eta_1\omega_0e^{\frac{-\pi i}{4}}\\&=&\eta_1\omega_0\sigma_2,
\end{eqnarray*} in the last equality we use 
$e^{\frac{-\pi i}{4}}I=(e^{\frac{\pi i}{4}}I)(iI)^{-1}=e^{\frac{\pi i}{4}}I=\sigma_3$ in 
$\SU(8)/\langle iI\rangle$. 

Let \[M_1=\langle\left(\begin{array}{cc}-I_{4}&\\&I_{4} \\\end{array}\right), 
\left(\begin{array}{cc}0_{4}&I_{4} \\I_{4}&0_{4}\\ \end{array}\right)\rangle,\] 
\[M_2=\langle\diag\{-I_{4},I_{4}\},\diag\{-I_{2},I_{2},-I_{2},I_{2}\}\rangle,\] then 
$(\fre_7^{\sigma_3})^{M_1}\cong\mathfrak{su}(4)$ and 
$(\fre_7^{\sigma_3})^{M_2}\cong(\mathfrak{sp}(1))^{4}\oplus(i\bbR)^{3}$.

\begin{lemma}\label{M to F}
In $G$, $M_1\sim F_1$ and $M_2\sim F_2$.
\end{lemma}

\begin{proof}
Since $\mathfrak{su}(4)$ is not a symmetric subalgebra of 
$\fre_7^{F_2}\cong\mathfrak{so}(8)\oplus(\mathfrak{sp}(1))^{3}$ and
$(\mathfrak{sp}(1))^{4}\oplus(i\bbR)^{3}$ is not a symmetric subalgebra of 
$\fre_7^{F_1}\cong\mathfrak{su}(6)\oplus(i\bbR)^{2}$, so 
$M_1\sim F_1$, $M_2\sim F_2$.
\end{proof}

When $\sigma_3\in F$, $F\subset G^{\sigma_3}\cong(\SU(8)/\langle iI\rangle)
\rtimes\langle\omega_0\rangle$, where $\omega_0^{2}=1$ and 
$\mathfrak{su}(8)^{\omega_0}=\mathfrak{so}(8)$. If $F$ has no elements 
conjugate to $\sigma_2$, from the description of conjugacy classes of 
involutions in $(G^{\sigma_3})_0\cong\SU(8)/\langle iI\rangle$, we have 
$x\sim\tau_2=\diag\{-I_{4},I_{4}\}$ for any $1\neq x\in H_{F}$.

\begin{lemma}\label{E7 to AD:m}
For an elementary abelian 2-group $F\subset G=\Aut(\fre_7)$ contains $\sigma_3$ and 
without elements conjugate to $\sigma_2$, in the inclusion 
$F\subset G^{\sigma_3}\cong(\SU(8)/\langle iI\rangle)\rtimes\langle\omega_0\rangle$,
$H_{F}\subset\SU(8)/\langle iI\rangle$, and the homomorphism \[H_{F}\longrightarrow
\SU(8)/\langle iI,\sigma_3\rangle=\SU(8)/\langle\frac{1+i}{\sqrt{2}}I\rangle=\PSU(8)\]
is injective. Moreover, the map $m: H_{F}\times H_{F}\longrightarrow\{\pm{1}\}$ is equal to 
the similar map when $H_{F}$ is viewed as a subgroup of 
$\SU(8)/\langle\frac{1+i}{\sqrt{2}}I\rangle=\PSU(8)$.  
\end{lemma}
\begin{proof}
$H_{F}\subset\SU(8)/\langle iI\rangle$ since any involution in 
$\omega_0\SU(8)/\langle iI\rangle$ is conjugate to $\sigma_2$ or $\sigma_3$. For any 
involution $x\in F\cap(\SU(8)/\langle iI\rangle)$, exactly one of $x,x\sigma_3$ is 
conjugate to $\sigma_1$, so the map $H_{F}\longrightarrow\PSU(8)$ is injective. The two 
maps $m$ are equal follows from Lemma \ref{M to F}. 
\end{proof}


When $\rank(F/H_{F})=1$, $F\subset (G^{\sigma_3})_0\cong\SU(8)/\langle iI\rangle$. 
Let $y_{1}=\diag\{-I_{4},I_{4}\}$,  \[y_{2}=\diag\{-I_{2},I_{2},-I_{2},I_{2}\}, 
y_{3}=\diag\{-1,1,-1,1,-1,1,-1,1\},\] \[y_{4}=\left(\begin{array}{cc}0_{4}&I_{4}
\\I_{4}&0_{4}\\\end{array}\right),\ y_{5}=\left(\begin{array}{cccc}0_{2}&I_{2}&&
\\I_{2}&0_{2}&&\\&&0_{2}&I_{2}\\&&I_{2}&0_{2}\\\end{array}\right),\] 
\[y_{6}=\left(\begin{array}{cccccccc}0&1&&&&&&\\1&0&&&&&&\\&&0&1&&&&\\&&1&0&&&&\\
&&&&0&1&&\\&&&&1&0&&\\&&&&&&0&1\\&&&&&&1&0\end{array}\right).\]

For each $(r,s)$ with $r+s\leq 3$, let $F''_{r,s}=\langle \sigma_3,y_1,y_2,...,y_{r+s},
y_{4},...,y_{3+s}\rangle$. One has $A_{F''_{r,s}}=\ker m=\langle y_1,...,y_{r}\rangle$, 
so these $F''_{r,s}$ are not conjugate to each other.

When $\rank(F/H_{F})=2$, we may assume that $\omega_0\in F$, then \[F\subset
(G^{\sigma_3})^{\omega_0}=(\SO(8)/\langle -I\rangle)\times
\langle\sigma_3,\omega_0\rangle.\] 
Let $x_{1}=\diag\{-I_{4},I_{4}\}$, $x_{2}=
\diag\{-I_{2},I_{2},-I_{2},I_{2}\}$, \[x_{3}=\diag\{-1,1,-1,1,-1,1,-1,1\}.\] For each 
$r\leq 3$, let $F'_{r}=\langle\sigma_2,\omega_0,x_1,\cdots,x_{r}\rangle$.



\begin{prop}\label{prop:E7 to AD}
For an elementary abelian 2-group $F\subset G$, if $F$ contains an element conjugate 
to $\sigma_3$ but contains no elements conjugate to $\sigma_2$, then $F$ is conjugate 
to one of $\{F''_{r,s}: r+s\leq 3\}, \{F'_{r}: r\leq 3\}$. Any two groups in 
$\{F''_{r,s}\},\{F'_{r}\}$ are non-conjugate.
$\rank A_{F''_{r,s}}=\rank A_{F'_{r}}=r$.
\end{prop}

\begin{proof}
For the first statement, we may assume that $\sigma_3\in F$, 
then \[F\subset G^{\sigma_3}\cong (\SU(8)/\langle iI\rangle)
\rtimes\langle\omega_0\rangle.\] When $\rank(F/H_{F})=1$, 
$F\subset G^{\sigma_3}_0\cong\SU(8)/\langle iI\rangle$. 
As $F$ has no elements conjugate to $\sigma_2$, any element of $F$ is conjugate to 
$\tau_2$ or $\sigma_3\tau_2$ in $\SU(8)/\langle iI\rangle$, where
$\tau_2=\left(\begin{array}{cc}-I_{4}&\\&I_{4}\\\end{array}\right)$. 
Then $F\sim F''_{r,s}$ for some $r,s\geq 0$ with $r+s\leq 3$ by Subsection 
\ref{Subsection:A}.

When $\rank(F/H_{F})=2$, we may assume that $\omega_0\in F$ as well, then \[F\subset
(G^{\sigma_3})^{\omega_0}=(\SO(8)/\langle -I\rangle)\times
\langle\sigma_3,\omega_0\rangle.\] We have $H_{F}=F\cap\SO(8)/\langle -I\rangle$ and 
any involution in $F$ is conjugate to $\eta_2=\diag\{-I_4,I_4\}$ in 
$\SO(8)/\langle-I\rangle$. Then $F\sim F'_{r}$ for some $r\leq 3$ by Subsection 
\ref{Subsection:a subclass}.

By Lemma \ref{E7 to AD:m}, we get $\rank A_{F''_{r,s}}=\rank A_{F'_{r}}=r$. 
By this, we get any two groups in $\{F''_{r,s}\},\{F'_{r}\}$ are non-conjugate.
\end{proof}

\subsection{Pure $\sigma_1$ subgroups}

A subgroup $F$ of $\Aut(\fre_7)$ is called a {\it pure $\sigma_1$ subgroup} if all of 
its non-trivial elements are conjugate to $\sigma_1$ in  $\Aut(\fre_7)$. 

For $\sigma_1$, $G^{\sigma_1}\cong (\Spin(12)\times\Sp(1))/\langle(c,1),(-c,-1)\rangle$ 
where $c=e_1e_2...e_{12}$, and $\frp=M_{+}\otimes\bbC^{2}$. 
$(e_1e_2e_3e_4,1)$, $(e_1e_2,\textbf{i})$, $(e_1e_2e_3e_4e_5e_6,\textbf{i})$, 
$(\Pi,1)$, $(\Pi,-1)$ represent the conjugacy classes of involutions in $G^{\sigma_1}$ 
except $\sigma_1=(1,-1)$, where \[\Pi= \frac{1+e_1e_2}{\sqrt{2}}
\frac{1+e_3e_4}{\sqrt{2}} \frac{1+e_5e_6}{\sqrt{2}}\frac{1+e_7e_8}{\sqrt{2}}
\frac{1+e_9e_{10}}{\sqrt{2}}\frac{1+e_{11}e_{12}}{\sqrt{2}}\in\Spin(12).\] 
And in $G$, we have  $(e_1e_2e_3e_4,1)\sim \sigma_1$, $(e_1\Pi e_1,i)\sim\sigma_1$, 
\[(e_1e_2,\textbf{i})\sim\sigma_2, (e_1e_2e_3e_4e_5e_6,\textbf{i}) \sim \sigma_3,\] 
\[(\Pi,1)\sim\sigma_2,\ (\Pi,-1)\sim\sigma_3.\] The identification 
of conjugacy classes is by calculating $\fru_0^{\sigma}$ from the consideration of 
weights of $M_{+}$ as an $\Spin(12)$ module. 


Let $K_1=\langle(e_1\Pi e_1^{-1},\textbf{i}),(e_1\Pi'e_1^{-1},\textbf{j})\rangle$, 
$K_2=\langle(e_1e_2e_3e_4,1),(e_5e_6e_7e_8,1)\rangle$, \begin{eqnarray*}&&
K_3=\langle(e_1\Pi e_1^{-1},\textbf{i}),(-e_1e_2e_3e_4,1)\rangle,\ 
K_4=\langle(e_1\Pi e_1^{-1}, \textbf{i}), (e_1e_2e_3e_4,1)\rangle,\\&& 
K_5=\langle(e_1e_2e_3e_4,1),(e_1e_2e_5e_6,1)\rangle,\\\end{eqnarray*} where 
\[\Pi'=\frac{1+e_1e_3}{\sqrt{2}} \frac{1+e_4e_2}{\sqrt{2}} \frac{1+e_5e_7}{\sqrt{2}}
\frac{1+e_8e_6}{\sqrt{2}}\frac{1+e_9e_{11}}{\sqrt{2}}\frac{1+e_{12}e_{10}}{\sqrt{2}}.\]

\begin{lemma}
We have $\Pi^{2}=\Pi'^{2}=[\Pi,\Pi']=c$.    
\end{lemma}
\begin{proof}
$\Pi^{2}=\Pi'^{2}=c$ is clear. Calculation show \[\Pi\Pi'=\frac{1+e_1e_4}{\sqrt{2}}
\frac{1+e_2e_3}{\sqrt{2}}\frac{1+e_5e_8}{\sqrt{2}}\frac{1+e_6e_7}{\sqrt{2}}
\frac{1+e_9e_{12}}{\sqrt{2}}\frac{1+e_{10}e_{11}}{\sqrt{2}},\] so $(\Pi\Pi')^{2}=c$. Then 
$[\Pi,\Pi']=\Pi\Pi'\Pi^{-1}\Pi'^{-1}=\Pi\Pi'(c\Pi)(c\Pi')=(\Pi\Pi')^{2}=c$.    
\end{proof}

\begin{lemma}\label{K to F}
In $G$, $K_1\sim K_3\sim K_5\sim\Gamma_1$ and $K_2\sim K_4\sim\Gamma_2$.
\end{lemma}

\begin{proof}
Since $(\fru_0^{\sigma_1})^{K_1}\cong\mathfrak{sp}(3)$ is not a 
symmetric subgroup of $\fru_0^{\Gamma_2}\cong\mathfrak{so}(8)\oplus
(\mathfrak{sp}(1))^{2}$ and $(\fre_7^{\sigma_1})^{K_2}\cong(\mathfrak{sp}(1))^{7}$ 
is not a symmetric subgroup of $\fru_0^{\Gamma_1}\cong\mathfrak{su}(6)\oplus(i\bbR)^{2}$, 
we get $K_1\sim\Gamma_1, K_2\sim\Gamma_2$.

For a Cartan subalgebra $\frh_0$ of $\fre_7$, we may assume that $\sigma_1=
\exp(\pi i H'_2)$. Then $\frg^{\sigma_1}$ has a simple root system 
\[\{\alpha_2,\alpha_4,\alpha_5,\alpha_6,\beta,\alpha_7\}(\textrm{Type } \D_6)
\bigsqcup\{\alpha_1\},\] where $\beta=\alpha_1+2\alpha_3+2\alpha_4+\alpha_2+\alpha_5$.
By identifying conjugacy (classes of) elements in $\exp(\frh_0)$ and in 
$\Spin(12)\times\Sp(1)/\langle(c,1),(-c,-1)\rangle$, we get the conjugacy relations,  
\[(\exp(\pi i H'_1),\exp(\pi i H'_3),\exp(\pi i H'_2))\sim(\sigma_1,
(e_1 \Pi e_1^{-1},\textbf{i}), e_1e_2e_3e_4)\] and \[(\exp(\pi i H'_1),
\exp(\pi i H'_2),\exp(\pi i H'_4))\sim(\sigma_1,e_1e_2e_3e_4,e_1e_2e_5e_6).\] 
Then we have  $K_3\sim K_5\sim\Gamma_1$, $K_4\sim\Gamma_2$.
\end{proof}

\begin{lemma}\label{L:Spin(12)}
In $\Spin(12)$, $\Pi\sim\Pi^{-1}$, $\Pi\not\sim -\Pi$, and 
$\Pi\not\sim\pm{}e_1\Pi e_1^{-1}$. 
\end{lemma}
\begin{proof}
We have $(e_1e_3e_5e_7e_9e_{11})\Pi(e_1e_3e_5e_7e_9e_{11})^{-1}=\Pi^{-1}$, so 
$\Pi\sim\Pi^{-1}$. Since \[\SO(12)^{\pi(\Pi)}=\{g\in\Spin(12)|g\Pi g^{-1}=\pm{\Pi}\}
/\langle-1\rangle,\] where $-1\in\{g\in\Spin(12)|g\Pi g^{-1}=\Pi\}$, 
$\SO(12)^{\pi(\Pi)}=\U(6)$ is connected implies $\Pi\not\sim -\Pi$. 
We have $\pi(\Pi)=J_{6}\in\SO(12)$ and $\pi(\pm{}e_1\Pi e_1^{-1})=
I_{1,11}J_{6}I_{1,11}^{-1}$. Since $J_{6}\not\sim_{\SO(12)}I_{1,11}J_{6}I_{1,11}^{-1}$, 
so $\Pi\not\sim_{\Spin(12)}\pm{}e_1\Pi e_1^{-1}$. 
\end{proof}

\begin{lemma}\label{L:Gamma1}
We have $\Aut(\fre_7)^{\Gamma_1}=(\Aut(\fre_7)^{\Gamma_1})_0\rtimes\langle(e_1\Pi'e_1,\textbf{j})
\rangle$ and \[(\Aut(\fre_7)^{\Gamma_1})_0\cong(\SU(6)\times\U(1)\times\U(1))/\langle
(\omega I,\omega^{-1},1),(-I,1,1)\rangle.\] 
\end{lemma}

\begin{proof}
First we calculate $\Spin(12)^{\Pi}$.    
We have \[\SO(12)^{\pi(\Pi)}\cong\U(6)=(\SU(6)\times\U(1))/\langle\eta I,\eta^{-1}\rangle,\] 
$\eta=e^{\frac{\pi i}{3}}$. Then $\Spin(12)^{\Pi}=(\SU(6)\times A)/Z$, where 
\[A=\{\prod_{1\leq jleq 6}(\cos\theta+\sin\theta e_{2j-1}e_{2j}):\theta\in\bbR\}\cong\U(1)\] 
and $Z\subset Z(\SU(6))\times A$. The isomorphism $\U(1)\cong A$ maps $-1\in\U(1)$ 
to $c\in A$, and $\pi(c)=-I\in\SO(12)$, so 
\[\pi: \Spin(12)^{\Pi}\longrightarrow\SO(12)^{\pi(\Pi)}\] is an isomorphism when it is 
restricted to $\SU(12)$ or $A$. 

We show that $-c\in\SU(6)\subset\Spin(12)^{\Pi}$. For this, we first just look at the case of 
$n=4$, for $\Pi_0=\frac{1+e_1e_2}{\sqrt{2}}\frac{1+e_3e_4}{\sqrt{2}}\in\Spin(4)$ with 
$\Pi_{0}^{2}=c_0=e_1e_2e_3e_4$. We have an isomorphism $\Spin(4)\cong\Sp(1)\times\Sp(1)$, 
which maps $-1\in\Spin(4)$ to $(-1,-1)\in\Sp(1)\times\Sp(1)$, maps $c_0\in\Spin(4)$ to 
$(-1,1)\in\Sp(1)\times\Sp(1)$. Then $\Pi\in\Spin(4)$ is mapped to 
$(\mathbf{i},1)\in\Sp(1)\times\Sp(1)$ (or $(\mathbf{i},-1)\in\Sp(1)\times\Sp(1)$). 
As $(\Sp(1)\times\Sp(1))^{(\mathbf{i},\pm{1})}=\U(1)\times\Sp(1)$, so $(1,-1)$ is in 
the semisimple part $\Sp(1)$ of it. Then $-c_0\in\Spin(4)$ in the $\SU(2)$ part of 
$\Spin(4)^{\Pi}$. 

As $\Pi$ is in block form, so $-c\in\SU(6)\subset\Spin(12)^{\Pi}$ as well. Since 
$(-c)c=-1\neq 1\in\Spin(12)$, $\pi(-1)=1$, and $\pi$ is a 2-fold covering, so 
\[\Spin(12)^{\Pi}=(\SU(6)\times\U(1))/\langle(\omega I,\omega^{-1})\rangle\] when we 
identify $A=\U(1)$. By Lemma \ref{L:Spin(12)} (and Steinberg 's theorem), we get 
$\Aut(\fre_7)^{\Gamma_1}=\Aut(\fre_7)^{\Gamma_1})_0\langle(e_1\Pi'e_1,\textbf{j})\rangle$. 
The description $(\Aut(\fre_7)^{\Gamma_1})_0$ follows from the description of 
$\Spin(12)^{\Pi}$ as above. 
\end{proof}

In $G^{\sigma_1}\cong \Spin(12)\times\Sp(1)/\langle(c,1),(-c,-1)\rangle$, 
let $H_1=\langle\sigma_1,(e_1\Pi e_1^{-1},\textbf{i}),(e_1\Pi' e_1^{-1},\textbf{j})\rangle$, 
\[H_2=\langle\sigma_1,(e_1e_2e_3e_4,1),(e_5e_6e_7e_8,1)\rangle,\ 
H_3=\langle\sigma_1,(e_1\Pi e_1^{-1}, \textbf{i}),(e_1e_2e_3e_4,1)\rangle.\] 
Then any Klein four subgroup of $H_1$ is conjugate to $F_1$, any Klein four subgroup of 
$H_2$ is conjugate to $F_2$, a Klein four subgroup of $H_3$ is conjugate to $F_2$ 
if and only if it contains $(e_1e_2e_3e_4,1)$, otherwise it is conjugate to $F_1$.

\begin{lemma}\label{L:H1}
$G^{H_1}=(\Sp(3)/\langle -I\rangle)\times H_1$, and involutions $I_{1,2}, \textbf{i}I\in
\Sp(3)/\langle-I\rangle$ are conjugate to $\sigma_1,\sigma_2$ in $\Aut(\fre_7)$.  
\end{lemma}

\begin{proof}
$G^{H_1}=(\Sp(3)/\langle -I\rangle)\times H_1$ follows from Lemma \ref{L:Gamma1} and 
the fact \[\mathfrak{su}(6)^{e_1\Pi'e_1^{-1}}=\mathfrak{sp}(3).\] 

A little more calculation by following the chain $\Sp(3)\subset\SU(6)\subset\SO(12)$ 
shows, $I_{1,2},\textbf{i}I\in\Sp(3)$ are conjugate to 
$e_1e_2e_3e_4,\Pi$ in $\Spin(12)$. Then $I_{1,2}, \textbf{i}I$ are conjugate to 
$\sigma_1,\sigma_2$ in $\Aut(\fre_7)$.
\end{proof}

\begin{lemma}\label{Pure Sigma1 rank 3}
Any rank $3$ elementary abelian 2 pure $\sigma_1$ subgroup of $G$ is 
conjugate to one of $H_1,H_2,H_3$.
\end{lemma}
\begin{proof}
For a rank $3$ pure $\sigma_1$ elementary abelian $2$-subgroup $F$ of $G$, 
we may assume that $\sigma_1\in F$. Then $F\subset G^{\sigma_1}\cong
\Spin(12)\times\Sp(1)/\langle(c,1),(-c,-1)\rangle$ and any element of $F-\{1,\sigma_1\}$ 
is conjugate to $(e_1\Pi e_1^{-1},\textbf{i})$ or $(e_1e_2e_3e_4,1)$.

When any Klein four subgroup of $F$ is conjugate to $\Gamma_1$, we have $F\sim H_1$ by 
Lemma \ref{K to F}. When any Klein four subgroup of $F$ is conjugate to $\Gamma_2$, 
similarly we have $F\sim H_2$ by Lemma \ref{K to F}. For the remaining cases,  
it is clear that $F\sim H_3$.
\end{proof}

We have defined subgroups $\{F_{r,s}: r\leq 2,s\leq 3\}$ and $\{F''_{r,s}: r+s\leq 3\}$ 
in the last two subsections. $F_{r,s}$ contains a subgroup conjugate to 
$F_6$, $\rank(F''_{r,s}/H_{F''_{r,s}})=1$ and $F''_{r,s}$ doesn't contain 
any element conjugate to $\sigma_2$. For any $(r,s)$ with $r+s \leq 3$, let 
\[F'''_{r,s}=H_{F''_{r,s}}=\{1\}\cup\{x\in F''_{r,s}|x\sim\sigma_1\};\] for 
any $r\leq 2$, let \[F''_{r}=H_{F_{r,3}}=\{1\}\cup\{x\in F_{r,3}|x\sim \sigma_1\}.\]

\begin{lemma}\label{Lemma: pure sigma1}
Any pure $\sigma_1$ elementary abelian 2-group $F\subset G$ is conjugate to 
$F''_{r+3}$ for some $r\leq 2$ or $F'''_{r,s}$ for some $(r,s)$ with $r+s\leq 3$.
\end{lemma}

\begin{proof}
When $F$ contains a subgroup conjugate to $H_1$, we may assume that $H_1\subset F$, 
then \[F\subset G^{H_1}=(G^{\sigma_1})^{(e_1\Pi e_1,\textbf{i}),(e_1\Pi' e_1,\textbf{j})}
\cong(\Sp(3)/\langle-I\rangle)\times\langle\sigma_1,(e_1\Pi e_1,\textbf{i}),
(e_1\Pi' e_1,\textbf{j})\rangle.\] Since $F$ is pure $\sigma_1$, so any non-identity 
element of $F\cap(\Sp(3)/\langle-I\rangle)$ is conjugate to $I_{1,2}$ in 
$\Sp(3)/\langle-I\rangle$. Then $F\cap(\Sp(3)/\langle-I\rangle)$ is conjugate to 
a subgroup of $\langle I_{2,1},I_{1,2}\rangle$, which is a subgroup of 
$\langle\textbf{i}I,\textbf{j}I,I_{2,1},I_{1,2}\rangle$. Non-identity elements of 
$\langle\textbf{i}I,\textbf{j}I\rangle$ are all conjugate to $\sigma_2$ in $G$, 
so $F$ is conjugate to some $H_{F_{r,s}}=\{1\}\cup\{x\in F_{r,s}|x\sim\sigma_1\}$. 
Since $H_1\subset F$, we have $s=3$. Then $F$ is conjugate to $F''_{r}$.

If $F$ doesn't contain any subgroup conjugate to $H_1$ but contains a subgroup 
conjugate to $\Gamma_1$, we may assume that $\sigma_1,(e_1 \Pi e_1^{-1}, \textbf{i})\in F$. 
Since $F$ doesn't contain any subgroup conjugate to $H_1$, so 
\[F\subset (G^{\sigma_1})^{(e_1 \Pi e_1^{-1},\textbf{i})}_0\cong(\SU(6)\times\U(1)
\times\U(1))/\langle(e^{\frac{2\pi i}{3}}, e^{\frac{2\pi i}{3}},1),(-I,1,1)\rangle.\] 
Since $F$ is pure $\sigma_1$, we have $F=(F\cap\SU(6)/\langle-I\rangle)
\times\langle\sigma_1,(e_1 \Pi e_1^{-1}, \textbf{i})\rangle$ and any element in 
$F\cap\SU(6)/\langle-I\rangle$ is conjugate to $I_{2,4}$. Then $F$ is toral.
If $F$ doesn't contain any subgroup conjugate to $\Gamma_1$, then any 
Klein four subgroup of $F$ is conjugate to $\Gamma_2$. When $\rank(F)\geq 3$, 
we may assume that $H_2\subset F$. There are no elements 
$x\in (G^{\sigma_1})^{H_2}-H_2$ such that any Klein four subgroup of 
$\langle x,H_2\rangle$ is conjugate to $\Gamma_2$, so $\rank(F)\leq 3$. Then $F$ is 
conjugate to one of $1,\langle\sigma_1\rangle,\Gamma_2,H_2$, so $F$ is toral. 
For a toral and pure $\sigma_1$ subgroup $F$, there exists a 
Cartan subalgebra $\frh_0$ so that $F\subset\exp(\frh_0)$. Choose a 
Chevelley involution $\theta$ of $\fre_7$ with respect to $\frh_0$. 
Then $F'=\langle F,\theta\rangle$ satisfies $\Res(F'/H_{F'})=1$ and 
any involution in $F'-H_{F'}$ is conjugate to $\sigma_3$. Then $F'$ is 
conjugate to $F''_{r,s}$ for some $(r,s)$ with $r+s\leq 3$. 
Thus $F$ is conjugate to $F'''_{r,s}$.
\end{proof}

\begin{prop}\label{prop:pure sigma1}
For any $r+s\leq 3$, $\rank A_{F'''_{r,s}}=r$; for any $r\leq 2$, $\rank A_{F''_{r}}=r$.
Any two groups in $\{F'''_{r,s}: r+s\leq 3\},\{F''_{r}: r\leq 2\}$ are non-conjugate.
\end{prop}

\begin{proof}
By Propositions \ref{E7 rank and defect} and \ref{prop:E7 to AD}, we get 
$\rank A_{F'''_{r,s}}=r$ and $\rank A_{F''_{r}}=r$. Then any two groups in 
$\{F'''_{r,s}: r+s\leq 3\},\{F''_{r}: r\leq 2\}$ are non-conjugate.
\end{proof}

\subsection{Automizer groups and inclusion relations}

\begin{prop}\label{uniqueness E7}
For an isomorphism $f: F\longrightarrow F'$ between two elementary abelian 2-subgroups 
of $G$, if $f(x)\sim x$ for any $x\in F$ and $m_{F'}(f(x),f(y))=m_{F}(x,y)$ 
for any $x,y\in H_{F}$, then $f=\Ad(g)$ for some $g\in G$. 
\end{prop}

\begin{proof}
When $F$ contains an element conjugate to $\sigma_2$, we may assume that 
$\sigma_2\in F\cap F'$ and $f(\sigma_2)=\sigma_2$, then $F,F'\subset G^{\sigma_2}\cong
\langle\omega\rangle\rtimes((\E_{6}\times\U(1)) /\langle(c,e^{\frac{2\pi i}{3}}))$.
 From the description of conjugacy classes of elements in $G^{\sigma_2}$, we get 
$f(x)\sim_{G^{\sigma_2}}x$ for any $x\in F$ by the assumption in the proposition. 
Then $f=\Ad(g)$ for some $g\in G^{\sigma_2}$ by Proposition \ref{uniqueness:E6}. 

When $\rank(F/H_{F})=1$ and $F$ contains no elements conjugate to $\sigma_2$, 
we may assume that $\sigma_3\in F\cap F'$ and $f(\sigma_3)=\sigma_3$, then 
$F,F'\subset G^{\sigma_3}\cong\langle\omega_0\rangle\rtimes(\SU(8)/\langle iI\rangle)$ 
and any element in $(F\cup F')-\langle\sigma_3\rangle$ is conjugate to $I_{4,4}$.  
Since the function $m$ on $H_{F}\times H_{F}$ and $H_{F'}\times H_{F'}$ is identical 
to the anti-symmetric bilinear form when $H_{F},H_{F'}$ are seen as subgroups of $\PU(8)$. 
Then $f=\Ad(g)$ for some $g\in G^{\sigma_3}$ by Proposition \ref{Subclass: classification}. 

When $\rank(F/H_{F})=2$ and $F$ contains no elements conjugate to $\sigma_2$, 
we may assume that $\sigma_3,\omega_0\in F$, then 
$F,F'\subset(G^{\sigma_3})^{\omega_0}\cong\SO(8)/\langle-I\rangle$ and any element 
in $(F\cup F')-\langle\sigma_3,\omega_0\rangle$ is conjugate to $I_{4,4}$. Then $f=\Ad(g)$ 
for some $g\in (G^{\sigma_3})^{\omega_0}$ by Proposition \ref{Subclass: classification}.

When $F$ is pure $\sigma_1$, we get the conclusion by the considering the preserving 
of $m_{F},m_{F'}$ under $f$. 
\end{proof}

\begin{prop}
We have the following description for Automizer groups,
\begin{itemize}
\item[(1)]{$r\leq 2$, $s\leq 3$, $W(F_{r,s})\cong\Hom(\bbF_2^{2},\bbF_2^{r})\rtimes
(\GL(2,\bbF_2)\times P(r,s,\bbF_2))$;}

\item[(2)]{$r\leq 2$, $s\leq 3$, $W(F'_{r,s})\cong\bbF_2^{r}\rtimes P(r,s,\bbF_2)$;}

\item[(3)]{$\epsilon+\delta\leq 1$, $r+s\leq 2$, \[W(F_{\epsilon,\delta,r,s})=
(\bbF_2^{r+2s+\epsilon+2\delta+1}\rtimes\Hom(\bbF_2^{\epsilon+2\delta+2s+1},\bbF_2^{r}))
\rtimes(\GL(r,\bbF_2)\times\Sp(\delta+s;\epsilon)).\]}

\item[(4)]{$\epsilon+\delta\leq 1$, $r+s\leq 2$, \[W(F'_{\epsilon,\delta,r,s})=
\Hom(\bbF_2^{\epsilon+2\delta+2s+1},\bbF_2^{r})\rtimes(\GL(r,\bbF_2)\times
\Sp(\delta+s;\epsilon)).\]}

\item[(5)]{$r+s\leq 3$, $W(F''_{r,s})\cong(\bbF_2^{r+2s}\rtimes\Hom(\bbF_2^{2s},\bbF_2^{r}))
\rtimes(\GL(r,\bbF_2)\times\Sp(s))$;}

\item[(6)]{$r\leq 3$, $W(F'_{r})\cong\Hom(\bbF_2^{2},\bbF_2^{r})\rtimes(\GL(r,\bbF_2)
\times\GL(2,\bbF_2))$;}

\item[(7)]{$r+s\leq 3$, $W(F'''_{r,s})\cong\Hom(\bbF_2^{2s},\bbF_2^{r})\rtimes
(\GL(r,\bbF_2)\times\Sp(s))$;}

\item[(8)]{$r\leq 2$, $W(F''_{r})\cong P(r,3,\bbF_2)$.}
\end{itemize}
\end{prop}

\begin{proof}
By Proposition \ref{uniqueness E7}, we need to find automorphisms of $F$ preserves 
conjugacy classes of involutions and the form $m$ on $H_{F}$. 

We prove $(4)$. Let $F=F'_{\epsilon,\delta,r,s}$. $F$ has a decomposition $F=A_{F}\times F'$ with 
$A_{F}=\bbF_2^{r}$ the translation subgroup and $F'\sim F'_{\epsilon,\delta,0,s}$. 
By Proposition \ref{uniqueness E7}, we have \[W(F)\cong\Hom(F',A_{F})\rtimes
(\GL(r,\bbF_2)\times W(F')).\] So we only need to show in the case the $r=0$. Assume 
$r=0$ from now on.  

Any element in $W(F)$ preserves the symplectic form $m$ on $H_{F}$. Since  
$\rank\ker m=\epsilon$, so we have a homomorphism 
$p: W(F)\longrightarrow\Sp(\delta+s;\epsilon)$. 
We show this homomorphism is an isomorphism, which finishes the proof.

For any $f: F\longrightarrow F$ with $f|_{H_{F}}=1$, since 
$F=H_{F}\rtimes\langle z\rangle$ with $z\sim\sigma_2$, let $f(z)=zx_0$ for some 
$x_0\in H_{F}$. The for any $x\in H_{F}$, $f(zx)=zxx_0$, so $zx\sim zxx_0$. This 
just said $x_0\in A_{F}$. Since we  assume $r=0$ (equivalent to $A_{F}=1$), so $x_0=1$ 
and $f=\id$. Then $p$ is injective. 

By Proposition \ref{uniqueness E7}, $W(F)$ permutes transitively elements in $F$ 
conjugate to $\sigma_2$. There are \[\frac{2^{2\delta+2s+\epsilon}+(1-\epsilon)
(-1)^{\delta}2^{\epsilon+\delta+s}}{2}=2^{s+\delta-1}(2^{s+\delta+\epsilon}+
(1-\epsilon)(-1)^{\delta}2^{\epsilon})\] such elements. It is clear that the stabilizer 
of $W(F)$ at $z$ is $\Sp(s;0,\delta)$. So \[|W(F)|=|\Sp(s;\epsilon,\delta)|
2^{s+\delta-1}(2^{s+\delta+\epsilon}+(1-\epsilon)(-1)^{\delta}2^{\epsilon}).\] 
By Propositions \ref{Compare symplectic 1} and \ref{Compare symplectic 2}, 
this is also equal to $|\Sp(s+\delta;\epsilon)|$. Then $p$ is surjective. 

$(3)$ follows from $(4)$ immediately.

For the other case, we just use the fact ``$\rank A_{F}=r$'' and the form ``$m$ on 
$H_{F}/A_{F}$ is non-degenerate'' in these cases. 
\end{proof}

We still have the following containment relations, 
\begin{eqnarray*}&& F'''_{\epsilon+r,\delta+s}\subset F'_{\epsilon,\delta,r,s},\ 
F'''_{\epsilon+r,\delta+s}\subset F''_{\epsilon+r,\delta+s},\ 
F''_{\epsilon+r,\delta+s}\subset F_{\epsilon,\delta,r,s},\\&& 
F'_{r+s+\delta}\subset F_{\epsilon,\delta,r,s},\quad F''_{r+3}\subset F'_{r,3},
\end{eqnarray*} together with those obvious relations, they consist in 
all containment relations (in the sense of conjugacy) between these subgroups

\section{$\E_8$}\label{S:E8}

Let $G=\E_8=\Aut(\mathfrak{e}_8)$, which has two conjugacy classes of involutions 
(with representatives $\sigma_1,\sigma_2$), we have 
\[G^{\sigma_1}\cong(\E_7\times\Sp(1))/\langle(c,-1)\rangle,\  
G^{\sigma_2}\cong\Spin(16)/\langle c'\rangle,\] where $c$ is the unique 
non-trivial central element of $E_7$, $c'=e_1 e_2... e_{16}$. 

There are four conjugacy classes of Klein four subgroups in $G$ with 
involution types as \begin{eqnarray*}
&&\Gamma_1: (\sigma_1,\sigma_1,\sigma_1),\ \Gamma_2: (\sigma_1,\sigma_1,\sigma_2),\
\Gamma_3: (\sigma_1,\sigma_2,\sigma_2),\ \Gamma_4:(\sigma_2,\sigma_2,\sigma_2).
\end{eqnarray*}

In $G^{\sigma_1}\cong(\E_7\times\Sp(1))/\langle(c,-1)\rangle$, let 
$\eta_1,\eta_2\in\E_7$ be involutions such that there exists Klein four groups 
$F,F'\subset\E_7$ with non-identity elements all conjugate to $\eta_1$ or $\eta_2$ 
respectively and $\fre_7^{F}\cong \mathfrak{su}(6)\oplus (i\bbR)^2$,   
$\fre_7^{F'}\cong \mathfrak{so}(8)\oplus (\mathfrak{sp}(1))^{3}$, then 
$c\eta_1\sim_{E_7}\eta_2$, $c\eta_2\sim_{E_7}\eta_1$. 
Let $\tau_1=(\eta_1,1), \tau_2=(\eta_2,1)\in G^{\sigma_1}$, then 
$\tau_1\sim\sigma_1$ and $\tau_2\sim\sigma_2$. 

Let $\eta_3,\eta_4\in\E_7$ be involutions with $\eta_3^{2}=\eta_4^{2}=c$ and 
$\fre_7^{\eta_3}\cong\fre_6\oplus i\bbR$, $\fre_7^{\eta_4}\cong\mathfrak{su}(8)$. 
Then $c\eta_3\sim_{\E_7}\eta_3$, $c\eta_4\sim_{\E_7}\eta_4$. 
Let $\tau_3=(\eta_3,\textbf{i})$, $\tau_4=(\eta_4,\textbf{i})$, then 
$\tau_3\sim\sigma_1$ and $\tau_4\sim\sigma_2$.

Then $\tau_1,\tau_2,\tau_3,\tau_4$ represent all conjugacy classes of involutions in 
$G^{\sigma_1}$ except $\sigma_1$.

The relations $c\eta_1\sim_{\E_7}\eta_2$, $c\eta_2\sim_{\E_7}\eta_1$, 
$c\eta_3\sim_{\E_7}\eta_3$, $c\eta_4\sim_{\E_7}\eta_4$ are known from 
understanding of involutions in $\E_7$ and in $\E_7/\langle c\rangle\cong\Aut(\fre_7)$. 

Since $\sigma_1=\exp(\pi iH'_1)$, \[\{\alpha_8,\beta=\alpha_1+2\alpha_3+2\alpha_4+\alpha_2+\alpha_5,
\alpha_7,\alpha_6,\alpha_5,\alpha_4,\alpha_2\}(\textrm{Type }\E_7)\bigsqcup\{\alpha_1\}\] is a simple 
system of $\frg^{\sigma_1}$. We may choose $\eta_1,\eta_2,\eta_3,\eta_4$ such that 
$\tau_1=\exp(\pi i H'_8)$, $\tau_2=\exp(\pi i (H'_8+H'_2))$, $\tau_3=\exp(\pi i (H'_3+H'_4))$,  
$\tau_4=\exp(\pi i (H'_3+H'_4+H'_8))$, then 
\[\tau_1\sim\sigma_1, \tau_2\sim\sigma_2, \tau_3\sim\sigma_1, \tau_4\sim\sigma_2.\] 

For $\sigma_2$, $G^{\sigma_2}\cong\Spin(16)/\langle c\rangle$ and $\frp\cong M_{+}$, $\sigma_2=-1$.   
Let $\tau_1=\ e_1e_2e_3e_4$, $\tau_2=e_1e_2e_3...e_8$, $\tau_3=\Pi$, $\tau_4=-\Pi$, 
where \[\Pi=\frac{1+e_1e_2}{\sqrt{2}}\frac{1+e_3e_4}{\sqrt{2}}...\frac{1+e_{15}e_{16}} {\sqrt {2}}.\] 
Then $\tau_1,\tau_2,\tau_3,\tau_4$ represent all conjugacy classes of involutions in 
$G^{\sigma_1}$ except $\sigma_2$.

In $G$, we have $\tau_1\sim\tau_3\sim\sigma_1$, $\tau_2\sim\tau_4\sim\sigma_2$. These are obtained from 
calculating various $\dim\fru_0^{\sigma}$. 

And in $G^{\sigma_2}$, we have 
\[\sigma_2\tau_1\sim_{G^{\sigma_2}}\tau_1,\ \sigma_2\tau_2\sim_{G^{\sigma_2}}\tau_2,\]  
\[\sigma_2\tau_3\sim_{G^{\sigma_2}}\tau_4,\ \sigma_2\tau_4\sim_{G^{\sigma_2}}\tau_3.\]  
These are clear in $\Spin(16)/\langle c\rangle$. 

For any $x\in F$ with $x\sim \sigma_1$, let $H_{x}=\{y\in F|xy\not\sim y\}$.

\begin{definition}
Let \[H_{F}:=\langle\{H_{x}: x\in F, x\sim\sigma_1\}\rangle=
\langle\{x: x\in F,x\sim\sigma_1\}\rangle.\]
\end{definition}

\begin{lemma} \label{H subgroup in E8}
For any $x$ with $x\sim\sigma_1$, $H_{x}$ is a subgroup and $\rank(F/H_{x})\leq 2$.
\end{lemma}

\begin{proof}
We may assume that $x=\sigma_1$, then $F\subset G^{\sigma_1}\cong\E_7\times 
\Sp(1)/\langle(c,-1)\rangle$. For an element $y\in F\subset G^{\sigma_1}$ with $y^{2}=1$, 
$\sigma_1 y\not\sim y$ if and only if $y$ is conjugate to $1,\sigma_1,\tau_1,\tau_2$ in 
$G^{\sigma_1}$. Then it is also equivalent to $y\in\E_7\subset G^{\sigma_1}$. So 
$H_{x}=F \cap\E_7$ and it is a subgroup. Then $F/H_{x}\subset G^{\sigma_1}/\E_7\cong\Sp(1)/
\langle-1\rangle$, so $\rank(F/H_{x})\leq 2$.
\end{proof}

For any $x\in F$, define $\mu(x)=1$ if $x\sim\sigma_2$ or $x=1$; $\mu(x)=-1$ if 
$x\sim\sigma_1$. For any $x,y\in F$, define $m(x,y)=\mu(x)\mu(y)\mu(xy)$. 
In general $m$ is not a bilinear form.

Define the {\it translation subgroup} \[A_{F}=\{x\in F|\mu(x)=1 \textrm{ and }m(x,y)=1
\textrm{ for any } y\in F\}\] and the {\it defect index} 
$\defe(F)=|\{x\in F:\mu(x)=1\}|-|\{x\in F:\mu(x)=-1\}|$. 

\begin{definition}
We call $\Res(F):=\rank(F/H_{F})$ the residual rank of $F$, and 
\[\Res'(F)=max\{\rank(F/H_{x})|x\in F, x\sim\sigma_1\}\] the second residual rank of $F$.
\end{definition}

Let $X=X_{F}=\{x\in F| x\sim\sigma_1\}$, define a graph with vertices set $X$ by drawing an 
edge connecting $x,y\in X$ if and only if $xy\sim\sigma_2$. It is clear that this graph 
$X$ is invariant under multiplication by elements in $A_{F}$. 
Let \[\Graph(F)=X_{F}/A_{F}\] be the quotient graph of the graph $X_{F}$ modulo the action of $A_{F}$.

\subsection{Subgroups from $\E_6$}
For an elementary abelian 2-group $F\subset G$, if $F$ contains a subgroup conjugate 
to $\Gamma_1$, we may assume that $\Gamma_1=\langle\sigma_1,\tau_3\rangle\subset F$, then 
\[F\subset G^{\Gamma_1}=((\E_6\times\U(1)\times\U(1))/\langle(c,e^{\frac{2\pi i}{3}},1)\rangle)
\rtimes\langle\omega\rangle,\] where $\omega^{2}=1$, 
$(\fre_6\oplus i\bbR\oplus i\bbR)^{\omega}=\frf_4\oplus 0\oplus 0$ and 
$\Gamma_1=\langle(1,-1,1),(1,1,-1)\rangle$. 

Let $G_{\Gamma_1}=\E_6\rtimes\langle\omega\rangle\subset G^{\Gamma_1}$. We have a $3$-fold 
covering $\pi: G_{\Gamma_1}\longrightarrow\Aut(\fre_6)$ and an inclusion 
$p: G_{\Gamma_1}\subset G^{\Gamma_1}$. For an elementary abelian 2-subgroup $F'$ of 
$\Aut(\fre_6)$, $p(\pi^{-1}(F'))\times\Gamma_1$ is the direct product of its (unique) 
Sylow 2-subgroup $F$ and $\langle(c,1,1)\rangle$. 
Let $\{F_{r,s}:r\leq 2,s\leq 3\}$, $\{F'_{r,s}:r\leq 2,s\leq 3\}$, 
$\{F_{\epsilon,\delta,r,s}:\epsilon+\delta\leq 1,r+s\leq 2\}$, 
$\{F'_{\epsilon,\delta,r,s}:\epsilon+\delta\leq 1,r+s\leq 2,s\geq 1\}$ 
be elementary abelian 2-subgroups of $\E_8$ obtained from elementary abelian 2-subgroups 
of $\Aut(\fre_6)$ with the corresponding notation in this way.

Let $\theta_1,\theta_2\in\E_6$ be involutions with 
$\fre_6^{\theta_1}\cong\mathfrak{su}(6)\oplus\mathfrak{sp}(1)$, 
$\fre_6^{\theta_2}\cong \mathfrak{so}(10)\oplus i\bbR$, and let 
$\theta_3=\omega,\theta_4\in\omega E_6$ be involutions with 
$\fre_6^{\theta_3}\cong\frf_4\oplus0\oplus0$, 
$\fre_6^{\theta_4}\cong\mathfrak{sp}(4)\oplus0\oplus0$. 
From the description of involution classes in $G^{\sigma_1}$, we have 
$\theta_1\sim\theta_3\sim\sigma_1$, $\theta_2\sim\theta_4\sim\sigma_2$ and 
\[\forall \sigma\in F_1-\{1\},\ \theta_1\sigma\sim\theta_4\sigma\sim\sigma_2,\ 
\theta_2\sigma\sim\theta_3\sigma\sim\sigma_1.\]

Noe that, $F'_{r,3}$ contains a rank 3 pure $\sigma_1$ subgroup, actually it is conjugate 
to $F_{r,2}$.


\begin{prop}
For an elementary abelian 2-group $F\subset\E_8$, if $F$ contains a subgroup 
conjugate to $\Gamma_1$, then $F$ is conjugate to one of 
$\{F_{r,s}:r\leq 2,s\leq 3\}$, $\{F'_{r,s}:r\leq 2,s\leq 2\}$, 
$\{F_{\epsilon,\delta,r,s}:\epsilon+\delta\leq 1,r+s\leq 2\}$, 
$\{F'_{\epsilon,\delta,r,s}:\epsilon+\delta\leq 1,r+s\leq 2,s\geq 1\}$.
\end{prop}

\begin{proof}
The proof is similar as that for Proposition \ref{E7 to E6 II}. 
\end{proof}

\begin{prop}\label{E8 rank and defect}
We have the following formulas for $\Res F,\Res' F,\rank A_{F}, \defe F$, 
\begin{itemize}
\item[1.]{For $F=F_{r,s}$, $r\leq 2$, $s\leq 3$, $(\Res F,\Res' F)=(0,2)$, 
$\rank A_{F}=r$, $\defe F=3\cdot 2^{r+1}(2^{s}-2)$;}
\item[2.]{For $F=F'_{r,s}$, $r\leq 2$, $s\leq 2$, $(\Res F,\Res' F)=(0,1)$, 
$\rank A_{F}=r$, $\defe F=2^{r+1}(2^{s}-2)$;}
\item[3.]{For $F=F_{\epsilon,\delta,r,s}$, $\epsilon+\delta\leq 1$, 
$r+s\leq 2$, $(\Res F,\Res'F)=(1,2)$, $\rank A_{F}=r$, 
$\defe F=(1-\epsilon)(-1)^{\delta+1}2^{r+s+\delta+1}+2^{\epsilon+r+2\delta+2s}$;}
\item[4.]{For $F=F'_{\epsilon,\delta,r,s}$, $\epsilon+\delta\leq 1$, 
$r+s\leq 2$, $s\geq 1$, $(\Res F,\Res'F)=(0,1)$, $\rank A_{F}=r$, 
$\defe F= (1-\epsilon)(-1)^{\delta+1} 2^{r+s+\delta+1}$.}
\end{itemize}

The groups $\{F_{r,s}:r\leq 2,s\leq 3\}$, $\{F'_{r,s}:r\leq 2,s\leq 2\}$, 
$\{F_{\epsilon,\delta,r,s}:\epsilon+\delta\leq 1,
r+s\leq 2\}$, $\{F'_{\epsilon,\delta,r,s}:\epsilon+\delta\leq 1,
r+s\leq 2,s\geq 1\}$ are pairwise non-conjugate.
\end{prop}

\begin{proof}
The formulas in the first statement follow from the construction of these 
subgroups.

For the second statement, $(\Res F,\Res' F)$ distinguish most of them except for 
some possible pairs $(F'_{r',s'},F'_{\epsilon,\delta,r,s})$. Suppose some pair 
$(F'_{r',s'},F'_{\epsilon,\delta,r,s})$ is conjugate. From the formulas in the 
first statement, we have $r'=r$ (by $A_{F}$), $s'-1=(1-\epsilon)(-1)^{\delta}$ 
(by the sign of $\defe F$) and $s'=2s+2\delta+\epsilon$ (by $\rank F/A_{F}$). 
Since $s'\leq 2$ and $s\geq 1$, the last equality implies $\epsilon=\delta=0$, 
$s=1$ and $s'=2$. Then the second equality implies $s'=1$. We have a contradiction. 
\end{proof}

\subsection{Other subgroups}

In $G^{\sigma_1}\cong(\E_7\times\Sp(1))/\langle(c,-1)\rangle$, choose 
$x_1,x_2\in\E_7$ with $x_1\sim x_2\sim x_1x_2\sim\tau_4$, then 
$(G^{\sigma_1})^{x_1,x_2}=\SO(8)/\langle-I\rangle\times\langle\sigma_1,x_1,x_2\rangle$. 
Let $z_1=\diag\{-I_4,I_4\}$, $z_2=\diag\{-I_2,I_2,-I_2,I_2\}$, 
$z_3=\diag\{-1,1,-1,1,-1,1,-1,1\}$. Define  
\[F''_{r,s}=\langle z_1,\cdots,z_{r},\sigma_1,x_1,\cdots,x_{s}\rangle\] 
for any $r\leq 3$, $s\leq 2$, 

In $G^{\sigma_2}\cong\Spin(16)/\langle c\rangle$, $c=e_1e_2\cdots e_{16}$, let 
$y_1=\sigma_1=-1$, \[y_2=e_1e_2e_3e_4e_5e_6e_7e_8, 
y_3=e_1e_2e_3e_4e_{9}e_{10}e_{11}e_{12},\]  
\[y_4=e_1e_2e_5e_6e_9e_{10}e_{13}e_{14}, y_5=e_1e_3e_5e_7e_9e_{11}e_{13}e_{15}.\] Define  
$F'_{r}=\langle y_1,\cdots,y_{r}\rangle$ for any $r\leq 5$.

\begin{lemma}\label{E7:AY}
For an elementary abelian 2-group $F\subset G$, if $F$ contains no Klein four subgroup 
conjugate to $\Gamma_1$, but contains an element conjugate to $\sigma_1$, then 
\[\rank H_{F}/A_{F}=1.\] 
\end{lemma}
\begin{proof}
Recall that, $H_{F}$ is a subgroup of $F$ generated by elemnts conjugate to $\sigma_1$. 
Let $Y_{F}=\{x\in H_{F}:x\sim\sigma_2\}\cup\{1\}$, we show that $A_{F}=Y_{F}$ in this 
case.   

Choose any $x_0\in F$ with $x_0\sim\sigma_1$. For any other $x\in F$ with $x\sim\sigma_1$, 
since $F$ contains no Klein four subgroup 
conjugate to $\Gamma_1$, so $xx_0\sim\sigma_2$. Then $x\in H_{x_0}$. By this, we know 
$H_{F}\subset H_{x_0}$, so $H_{x_{0}}=H_{F}$ as the containment relation in the converse 
direction is obvious. Similarly $H_{x}=H_{F}$ for any $x\in F$ with $x\sim\sigma_1$.  
Then for any distinct $y_1,y_2\in Y_{F}$ ($y_1\sim y_2\sim\sigma_2$), $y_1y_2\sim\sigma_2$. So 
$Y_{F}$ is a subgroup of $H_{F}$. 

Then it is clear that $Y_{F}=A_{F}$. So $\rank H_{F}/A_{F}=\rank H_{F}/Y_{F}=1$.  
\end{proof}

\begin{lemma}\label{E8 pure sigma2}
For an elementary abelian 2-group $F\subset G$, if $F$ contains no Klein four subgroup 
conjugate to $\Gamma_1$, then $F$ is conjugate to one of 
$\{F''_{r,s}:r\leq 3,s\leq 2\},\{F'_{r}:r\leq 5\}$.
\end{lemma}

\begin{proof}
When $F$ contains no Klein four subgroup conjugate to $\Gamma_1$, but contains an element  
conjugate to $\sigma_1$, we may assume that $\sigma_1\in F$. Then 
\[F\subset G^{\sigma_1}\cong(\E_7\times\Sp(1))/\langle(c,-1)\rangle.\]  
Modulo $\Sp(1)$, we get a homomorphism 
\[\pi: F\longrightarrow\E_7/\langle c\rangle\cong\Aut(\fre_7).\]

Since we assume $F$ contains no Klein four subgroup conjugate to $\Gamma_1$, any element in 
$F-\langle\sigma_1\rangle$ is conjugate to $\tau_1=(\eta_1,1)$, $\tau_2=(\eta_2,1)$ or 
$\tau_4=(\eta_4,\textbf{i})$ in $(\E_7\times\Sp(1))/\langle(c,-1)\rangle$, and any Klein 
for subgroup of $F\cap\E_7$ has at least elment conjugate to $\eta_2$. Then 
$F'=\pi(F)\subset\Aut(\fre_7)$ contains no elemnts conjugate to $\eta_3$, and no Klein 
four subgroups whose fixed point subalgebra is isomorphic to $\mathfrak{su}(6)\oplus (i\bbR)^2$. 
In the case of $\E_7$ (Section \ref{E7}), it corresponds to the elementary abelian 2-subgroup 
$F'$ with no elements conjugate to $\sigma_2$ and the map $m$ on $H_{F'}$ is trivial. By 
Section \ref{E7}, we get  $F\sim F''_{r,s}$ for some $(r,s)$ with $r\leq 3,s\leq 2$. 

When $F$ is pure $\sigma_2$, we may assume that $\sigma_2\in F$. Then 
$F\subset G^{\sigma_2}\cong \Spin(16)/\langle c\rangle$ and any element in 
$F-\langle\sigma_2\rangle$ is conjugate to $e_1e_2e_3e_4e_5e_6e_7e_8 $ in 
$\Spin(16)/\langle c\rangle$. Then $F\sim F'_{r}$ for some $r\leq 5$. 
\end{proof}

\subsection{Involution types and Automizer groups}

\begin{prop}\label{E8: uniqueness}
For an isomorphism $f:F\longrightarrow F'$ between two elementary abelian 2-subgroups of 
$G$, if $f(x)\sim x$ for any $x\in F$, then $f=\Ad(g)$ for some $g\in G$.
\end{prop}

\begin{proof}
When $F$ contains a Klein four subgroup conjugate to $\Gamma_1$, this reduces to 
the similar statement in $\Aut(\fre_6)$ case.

When $F$ doesn't contain any Klein four subgroup conjugate to $\Gamma_1$, this is 
already showed in the proof of Lemma \ref{E8 pure sigma2}. 
\end{proof}

\begin{definition}
For an elementary abelian 2-group $F\subset G$, we say $F$ is the orthogonal 
direct product of other subgroups $K_1,\cdots,K_{t}$ if there exists an isomorphism 
$f: K_1\times\cdots\times K_{t}\longrightarrow F$ with 
\[\mu(f(x_1,...,x_{t}))=\mu(x_1)\cdots\mu(x_{t})\] for any 
$(x_1,...,x_{t})\in K_1\times\cdots\times K_{t}$.
\end{definition}

Let $A=\langle\sigma_2\rangle$, $B_{s}$($s\leq 3$) be a rank $s$ pure $\sigma_1$ subgroup, 
$B=B_1$, $C=F_3$, $D$ be a rank 3 subgroup with only one element conjugate to $\sigma_1$. 
Then the involution types of some elementary abelian 2-subgroups of $\E_8$ have the 
following description \begin{eqnarray*}&& F_{r,s}=A^{r}\times B_{s}\times B_{3};\ 
F'_{r,s}=A^{r}\times B_{s}\times B_2;\\&& 
F'_{\epsilon,\delta,r,s}=A^{r}\times C^{s}\times B^{\epsilon}\times B_2^{1+\delta};
\\&& F''_{r,1}=A^{r}\times B,\ F''_{r,2}=A_{r}\times C,\\&& 
F''_{r,3}=A^{r}\times D;\ F'_{r}=A^{r}\end{eqnarray*} 
$F_{\epsilon,\delta,r,s}$ ($s\geq 1$) doesn't a similar decomposition since elements 
in $F_{\epsilon,\delta,r,s}-F'_{\epsilon,\delta,r,s}$ are all conjugate to $\sigma_2$. 
With the involution type available, we can describe $\Graph(F)$. The graphs of $F_{r,s}$ 
is a complete bipartite graph with $s,3$ vertices in two parts, that of $F'_{r,s}$ is a 
complete bipartite graph with $s,2$ in two parts, that of $F''_{r,s}$ ($s\geq 1$) is 
a single vertex graph, that of $F'_{r}$ is an empty graph, the graphs of 
$F_{\epsilon,\delta,r,s},F'_{\epsilon,\delta,r,s}$ are not of bipartite form 
and a little more complicated. 

In summary, we have the following statement 

{\it ``the conjugacy class of an elementary abelian 2-subgroup $F\subset G$ is 
determined by the datum $(\rank F, \rank A_{F}, \Graph(F))$''}.

\begin{prop}\label{E8:m}
For an elementary abelian 2-subgroup $F\subset \E_8$, $m$ is a bilinear form on 
$F$ if and only if $F$ is not conjugate to any of $\{F_{r,s}:r\leq 2,s\leq 3\}\cup
\{F_{\epsilon,\delta,r,s}:\epsilon+\delta\leq 1,r+s\leq 2\}\cup
\{F''_{r,3}:r\leq 2\}$. 
\end{prop}

\begin{proof}
When $F$ is conjugate to one of $\{F_{r,s}:r\leq 2,s\leq 3\}\cup
\{F_{\epsilon,\delta,r,s}:\epsilon+\delta\leq 1,r+s\leq 2\}$, it contains 
a subgroup conjugate to $B_3,F_{0,0,0,0}$ or $D$. $B_3,F_{0,0,0,0},D$  
contains $7,3,1$ elements with $\mu$ value -1 respectively, so $m$ is not bilinear 
on them by Proposition \ref{rank three symplectic space}.   

When $F$ is conjugate to a subgroup in the other four families, $m$ is 
bilinear on $F$ follows from the orthogonal decomposition of it.  
\end{proof}

We can write the decomposition of involution types for some subgroups  
in a simpler way, \begin{eqnarray*}&& F'_{r,0}=A^{r}\times B_2,
\\&& F'_{r,1}=A^{r}\times B\times B_2=A^{r}\times B\times C,
\\&& F'_{r,2}=A^{r}\times B_2\times B_2=A^{r}\times C\times C, 
\\&& F'_{1,0,r,s}=A^{r}\times C^{s}\times B\times B_2^{1}
=A^{r}\times B\times C^{s+1}, 
\\&&  F'_{0,\delta,r,s}=A^{r}\times C^{s}\times B_2^{1+\delta}
=A^{r}\times B_2^{1-\delta}\times C^{s+2\delta}. 
\end{eqnarray*}

\begin{prop}
\begin{itemize}
\item[(1)]{$r\leq 2$, $s\leq 2$, $W(F_{r,s})\cong\Hom(\bbF_2^{3+s},\bbF_2^{r})\rtimes
\big(\GL(r,F_2)\times (\GL(s,\bbF_2)\times\GL(3,\bbF_2))\big)$;}

\item[(2)]{$r\leq 2$, $W(F_{r,3})\cong\Hom(\bbF_2^{6},\bbF_2^{r})\rtimes
\big(\GL(r,F_2)\times ((\GL(3,\bbF_2)\times\GL(3,\bbF_2))\rtimes S_2)\big)$;}

\item[(3)]{$r\leq 2$, $s\leq 2$, $W(F'_{r,s})\cong\Sp(r,s;2s-s^{2},\frac{(s-1)(s-2)}{2})$;}

\item[(4)]{$\epsilon+\delta\leq 1$, $r+s\leq 2$, $W(F_{\epsilon,\delta,r,s})=\bbF_2^{r+2s+\epsilon+2\delta+2}
\rtimes\Sp(r,s+\epsilon+2\delta;\epsilon,(1-\epsilon)(1-\delta))$;}

\item[(5)]{$\epsilon+\delta\leq 1$, $r+s\leq 2$, $W(F'_{\epsilon,\delta,r,s})=\Sp(r,s+\epsilon+2\delta;\epsilon,
(1-\epsilon)(1-\delta))$.}

\item[(6)]{$r\leq 3$, $s\leq 2$, $W(F''_{r,s})\cong\Hom(\bbF_2^{s},\bbF_2^{r+1})\rtimes\big(
(\mathbb{F}_2^{r}\rtimes\GL(r,\bbF_2))\times\GL(s))\big)$;}

\item[(7)]{$r\leq 5$, $W(F'_{r})\cong \GL(r,\mathbb{F}_2)$.}
\end{itemize}
\end{prop}

\begin{proof}
By Proposition \ref{E8: uniqueness}, we need to calculate automorphisms of $F$ preserving   
the function $\mu$ on $F$. 
  
$W(F_{r,s})=\Hom(\bbF_2^{3+s},\bbF_2^{r})\rtimes W(F_{0,s})$ and $W(F_{0,s})$ 
stabilizes $B_{s}\cup B_3\subset F_{0,s}$. By this we get $(1)$ and $(2)$.

When $m$ is bilinear, $(F,m,\mu)$ is a symplectic metric space, then we can  
identify $W(F)$ with the automorphism group of $(F,m,\mu)$. By this we get (3) and (5). 

$(4)$ follows from $(5)$ immediately. 

For (6), we have $A_{F}\subset H_{F}\subset F$ and $A_{F},H_{F}$ are preserved by $W(F)$.  
By Lemma \ref{E7:AY}, we have $\rank A_{F}=r$, $\rank H_{F}/A_{F}=1$, $\rank F/H_{F}=1$, 
then we get the conclusion. 

(7) is clear.
\end{proof}

We have an inclusion $i:\E_7\subset\E_8$ since $\E_8^{\sigma_1}\cong(\E_7\times\Sp(1))/
\langle(c,-1)\rangle$. Let $p:\E_7\longrightarrow\Aut(\mathfrak{e}_7)$ be the adjoint 
map (a 2-fold covering). For a pure $\sigma_1$ (that for $\E_7$ case) elementary abelian 
2-group $F\subset\Aut(\mathfrak{e}_7)$, $i(p^{-1}F)$ is an elementary abelian 2-group of 
$\E_8$.

\begin{prop}
An elementary abelin 2-group $F\subset\E_8$ is conjugate to $i(p^{-1}F')$ for some 
pure $\sigma_1$ group $F'\subset \Aut(\mathfrak{e}_7)$ if and only if $F$ contains 
an elementary $x$ with $x\sim\sigma_1$ and $H_{x}=F$. 
\end{prop}

\begin{proof}
It follows from the description of the conjugacy classes of involutions in 
$\E_8^{\sigma_1}\cong(\E_7\times\Sp(1))/\langle(c,-1)\rangle$.
\end{proof}

Any subgroup of $\E_8$ satisfying this condition is conjugate to one of   
\[\{F_{r,1}=A^{r}: r\leq 2\}, \{F'_{r,1}, r\leq 2\}, 
\{F'_{1,0,r,s}: r+s\leq 3, s\geq 1\}, \{F''_{r,1}: r\leq 3\},\] there are 13 classes in total.
Since there are also 13 classes of pure $\sigma_1$ subgroups in 
$\Aut(\mathfrak{e}_7)$, so for any two elementary abelian 2-groups  
$F,F'\subset\Aut(\mathfrak{e}_7)$, \[i(p^{-1}F)\sim_{\Aut(\mathfrak{e}_8)} 
i(p^{-1}F')\Leftrightarrow F\sim_{\Aut(\mathfrak{e}_7)} F'.\]


\begin{thebibliography}{99}

\bibitem[AGMV]{Andersen-Grodal-Moller-Viruel} Andersen, K. K. S.; Grodal, J.; 
Moller, J. M.; Viruel, A., \emph{The classification of $p$-compact groups for $p$ odd}. 
Ann. Math. (2) \textbf{167}  (2008),  no. 1, 95--210.
\bibitem[B]{Berger} M. Berger,
\emph{Les espaces sym\'etriques noncompacts}, Ann. Sci. \'Ecole Norm.
Sup. (3)  \textbf{74} (1957) 85--177.
\bibitem[Bo]{Borel} A.~Borel, \emph{Sous-groupes commutatifs et torsion des groupes de 
Lie compacts connexes}. (French) Tôhoku Math. J. (2) \textbf{13} 1961 216–240. 
\bibitem[Gr]{Griess} Griess, Robert L., Jr., \emph{Elementary abelian $p$-subgroups of 
algebraic groups}.  Geom. Dedicata \textbf{39}  (1991),  no. 3, 253--305.
\bibitem[He]{Helgason} Helgason, S. \emph{Differential geometry, Lie groups, and 
symmetric spaces}. Pure and Applied Mathematics, 80. Academic Press, Inc. 
[Harcourt Brace Jovanovich, Publishers], New York-London, 1978. xv+628 pp. 
ISBN: 0-12-338460-5.
\bibitem[HY]{Huang-Yu} Huang, J.-S.; Yu, J., \emph{Klein four subgroups of Lie algebra 
automorphisms.} ArXiv:1108.2397v1. 
\bibitem[Kn]{Knapp} Knapp, A. W., \emph{Lie groups beyond an introduction}. 
Progress in Mathematics, \textbf{140}. Birkhuser Boston, Inc., Boston, MA, 1996.
\bibitem[Se]{Serre} Serre, J.-P., \emph{1998 Moursund Lectures at the University of Oregon}. 
ArXiv:math/0305257.
\bibitem[Yu]{Yu}Yu, J., \emph{A note on closed subgroups of compact Lie groups.} 
ArXiv:0912.4497v2. 

\end{thebibliography}
\end{document}